\documentclass[11pt]{amsart}
\usepackage{amssymb,latexsym,amsmath,amscd,amsthm,amsfonts, enumerate}
\usepackage{multirow}
\usepackage{color}
\usepackage[all]{xy}
\usepackage{etex}
\usepackage{caption}
\usepackage{soul}
\usepackage{float}
\usepackage{titletoc}
\usepackage[normalem]{ulem}
\usepackage{graphicx}

\usepackage{tikz,tikz-cd}
\usepackage{mathrsfs}
\usepackage[all]{xy}
\usepackage{tikz}
\usepackage{extarrows}
\usepackage{tikz-cd}
\usetikzlibrary{calc}
\usetikzlibrary{matrix,arrows,decorations.pathmorphing}
\usetikzlibrary{snakes}
\usetikzlibrary{shapes.geometric,positioning}
\usetikzlibrary{arrows,decorations.pathmorphing,decorations.pathreplacing}
\usetikzlibrary{positioning,shapes,shadows,arrows,snakes}

\usepackage[colorlinks=true,hyperindex]{hyperref}
\newcommand\myshade{85}
\colorlet{mylinkcolor}{violet}
\colorlet{mycitecolor}{red}
\colorlet{myurlcolor}{cyan}

\hypersetup{
  linkcolor  = mylinkcolor!\myshade!black,
  citecolor  = mycitecolor!\myshade!black,
  urlcolor   = myurlcolor!\myshade!black,
  colorlinks = true,
}

\usepackage{tabularx}
\newcolumntype{E}{>{\hsize=0.5cm \centering\arraybackslash}X}%
\newcolumntype{C}[1]{>{\hsize=#1\hsize \centering\arraybackslash}X}%

\numberwithin{equation}{section}
\newtheorem{theorem}{Theorem}[section]
\newtheorem{proposition}[theorem]{Proposition}
\newtheorem{proposition-definition}[theorem]{Proposition-Definition}

\newtheorem{corollary}[theorem]{Corollary}
\newtheorem{lemma}[theorem]{Lemma}

\theoremstyle{definition}
\newtheorem{remark}[theorem]{Remark}
\newtheorem{example}[theorem]{Example}
\newtheorem{definition}[theorem]{Definition}

\newcommand{\End}{\operatorname{End}\nolimits}

\newcommand{\Hom}{\mathrm{Hom}}

\newcommand{\mr}{\mathrm{Mrig(A)}}
\newcommand{\gd}{\mathrm{gldim}}

\newcommand{\Ext}{\mathrm{Ext}}
\newcommand{\add}{\mathsf{add}}

\newcommand{\za}{\alpha}
\newcommand{\zb}{\beta}

\newcommand{\zD}{\Delta}

\newcommand{\zg}{\gamma}
\newcommand{\zG}{\Gamma}

\newcommand{\aaa}{{\bf{a}}}
\newcommand{\bbb}{{\bf{b}}}
\newcommand{\ccc}{{\bf{c}}}

\newcommand{\pd}{\operatorname{{\rm pd }}}

\newcommand{\ma}{\operatorname{{\rm mod-A }}}
\newcommand{\urm}{\operatorname{{\rm \underline{mod}-A }}}
\newcommand{\urmo}{\operatorname{{\rm \underline{mod}-A^{op} }}}
\newcommand{\tr}{\operatorname{{\rm Tr }}}%

\newcommand{\cals}{\mathcal{S}}

\newcommand{\calc}{\mathcal{C}}

\newcommand{\call}{\mathcal{L}}

\newcommand{\calm}{\mathcal{M}}

\newcommand{\calp}{\mathcal{P}}

\newcommand{\bbp}{\mathbb{P}}

\definecolor{dark-green}{RGB}{14,150,2}
\definecolor{red}{RGB}{250,0,0}

\newcommand{\bpoint}{\circ}
\newcommand{\rpoint}{\color{red}{\bullet}}

\begin{document}

\title[Maximal rigid modules and applications to higher Auslander-Reiten theory]{Maximal rigid modules over a gentle algebra and applications to higher Auslander-Reiten theory}


\author{Wen Chang}
\address{School of Mathematics and Statistics, Shaanxi Normal University, Xi'an 710062, China}
\thanks{The author is supported by the Fundamental Research Funds
	for the Central Universities (No. GK202403003) and the NSF of China (Grant No. 12271321).}
\email{changwen161@163.com}


\thanks{}

%


\begin{abstract}
We construct a bijective correspondence between the set of rigid modules over a gentle algebra and the set of admissible arc systems on the associated coordinated-marked surface. In particular, a maximal rigid module aligns with an equivalence class of admissible $5$-partial triangulations, which is an (admissible) set of simple arcs dissecting the surface into $s$-gons with $3\leqslant s\leqslant 5$. Furthermore, the rank of the maximal rigid module is equal to the rank of the algebra plus the number of internal $4$-gons and $5$-gons in the associated $5$-partial triangulation.

Subsequently, these results facilitate an exploration of the higher Auslander-Reiten theory for gentle algebras with global dimension $n$. The $\tau_m$-closures of injective modules are realized as admissible $(m+2)$-partial triangulations, where $\tau_m$ are higher Auslander-Reiten translations with $2\leqslant m \leqslant n$. Finally, we provide a complete classification of gentle algebras that are $\tau_n$-finite or $n$-complete introduced by Iyama \cite{I11}.
\end{abstract}

\maketitle
\setcounter{tocdepth}{2} 

\tableofcontents

\section*{Introduction}\label{Introductions}
In the representation of algebras, rigid modules play an important role, that is, the modules with trivial (once) self-extension. Notably, tilting modules, which give rise to derived equivalences of algebras, are special rigid modules. The theme of this paper is to classify the rigid modules, especially the maximal rigid modules, on gentle algebras, and subsequently employ these findings in the domain of higher Auslander-Reiten theory.

Gentle algebras emerged in the 1980s as a generalization of iterated tilted algebras of type $A_n$ \cite{AH81}, and affine type $\widetilde A_n$ \cite{AS87}. They are a classical and enduring object in the representation theory of algebras. In particular, rigid modules over gentle algebras are studied in \cite{S99}. It has been demonstrated that the endomorphism algebra of a rigid module over a gentle algebra remains a gentle algebra.

In recent years, topological and geometric methods have been widely applied in studying gentle algebras.
The geometric models for the module categories and the derived categories over gentle algebras are established using the surface models, respectively, in \cite{BC21} and \cite{OPS18} (see \cite{ABCJP10} for the special case of gentle algebras arising from surface triangulations). More generally, the derived categories of graded gentle algebras can be realized as partially wrapped Fukaya categories of graded marked surfaces, see \cite{HKK17,LP20}.
In \cite{C25}, the author refines the geometric model for the module category of a gentle algebra in \cite{BC21}, and then integrates it with the geometric model of the derived category in \cite{OPS18}, in the sense that each so-called zigzag curve on the surface represents
an indecomposable module as well as the minimal projective resolution of this module.
In particular, extensions of modules can thus be easily visualized as \emph{weighted-oriented intersections} on the surface, a feature that is important in studying rigid modules over gentle algebras in this paper.

More precisely, we consider a \emph{coordinated-marked surface} $(\cals,\calm,\zD^*)$, that is, a marked surface with a lamination considered in \cite{APS23,OPS18}, where $\cals$ is an oriented surface with boundaries, $\calm$ is a set of marked $\bpoint/\rpoint$-points on $\cals$, and $\zD^*$ is a set of simple arcs with $\rpoint$-endpoints which cut $\cals$ into polygons, each of which contains exactly one $\bpoint$-point. Then a gentle algebra $A=A(\zD^*)$ can be constructed, and conversely any gentle algebra is of this form \cite{BC21,OPS18,PPP19}. Furthermore, indecomposable (string) modules over $A$ correspond to a special kind of $\bpoint$-arcs called \emph{zigzag arcs}, that is, an arc with $\bpoint$-points as endpoints and factor through polygons of $\zD^*$ along the corners.

A collection of simple zigzag arcs on $(\cals,\calm, \zD^*)$ is called a \emph{s-partial triangulation}, if the arcs cut the surface into polygons, each of which contains at most one $\rpoint$-point from $\calm_{\rpoint}$, and the number of the edges of each polygon is bounded by $s$, see Definition \ref{def:par-tri}.
Furthermore, if the weight of any oriented intersection of arcs in a $s$-partial triangulation differs from one, then it is an \emph{admissible s-partial triangulation}, see Definition \ref{def:par-tri2}.

The first main result of this paper is as follows.

\begin{theorem}(Theorem \ref{thm:max-m2})
There is a bijection between the set of equivalent classes of admissible $5$-partial triangulations on $(\cals,\calm,\zD^*)$ and the set of the maximal rigid modules in $\ma$.
\end{theorem}

An admissible $5$-partial triangulation on $(\cals,\calm,\zD^*)$ is formed by polygons of the types $\textrm{\textbf{Fi}}, 1 \leqslant \textrm{\textbf{i}} \leqslant 5,$ in Figure \ref{fig:ply}.
We call two $5$-partial triangulations equivalent if the only possible different polygons of them are external $5$-gons, that is, polygons of type $\textbf{F3}$ in Figure \ref{fig:ply}. Then the bijection in the above theorem maps the maximal element in an equivalent class of admissible $5$-partial triangulations to a maximal rigid module which is a direct sum of the modules associated with the arcs in it. 

Noticing that no internal triangles, that is, the triangles formed by $\bpoint$-arcs, in an admissible arc system of simple zigzag arcs, see Lemma \ref{lem:ply}. On the other hand, adding one more arc to an internal $4$-gon or $5$-gon will produce an internal triangle. Therefore, the $5$-partial triangulation is a candidate to correspond to a rigid module with \emph{maximal property}.
It is well-known that any oriented (marked) surface has triangulations. In particular, a set of simple arcs can always be completed as a triangulation. However, in our setting, we have two more constraints when we consider admissible partial triangulations. 

(1) the arcs must be zigzag; 

(2) the weights of the oriented intersections must be different from one. 

The first constraint makes sure that the arc gives rise to a module over the gentle algebra, while the second one ensures that the direct sum of these modules is rigid. 
Both of these conditions are applied with respect to the predetermined ``coordinate" $\zD^*$. Therefore, an interesting phenomenon revealed by the above theorem is that although we have these constraints, we can always complete an admissible set of simple zigzag arcs to an admissible $5$-partial triangulation, disregarding the choice of $\zD^*$. 

We mention that a similar correspondence is established for \emph{maximal almost rigid modules} in \cite[Theorem 1.2]{BCGS24}, where a special kind of once-self-extension is allowed for an almost rigid module, and on the associated surface an internal triangle is allowed.

The second main result of this paper is as follows.

\begin{theorem}(Theorem \ref{thm:max-m3})
Let $A$ be a gentle algebra with rank $n$, and let $M$ be the maximal rigid module in $\ma$ associated with a $5$-partial triangulation $\calp_5$. Then the rank of $M$ equals 
$$e_1=n+f_4+f_5,$$
where $f_4$ and $f_5$ are, respectively, the numbers of the internal $4$-gons and $5$-gons in $\calp_5$.
\end{theorem}

In particular, the rank of any maximal rigid module in $\ma$ is not less than that of the algebra.
Furthermore, any maximal rigid module in $\ma$ has rank $n$ if and only if $A$ is a hereditary algebra of type $A$ or of affine type $A$, see Corollary \ref{thm:ply}. On the other hand, if we consider the self-orthogonal modules, that is, the modules without self-extensions of all degrees (not just for degree one), then the rank may be strictly less than the rank of the algebra, see, for example, \cite[Theorem 2.23]{C24}.

The main idea used in the proofs of the first part is to reduce the problem to those on the disks by using a cutting surface technique introduced in \cite{C24}.

The second part of this paper applies the above results to the higher Auslander-Reiten theory of gentle algebras. 
As a higher analogue of the classical Auslander-Reiten theory introduced in the 1970s, the higher Auslander-Reiten theory was introduced by Iyama in 2004
\cite{I07a,I07b}. In addition to representation theory
\cite{HI11,IO11,M14}, it has exhibited connections to
commutative algebra, commutative and non-commutative algebraic geometry, and
combinatorics, see for example
\cite{AIR15,HIMO23,HIO14,HIO14,IW14,OT12}. 

In the higher Auslander-Reiten theory, the higher Auslander-Reiten translation $\tau_n$ is a fundamental concept that induces an equivalence of the stable categories of $n$-cluster tilting subcategories in $\ma$, which is a central research object in the higher Auslender-Reiten theory. A $n$-cluster tilting subcategory contains the $\tau_n$-closure of the injective modules. In particular, if the global dimension of $A$ does not exceed $n$ and $\ma$ has a $n$-cluster tilting module, then the $\tau_n$-orbit of injective modules gives rise to the unique $n$-cluster tilting subcategory of $\ma$. 

In other words, there are not many $n$-cluster tilting modules in $\ma$. For example, it is proved in \cite{HJS22} that for a gentle algebra $A$, if $\ma$ contains a $n$-cluster
tilting subcategory for some $n \geqslant 2$, then $A$ is a Nakayama algebra with zero radical square. In \cite{I11} a relative version of $n$-cluster tilting modules is introduced. More precisely, the $\tau_n$-orbit of injective modules may be a (relative) $n$-cluster tilting module in the right-perpendicular category of a tilting module, rather than in the whole module category. When the tilting module is the regular module $A$, we go back to the original $n$-cluster tilting settings.

Furthermore, Iyama introduced so-called \emph{$n$-complete algebras}, see Definition \ref{complete}, such that the existence of (relative) $n$-cluster tilting modules is guaranteed iteratively. Note that $n$-complete algebra is a special kind of the so-called \emph{$\tau_n$-finite algebra}, which is an algebra with global dimension bounded by $n$ and $\tau_n^\ell(DA)=0$ for sufficiently large $\ell$ and the injective module $DA$. 

We will give a classification of gentle algebras that are $\tau_n$-finite or $n$-complete.
In the first place, we study the $\tau_m$-orbit of injectives by using surface models. It is proven that for $2\leqslant m\leqslant n$, the $\tau_m$-orbit of the injective modules in $\ma$ gives rise to an admissible $(m+2)$-partial triangulation. Therefore, the direct sum of the modules in the orbit is a rigid module. In particular, the $\tau_2$-orbit of the injective modules gives rise to an admissible $4$-partial triangulation, for which the associated module is a maximal rigid module, see Proposition \ref{prop:geoinjcl}.

Then we introduce \emph{$\tau_n$-sequences} and \emph{$\tau_n$-cycles} in Definition \ref{def:tseq}, which are given by combinatorial of walks in the algebra.
The following theorem gives a classification of $\tau_n$-finite gentle algebras and $n$-complete gentle algebras.

\begin{theorem}(Theorem \ref{thm:taufinite}, Theorem \ref{thm:ncomp})
Let $A$ be a gentle algebra with the global dimension equal to $n$.

(1) $A$ is $\tau_n$-finite if and only if there is no $\tau_n$-cycle in $A$.

(2) $A$ is $n$-complete if and only if, for any walk $\sigma=\aaa_1\cdots\aaa_{n}$ in $A$ with full relations, the degree of the source is one, or equivalently, $\aaa_1$ is the unique arrow adjacent to it.
\end{theorem}

\section{Preliminaries}
\label{Preliminaries}

In this paper, a quiver will be denoted by $Q=(Q_0, Q_1)$, where $Q_0$ is the set of vertices and $Q_1$ is the set of arrows. The numbers of the vertices and the arrows of $Q$ are $|Q_0|$ and $|Q_1|$, respectively. For an arrow $\aaa$, $s(\aaa)$ is the source, and $t(\aaa)$ is the target of it.
The arrows in a quiver are composed from left to right as follows: for the arrows $\aaa$ and $\bbb$ we write $\aaa\bbb$ for the path from the source of $\aaa$ to the target of $\bbb$.

An algebra $A$ will be assumed to be basic with finite dimension over an algebraically closed field $k$. In general, we consider the right modules, where $\ma$ is the category of finite-dimensional modules over $A$.
For a module $M$, we denote by $\add M$ the subcategory of
$\ma$ consisting of direct summands of finite direct sums of
copies of $M$. For example, $\add A$ is the category
of finitely generated projective $A$-modules, and
$\add DA$ is the category
of finitely generated injective $A$-modules.

\subsection{Marked surfaces}\label{subsection: geo-module categories}

We recall some concepts about marked surfaces, and the construction of gentle algebras from coordinates of the surfaces, for which there are many references such as \cite{BC21,HKK17,LP20,PPP19,PPP21}, in this paper, we closely follow \cite{OPS18} and \cite{APS23}.

\begin{definition}
	\label{definition:marked surface}
	A \emph{marked surface} is a pair $(\cals,\calm)$, where
	\begin{enumerate}[\rm(1)]
		\item $\cals$ is an oriented surface with non-empty boundaries with
		connected components $\partial \cals=\sqcup_{i=1}^{b}\partial_i \cals$;
\item $\calm = \calm_{\bpoint} \cup \calm_{\rpoint} \cup \calp_{\bpoint}$ is a finite set of \emph{marked points} on $\cals$. The elements of~$\calm_{\bpoint}$ and~$\calm_{\rpoint}$ are on the boundary of $\cals$, which will be respectively represented by symbols~$\bpoint$ and~$\rpoint$. Each connected component $\partial_i \cals$ is required to contain at least one marked point of each color, where the points~$\bpoint$ and~$\rpoint$ are alternating on $\partial_i \cals$. The elements in $\calp_{\bpoint}$ are in the interior of $\cals$. We refer to these points as \emph{punctures}, and we will also represent them by the symbol $\bpoint$.
	\end{enumerate}
\end{definition}
Let $(\cals,\calm)$ be a marked surface.
	\begin{enumerate}[\rm(1)]
		\item An \emph{arc} is a non-contractible curve, with endpoints in~$\calm_{\bpoint}\cup\calm_{\rpoint}$. It is an $\rpoint$-\emph{arc} if the endpoints are from $\calm_{\rpoint}$, and it is an $\bpoint$-\emph{arc} if the endpoints are from $\calm_{\bpoint}\cup\calp_\bpoint$.
		\item A \emph{loop} is an arc whose endpoints coincide.

		\item A \emph{simple arc} is an arc without interior self-intersections.
	\end{enumerate}

In order for some definitions and notations to be well-defined in the case of a loop, we will treat the unique endpoint of a loop as two distinct endpoints.
On the surface, all curves are considered up to homotopy with respect to the boundary components and the punctures, and all intersections of curves are required to be transversal.

To realize the gentle algebras and associated module categories on marked surfaces, we need the following so-called coordinates of the marked surfaces.

\begin{definition}\label{definition:addmissable dissections}
\begin{enumerate}[\rm(1)]
\item A collection of simple $\rpoint$-arcs is called a \emph{simple coordinate}, if the arcs have no interior intersections and they cut the surface into polygons each of which contains exactly one $\bpoint$-point  from $\calm_{\bpoint}\cup\calp_{\bpoint}$. 
\item  A collection of simple $\bpoint$-arcs is called an \emph{dissection}, if the arcs have no interior intersections and cut the surface into polygons, each of which contains exactly one $\rpoint$-point from $\calm_{\rpoint}$.
\end{enumerate}
\end{definition}

We denote a simple coordinate by 
$\zD^*$, and call the triple $(\cals,\calm,\zD^*)$ a \emph{coordinated-marked surface}. Then there is a unique dissection $\zD$ that is the \emph{dual} of $\zD^*$, that is, each $\rpoint$-arc in $\zD^*$ intersects a unique $\bpoint$-arc in $\zD$ only once. We call $\zD$ the \emph{simple dissection} (with regard to $\zD^*$).

Let $\ell^*$ be an $\rpoint$-arc in $\zD^*$. We call an $\bpoint$-arc $t^{-1}(\ell^*)$ obtained from $\ell^*$ by anticlockwise rotating both endpoints to the next $\bpoint$-points on the same boundary components the \emph{anti-twist} of $\ell^*$, where anticlockwise means the interior of the surface is on the left when following a boundary component.
Then it is easy to see that the set $\{t^{-1}(\ell^*), \ell^*\in \zD^*\}$ is a dissection of the surface, which we
denote by $\zD_I$ and call \emph{injective dissection}. We dually define \emph{twist} $t(\ell^*)$ and call $\zD_P:=\{t(\ell^*), \ell^*\in \zD^*\}$ \emph{projective dissection}.

We mention that there are different names for the coordinates and dissections introduced above, for example, full formal arc systems \cite{HKK17}, admissible dissections \cite{APS23}, laminations \cite{OPS18}, and partial triangulations \cite{BC21}. 
In this paper, the words ``admissible" and ``partial triangulation" are introduced below, which have different meanings.
The names of simple dissection, injective dissection, and projective dissection will be verified in Proposition \ref{theorem:main arcs and objects}, that is, each arc in the dissection gives rise to a simple module, an injective module, and a projective module, respectively, in the module category of the associated gentle algebra.

\subsection{Gentle algebras from simple coordinates}\label{subsection: gentle algebras}

We recall how to construct a gentle algebra from a coordinated-marked surface $(\cals,\calm,\zD^*)$, where $\zD^*=\{\ell_i^*, 1\leqslant i \leqslant n\}$.

\begin{definition}\label{definition:gentle algebras}
We call an algebra $A=kQ/I$ a \emph{gentle algebra}, if $Q=(Q_0,Q_1)$ is a finite quiver and $I$ is an admissible ideal of $kQ$ satisfying the following conditions:
\begin{enumerate}[\rm(1)]
 \item Each vertex in $Q_0$ is the source of at most two arrows and the target of at most two arrows.

 \item For each arrow $\aaa$ in $Q_1$, there is at most one arrow $\bbb$ such that  ${\bf{0}} \neq \aaa\bbb\in I$; at most one arrow $\ccc$ such that  ${\bf{0}} \neq \ccc\aaa\in I$; at most one arrow $\bbb'$ such that $\aaa\bbb'\notin I$; at most one arrow $\ccc'$ such that $\ccc'\aaa\notin I$.

 \item $I$ is generated by paths of length two.
\end{enumerate}
\end{definition}

\begin{definition}\label{definition:oriented intersection}
Let $q$ be a common endpoint of arcs $\ell_i^*, \ell_j^*$ in $\zD^*$. An \emph{oriented intersection} from $\ell_i^*$ to $\ell_j^*$ at $q$ is an anticlockwise angle locally from $\ell_i^*$ to $\ell_j^*$ based at $q$ such that the angle is in the interior of the surface. An oriented intersection is \emph{minimal} if it is not a composition of two oriented intersections of arcs from $\zD^*$.
\end{definition}

\begin{definition}\label{definition:gentle algebra from dissection}
	We define the algebra $A(\zD^*)$ as the quotient of the path algebra $kQ(\zD^*)$ of the quiver $Q(\zD^*)$ by the ideal $I(\zD^*)$ defined as follows:
	\begin{enumerate}[\rm(1)]
		\item The vertices of $Q(\zD^*)$ are given by the arcs in $\zD^*$.
		\item Each minimal oriented intersection $\aaa$ from $\ell_i^*$ to $\ell_j^*$ gives rise to an arrow from $\ell_i^*$ to $\ell_j^*$, which is still denoted by $\aaa$.
		\item The ideal $I(\zD^*)$ is generated by paths $\aaa\bbb:\ell_i^*\rightarrow \ell_j^*\rightarrow \ell_k^*$, where the common endpoint of $\ell_i^*$ and $\ell_j^*$, and the common endpoint of $\ell_j^*$ and $\ell_k^*$ that respectively gives rise to $\aaa$ and $\bbb$ are different.
	\end{enumerate}
\end{definition}

Then it is not hard to check that $A(\zD^*)$ is a gentle algebra. 
Conversely, it is proved in \cite{BC21,OPS18,PPP19} that any gentle algebra arises from this way. So this establishes a bijection between the set of homeomorphism classes of coordinated-marked surfaces and the set of isomorphism classes of gentle algebras.

\begin{example}\label{ex:gen-alg}
See Figure \ref{fig:alg-ex} for an example of a coordinated-marked surface and the associated gentle algebra.
 \begin{figure}[ht]
		\begin{center}
\begin{tikzpicture}[>=stealth,scale=0.6]
\draw[line width=1.8pt,fill=white] (0,0) circle (4cm);
				\draw[thick,fill=gray!50] (0,0) circle (0.8cm);
				\path (0:4) coordinate (b1)
				(70:4) coordinate (b2)
				(110:4) coordinate (b3)
				(135:4) coordinate (b4)
				(200:4) coordinate (b5)
				(90:4) coordinate (b6)
				(-90:4) coordinate (b7)
				(-70:4) coordinate (b8)
				(-50:4) coordinate (b9)
				(45:4) coordinate (b10);
				
\draw[red,line width=1.3pt]plot [smooth,tension=0.6] coordinates {(0,-4) (-1.7,0) (-0.8,1.2) (0,0.8)};
\draw[red,line width=1.3pt]plot [smooth,tension=0.6] coordinates {(0,-4) (1.7,0) (.8,1.2) (0,0.8)};	
\draw [red, line width=1.5pt] (b6) to (0,0.8);
\draw[cyan,line width=1pt,red] (b6) to[out=-120,in=-10](b4);
\draw[cyan,line width=1pt,red] (b6) to[out=-60,in=-170](b10);
\draw[cyan,line width=1pt,red] (b7) to[out=50,in=170](b9);
\draw[red] (-1.5,2.5) node {$1$};
\draw[red] (.2,2.5) node {$2$};
\draw[red] (1.5,2.5) node {$3$};
\draw[red] (1.5,-1.8) node {$4$};
\draw[red] (-1.5,-1.8) node {$5$};
\draw[red] (1.5,-2.7) node {$6$};
\draw (6,0) node {};

\draw[thick,fill=white] (b1) circle (0.15cm)
				(b2) circle (.15cm)
				(b3) circle (.15cm)
				(b5) circle (.15cm)
				(b8) circle (.15cm);

\draw[thick,red, fill=red] (b4) circle (0.15cm)
				(b6) circle (.15cm)
				(b7) circle (.15cm)
				(b9) circle (.15cm)
				(b10) circle (.15cm);
				
\draw[thick, red ,fill=red] (0,0.8) circle (0.15cm);
\draw[thick, black ,fill=white] (0,-0.8) circle (0.15cm);	
\end{tikzpicture}
{\begin{tikzpicture}[ar/.style={->,very thick,>=stealth}]
				\draw(-5,3)node(6){$6$}
				(-4,1.5)node(4){$4$}
				(-2,1.5)node(5){$5$}
				(-3,0)node(2){$2$}
				(-4,-1.5)node(3){$3$}
				(-2,-1.5)node(1){$1$};
			    \draw[ar](6) to  (4) ;
				\draw[ar](4) to (5);
				\draw[ar](4) to  (2);
				\draw[ar](2)to  (5);
				\draw[ar](2) to  (3);
				\draw[ar](1) to  (2);
				\draw[ar](2)to (3);
	
   \draw[bend right,dotted,thick]($(6)!.6!(4)-(0,0.1)$) to ($(4)!.25!(2)-(0,0.1)$);
				\draw[bend right,dotted,thick]($(4)!.55!(2)-(0,0.1)$) to ($(2)!.3!(3)-(0,0.05)$);
				\draw[bend left,dotted,thick]($(5)!.55!(2)-(0,0.1)$) to ($(2)!.3!(1)-(0,0.05)$);
				
		\end{tikzpicture}}
		\end{center}	
  \caption{The left picture is a coordinated-marked surface $(\cals,\calm,\zD^*)$, where the arcs in $\zD^*$ are the $\rpoint$-arcs. The right picture shows the associated gentle algebra, where the dotted lines represent the (quadratic) relations in the algebra.} 
	\label{fig:alg-ex}
	\end{figure}
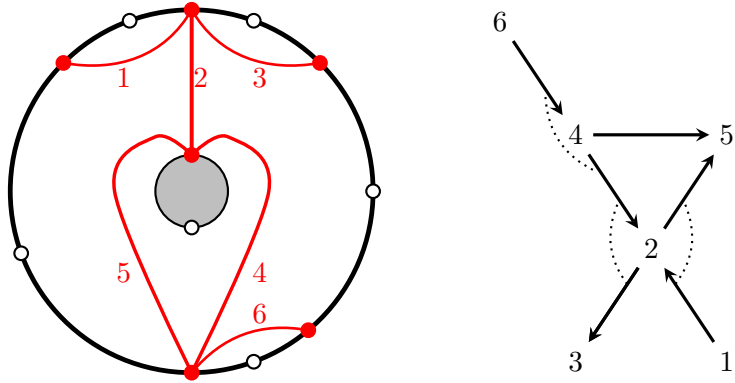
\end{example}

\subsection{The modules over gentle algebras}\label{subsection: gentle algebras}

In this subsection, we recall some basic definitions and constructions of the modules over gentle algebras.

It is well-known that any indecomposable module in $\ma$ is either a \emph{string module} or a \emph{band module}, which is parameterized by string and band combinatorics respectively \cite{BR87}.
The maps between the indecomposable modules are characterized in \cite{C89}.
Since the band module has self-intersection and plays no role in our consideration for rigid modules, we only recall the construction of string modules in the following. 
We refer the reader to \cite{BR87} for more details. 

For an arrow $\aaa$, let $\overline{\aaa}$ be its \emph{formal inverse} with $s(\overline{\aaa})=t(\aaa)$ and $t(\overline{\aaa})=s(\aaa)$. A \emph{walk} is a sequence $\sigma=\sigma_1\sigma_2\cdots\sigma_m$ of arrows and inverse arrows in $Q$ such that $t(\sigma_i)=s(\sigma_{i+1})$ and $\sigma_{i+1}\neq\overline{\sigma}_i$ for each $i$. Denote by $\overline{\sigma}=\overline{\sigma}_m\cdots\overline{\sigma}_2\overline{\sigma}_1$ the \emph{inverse} of $\sigma$.
A \emph{string} is a walk $\sigma$ that avoids relations, that is, there is no subsequence of $\sigma$ or $\overline{\sigma}$ belonging to $I$.
A \emph{direct string} is a string consisting of arrows and an \emph{inverse string} is a string consisting of formal inverses.
For each vertex $v$, we associate \emph{trivial string} $1_v$ with it.

Each string $\sigma$ defines a \emph{string module} $M_\sigma$, which is given by the representation of the quiver of type $A$ obtained by replacing every vertex in $\sigma$ with a copy of $k$ and every arrow by the identity map.
This gives a bijection between the inversion equivalent classes of strings and the isomorphism classes of string modules. In particular, $M_{1_v}$ is the simple module that arises from the vertex $v$ of $Q$.

\subsection{Zigzag arcs and string modules}\label{subsection: sur-module}
Now we recall the geometric model of the module category of a gentle algebra, see \cite{BC21,C25} for details.
Let's start with constructing a string module from an $\bpoint$-arc. 
In the subsection, let $(\cals,\calm,\zD^*)$ be a coordinated-marked surface, and let $A$ be the associated gentle algebra.

Note that $\zD^*$ cut the surface into polygons $\bbp$ each of which has exactly one marked point from $\calm_{\bpoint}$ or from $\calp_{\bpoint}$. These polygons will be called the \emph{polygons of $\zD^*$}.
We denote a polygon $\bbp$ of $\zD^*$ by $(\ell^*_{i_1},\cdots,\ell^*_{i_m})$, the ordered set of arcs in $\zD^*$ which form $\bbp$, where the arcs are ordered clockwise and where the index $1\leqslant j \leqslant m$ is considered modulo $m$ if $\bbp$ contains a puncture. For any $\ell^*_{i_j}\in \bbp$, we call $\ell^*_{i_{j-1}}$ (if exists) the \emph{predecessor} of $\ell^*_{i_{j}}$ in $\bbp$ and call $\ell^*_{i_{j+1}}$ (if exists) the \emph{successor} of $\ell^*_{i_{j}}$ in $\bbp$, see Figure \ref{figure:def-zigzag}.
In particular, $\ell^*_{i_{1}}$ has no predecessor and $\ell^*_{i_{m}}$ has no successor if $\bbp$ has a marked point from $\calm_{\bpoint}$, while $\ell^*_{i_{1}}$ has the predecessor $\ell^*_{i_{m}}$ and $\ell^*_{i_{m}}$ has the successor $\ell^*_{i_{1}}$ if $\bbp$ has a puncture.

{\bf Setting}: {\it Let $\za$ be an $\bpoint$-arc on $(\cals,\calm)$. After choosing a direction of $\za$, denote by $\bbp_0, \bbp_1,\cdots, \bbp_{n+1}$ the ordered polygons of $\zD^*$ that successively intersect with $\za$. Denote by $\ell^*_0, \ell^*_1,\cdots, \ell^*_{n}$ the ordered arcs in $\zD^*$ that successively intersect with $\za$ such that $\ell^*_i$ belongs to $\bbp_{i}$ and $\bbp_{i+1}$ for each $0 \leqslant i \leqslant n$.}

\begin{definition}\label{definition:zigzag arcs}
Let $\za$ be an $\bpoint$-arc on $(\cals,\calm,\zD^*)$. Under the notations in {\bf Setting}, we call $\za$ a \emph{zigzag arc} (with respect to $\zD^*$) if in each polygon $\bbp_{i+1}, 0 \leqslant i \leqslant n+1$, $\ell^*_{i+1}$ is the predecessor or the successor of $\ell^*_{i}$. Furthermore, if $\bbp_{i+1}$ is a polygon that contains a puncture and has $m$ edges with $m\neq 2$, then we also need the puncture is not in the (unique) triangle formed by the segments of $\ell^*_{i}$, $\ell^*_{i+1}$ and $\za$, see the right picture of Figure \ref{figure:def-zigzag}.
	\begin{figure}
			\begin{tikzpicture}[>=stealth,scale=0.8]
				\draw[red,very thick] (0,0)--(-1,2)--(2,3)--(5,2)--(4,0);
\draw [ gray!60, line width=4pt] (5,-.07)--(-1,-.07);	
\draw[line width=1.5pt](5,0)--(-1,0);\draw[red,thick,fill=red] (0,0) circle (0.1);
				\draw[red,thick,fill=red] (-1,2) circle (0.07);
				\draw[red,thick,fill=red] (2,3) circle (0.07);
				\draw[red,thick,fill=red] (5,2) circle (0.07);
				\draw[red,thick,fill=red] (4,0) circle (0.07);
				
				\draw[thick,fill=white] (2,0) circle (0.07);
				\node[red] at (.8,.5) {$\bbp_{i+1}$};
				\node at (.8,1.5) {\tiny$\za$};
				\node[red] at (3.9,2.7) {\tiny$\ell^*_{i_{j+1}}$};
				\node[red] at (0,2.7) {\tiny$\ell^*_{{i_j}}$};
				\node[red] at (-.9,.7) {\tiny$\ell^*_{{i_1}}$};	
    			\node[red] at (4.9,.7) {\tiny$\ell^*_{{i_m}}$};	
				\draw[bend right,thick](-1,2.5)to(5,2.5);
				\draw[bend right,thick,->](1.4,2.8)to(2.6,2.8);
				\node at (2,2.3) {\tiny$\aaa_{i+1}$};
				\node at (7,1.5) {};
			\end{tikzpicture}	
			\begin{tikzpicture}[>=stealth,scale=0.8]
				\draw[red,very thick] (0,0)--(-1,2)--(2,3)--(5,2)--(4,0);
				\draw[red,very thick](4,0)--(0,0);
				\draw[red,thick,fill=red] (0,0) circle (0.07);
				\draw[red,thick,fill=red] (-1,2) circle (0.07);
				\draw[red,thick,fill=red] (2,3) circle (0.07);
				\draw[red,thick,fill=red] (5,2) circle (0.07);
				\draw[red,thick,fill=red] (4,0) circle (0.07);
				
				\draw[thick,fill=white] (2,1) circle (0.07);
				\node[red] at (.8,.5) {$\bbp_{i+1}$};
				\node at (.8,1.5) {\tiny$\za$};
			\node[red] at (3.9,2.7) {\tiny$\ell^*_{i_{j+1}}$};
				\node[red] at (0,2.7) {\tiny$\ell^*_{{i_j}}$};
				\node[red] at (-.9,.7) {\tiny$\ell^*_{{i_1}}$};	
    			\node[red] at (5,.7) {\tiny$\ell^*_{{i_{m-1}}}$};	
    			\node[red] at (2,.3) {\tiny$\ell^*_{{i_m}}$};	
				\draw[bend right,thick](-1,2.5)to(5,2.5);
				\draw[bend right,thick,->](1.4,2.8)to(2.6,2.8);
				\node at (2,2.3) {\tiny$\aaa_{i+1}$};
				
		\end{tikzpicture}
		\begin{center}
			\caption{Two types of polygons formed by arcs in a simple coordinate $\zD^*$ and boundary segments, where a zigzag arc $\za$ passes through the polygon along an oriented intersection $\aaa_{i+1}$ of arcs in $\zD^*$. For the right polygon with a puncture, if $m\neq 2$, we need the puncture is not in the triangle formed by the segments of $\ell^*_{i_j}$, $\ell^*_{i_{j+1}}$ and $\za$.}\label{figure:def-zigzag}
		\end{center}
	\end{figure}
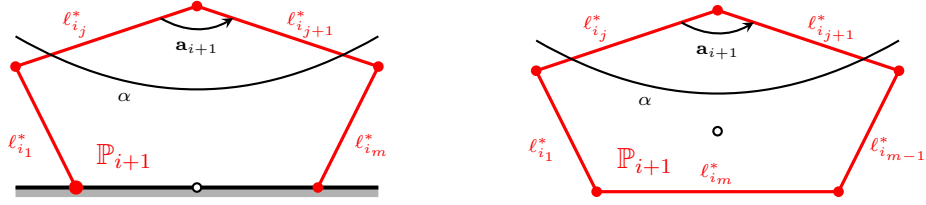
\end{definition}

We mention that the paper \cite{C25} modifies the surface model given in \cite{BC21}, and under this modification, 
a zigzag arc is a permissible arc in \cite{BC21}. 

{\bf Construction. String of a zigzag arc.}

Let $\za$ be a zigzag arc on $(\cals,\calm,\zD^*)$ with notation in {\bf Setting}, we associate a string $\sigma_\za$ of $A$ with it in the following way.
For each $0\leqslant i \leqslant n-1$, there is an arrow in $A$ from $\ell_i^*$ to $\ell_{i+1}^*$ arising from an oriented intersection of $\bbp_{i+1}$, which is a minimal oriented intersection, and we denote it by $\aaa_{i+1}$, see the pictures in Figure \ref{figure:def-zigzag}.
We associate a walk $\sigma_\za=\sigma_1\cdots\sigma_n$ with $\za$, where $\sigma_{i+1}=\aaa_{i+1}$ if $\za$ enter $\bbp_{i+1}$ through $\ell_i^*$ and leave through $\ell_{i+1}^*$, or $\sigma_{i+1}=\aaa^{-1}_{i+1}$ if $\za$ enter $\bbp_{i+1}$ through $\ell_{i+1}^*$ and leave through $\ell_i^*$. Since $\za$ is zigzag, it is straightforward to see that $\sigma_\za$ is a string of $A$.

\begin{definition}\label{prop-def:string and band from curves}
For a zigzag arc $\za$ on $(\cals,\calm,\zD^*)$, we call the (string) module $M_\za$ in $\ma$ arising from $\sigma_\za$ the \emph{string module of $\za$}.
\end{definition}

We have the following bijections, which can be found in \cite[Theorem 2]{BC21} and \cite[Theorem 2.8, 2.12]{C25}.
\begin{proposition}\label{theorem:main arcs and objects}
The map $M: \za\mapsto M_\za$ gives a bijection between zigzag arcs on $(\cals,\calm,\zD^*)$ and indecomposable string modules over $A$. 
In particular, 
\begin{enumerate}[\rm(1)]
\item each simple module is of the form $M_{\ell}$, $\ell\in 
\zD$;
\item each indecomposable injective module is of the form $M_{t^{-1}(\ell^*)}$, $t^{-1}(\ell^*)\in 
\zD_I$;
\item each indecomposable projective module is of the form $M_{t(\ell^*)}$, $t(\ell^*)\in 
\zD_P$.
\end{enumerate}
\end{proposition}

\subsection{Intersections as morphisms and extensions of modules}\label{subsection: sur-ext}
Let's interpret the intersections of zigzag arcs as the morphisms and extensions between the associated modules, see \cite{C25}. 
For convenience, we view any $\bpoint$-point as a zero zigzag arc, and then the module associated with it is just the zero module.

\begin{definition}\label{definition: weight}
Let $\za$ and $\zb$ be two zigzag arcs with an intersection point $p$ in a polygon $\bbp$ of $\zD^*$. 

An \emph{oriented intersection} $\mathfrak{p}$ from $\za$ to $\zb$ is a clockwise angle locally from $\za$ to $\zb$ based on $p$ such that the angle is in the interior of the surface with the convention that if $q$ is an interior point of the surface then the opposite angles are considered equivalent. 

The \emph{weight} $w(\mathfrak{p})$ of $\mathfrak{p}$ is defined as the number of minimal oriented intersections in $\bbp$ which in between the fan from $\za$ to $\zb$ along $\mathfrak{p}$, see the four cases in Figure \ref{figure:weight}. 

	\begin{figure}
		\begin{center}
			\begin{tikzpicture}[>=stealth,scale=.8]
				\draw[red,very thick] (1,0)--(3,1)--(3,3)--(1,4);
				\draw[red,very thick,dashed] (-1,4)--(1,4);
				\draw[red,very thick] (-1,0)--(-3,1)--(-3,3)--(-1,4);
				
				\draw[line width=1pt] (-2.5,4)--(0,0)--(2.5,4);
				\draw [ gray!60, line width=3pt] (-2,-.03)--(2,-.03);		
					\draw[line width=1.5pt] (-2,0)--(2,0);			\draw[thick,fill=white] (0,0) circle (0.07);
				\node[red] at (-1.5,3.5) {\tiny$\ell^*_{t}$};
				\node[red] at (1.5,3.5) {\tiny$\ell^*_{t+\omega}$};
				\node at (-1.7,2) {$\za$};
				\node at (1.7,2) {$\zb$};
				\node[red] at (-2.3,1) {\tiny$\ell^*_{1}$};
				\node[red] at (2.3,1) {\tiny$\ell^*_{n}$};
				\node[red] at (0,2.5) {$\bbp$};
				
				\draw[thick,bend left,->](-.2,.3)to(.2,.3);
				\node [] at (0,.6) {$\mathfrak{p}$};	
				\draw[red,thick,fill=red] (1,0) circle (0.07);
				\draw[red,thick,fill=red] (-1,0) circle (0.07);
				\draw[red,thick,fill=red] (3,1) circle (0.07);
				\draw[red,thick,fill=red] (3,3) circle (0.07);
				\draw[red,thick,fill=red] (1,4) circle (0.07);
				\draw[red,thick,fill=red] (1,0) circle (0.07);
				\draw[red,thick,fill=red] (-3,1) circle (0.07);
				\draw[red,thick,fill=red] (-3,3) circle (0.07);
				\draw[red,thick,fill=red] (-1,4) circle (0.07);
			\draw[] (6,-1);
			\end{tikzpicture}
		\begin{tikzpicture}[>=stealth,scale=.8]
				\draw[red,very thick] (-1,0)--(1,0)--(3,1)--(3,3)--(1,4);
				\draw[red,very thick,dashed] (-1,4)--(1,4);
				\draw[red,very thick] (-1,0)--(-3,1)--(-3,3)--(-1,4);
				
				\draw[line width=1pt] (-2.5,4)--(0,1.6)--(2.5,4);	
                \draw[thick,fill=white] (0,1.6) circle (0.07);
				\node[red] at (-1.5,3.5) {\tiny$\ell^*_{t}$};
				\node[red] at (1.5,3.5) {\tiny$\ell^*_{t+\omega}$};
				\node at (-1.5,2.4) {$\za$};
				\node at (1.5,2.4) {$\zb$};
				\node[red] at (-2.3,1) {\tiny$\ell^*_{1}$};
				\node[red] at (2.3,1) {\tiny$\ell^*_{n-1}$};
				\node[red] at (0,.3) {\tiny$\ell^*_{n}$};
				\node[red] at (0,3) {$\bbp$};
				
                \draw[very thick] (0,1.6) circle (0.4);				
				\draw[thick,->](.2,1.95)to(.3,1.88);
				\node [] at (0,2.3) {$\mathfrak{p}_1$};	
				\draw[thick,->](-.35,1.8)to(-.3,1.88);
				\node [] at (0.1,.9) {$\mathfrak{p}_2$};	
				\draw[red,thick,fill=red] (1,0) circle (0.07);
				\draw[red,thick,fill=red] (-1,0) circle (0.07);
				\draw[red,thick,fill=red] (3,1) circle (0.07);
				\draw[red,thick,fill=red] (3,3) circle (0.07);
				\draw[red,thick,fill=red] (1,4) circle (0.07);
				\draw[red,thick,fill=red] (1,0) circle (0.07);
				\draw[red,thick,fill=red] (-3,1) circle (0.07);
				\draw[red,thick,fill=red] (-3,3) circle (0.07);
				\draw[red,thick,fill=red] (-1,4) circle (0.07);
			\draw[] (-1,-1);
			\end{tikzpicture}
\begin{tikzpicture}[>=stealth,scale=.8]
				\draw[red,very thick,dashed] (1,0)--(3,1);
                \draw[red,very thick] (3,1)--(2,4);
				\draw[red,very thick,dashed] (-1,0)--(-3,1);
                \draw[red,very thick] (-3,1)--(-2,4);
				\draw[red,very thick] (-2,4)--(2,4);

				\draw[line width=1pt] (-1.7,4.7)--(3.7,.9);
				\draw[line width=1pt] (1.7,4.7)--(-3.7,.9);
				\draw [ gray!60, line width=3pt] (-2,-.04)--(2,-.04);		
					\draw[line width=1.5pt] (-2,0)--(2,0);			\draw[thick,fill=white] (0,0) circle (0.07);
				\node at (-1.5,2.8) {$\za$};
				\node at (1.5,2.8) {$\zb$};
				\node [red]at (0,1) {$\bbp$};
				
				\draw[thick,bend left,->](-.25,3.7)to(.25,3.7);
				\node [] at (0,4.2) {$\mathfrak{p}_1$};	
				
				\draw[thick,bend left,->](.25,3.3)to(-.25,3.3);
				\node [] at (0,2.8) {$\mathfrak{p}_1$};	
                
				\draw[thick,bend left,->](.25,3.7)to(.25,3.3);
				\node [] at (.8,3.5) {$\mathfrak{p}_2$};	
                
				\draw[thick,bend left,->](-.25,3.3)to(-.25,3.7);
				\node [] at (-.8,3.5) {$\mathfrak{p}_2$};	                \draw[red,thick,fill=red] (1,0) circle (0.07);
				\draw[red,thick,fill=red] (-1,0) circle (0.07);
				\draw[red,thick,fill=red] (3,1) circle (0.07);
				\draw[red,thick,fill=red] (2,4) circle (0.07);
				\draw[red,thick,fill=red] (1,0) circle (0.07);
				\draw[red,thick,fill=red] (-3,1) circle (0.07);
				\draw[red,thick,fill=red] (-2,4) circle (0.07);
                \draw[] (6,0);
\end{tikzpicture}
\begin{tikzpicture}[>=stealth,scale=.8]
				\draw[red,very thick,dashed] (1,0)--(3,1);
                \draw[red,very thick] (3,1)--(2,4);
				\draw[red,very thick,dashed] (-1,0)--(-3,1);
                \draw[red,very thick] (-3,1)--(-2,4);
				\draw[red,very thick] (-2,4)--(2,4);
                \draw[red,very thick] (-1,0)--(1,0);
				
				\draw[line width=1pt] (-1.7,4.7)--(3.7,.9);
				\draw[line width=1pt] (1.7,4.7)--(-3.7,.9);
	            \draw[thick,fill=white] (0,2) circle (0.07);
				\node at (-1.5,2.8) {$\za$};
				\node at (1.5,2.8) {$\zb$};
				\node[red] at (0,1) {$\bbp$};
				
				\draw[thick,bend left,->](-.25,3.7)to(.25,3.7);
				\node [] at (0,4.2) {$\mathfrak{p}_1$};	
				
				\draw[thick,bend left,->](.25,3.3)to(-.25,3.3);
				\node [] at (0,2.8) {$\mathfrak{p}_1$};	
                
				\draw[thick,bend left,->](.25,3.7)to(.25,3.3);
				\node [] at (.8,3.5) {$\mathfrak{p}_2$};	
                
				\draw[thick,bend left,->](-.25,3.3)to(-.25,3.7);
				\node [] at (-.8,3.5) {$\mathfrak{p}_2$};	                \draw[red,thick,fill=red] (1,0) circle (0.07);
				\draw[red,thick,fill=red] (-1,0) circle (0.07);
				\draw[red,thick,fill=red] (3,1) circle (0.07);
				\draw[red,thick,fill=red] (2,4) circle (0.07);
				\draw[red,thick,fill=red] (1,0) circle (0.07);
				\draw[red,thick,fill=red] (-3,1) circle (0.07);
				\draw[red,thick,fill=red] (-2,4) circle (0.07);
			\end{tikzpicture}
		\end{center}
		\begin{center}
			\caption{An intersection $p$ between two zigzag arcs $\za$ and $\zb$ gives rise to oriented intersections. When $p$ is a boundary point, it gives rise to a unique oriented intersection $\mathfrak{p}$ with weight $\omega$, see the left-above picture. When $p$ is a puncture, there are infinitely many oriented intersections from $\za$ to $\zb$ with weights $mn+\omega$, and infinitely many oriented intersections from $\zb$ to $\za$ with weights $mn-\omega$. In particular, $\omega(\mathfrak{p}_1)=\omega$ and $\omega(\mathfrak{p}_2)=n-\omega$ for $\mathfrak{p}_1$ and $\mathfrak{p}_2$ in the right-above picture. When $p$ is an interior non-puncture point on the surface, then it gives rise to an oriented intersection $\mathfrak{p}_1$ from $\za$ to $\zb$ with weight $\omega(\mathfrak{p}_1)=0$, and an oriented intersection $\mathfrak{p}_2$ from $\zb$ to $\za$ with weight $\omega(\mathfrak{p}_2)=1$, see the bottom two pictures.}\label{figure:weight}
		\end{center}
	\end{figure}
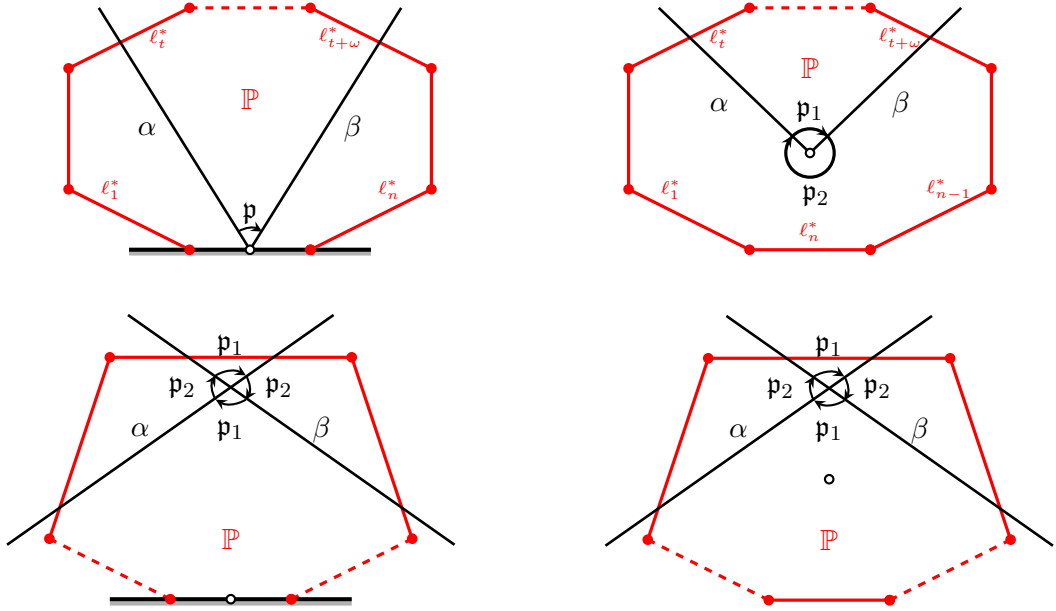
\end{definition}

\begin{remark}
The definition of an oriented intersection between $\bpoint$-arcs based on a common endpoint $p$ is similar to the definition of an oriented intersection between $\rpoint$-arcs $\ell^*$ in $\zD^*$ given in Definition \ref{definition:oriented intersection}. The difference is that here we use a clockwise orientation, rather than an anticlockwise orientation. We also mention that for the definition to be well-defined in the case of a loop, we treat the unique endpoint of a loop as two distinct endpoints.

There are three cases of an oriented intersection that arises from an intersection point $p$, depending on the position of $p$.

(1) If $p$ is a common endpoint of $\za$ and $\zb$ from $\calm_\bpoint$, then $p$ gives rises to a unique oriented intersection $\mathfrak{p}$, see the left-above picture in Figure \ref{figure:weight}.

(2) If $p$ is a common endpoint of $\za$ and $\zb$ from $\calp_\bpoint$, then there are infinitely many oriented intersections arising from $p$, see the right-above picture in Figure \ref{figure:weight}. 

(3) If $p$ is an interior point, then there are two oriented intersections arising from $p$, and the weights of them are, respectively, zero and one; see the bottom two pictures in Figure \ref{figure:weight}. The pictures may degenerate, that is, $\za$ or $\zb$ may start at the marked $\bpoint$-point in the polygon $\bbp$.
\end{remark}

The following theorem is given in \cite[Theorem 2.30]{C25}.

\begin{proposition}\label{prop:main-extensions}
Let $(\cals,\calm,\zD^*)$ be a coordinated-marked surface, and let $\za$ and $\zb$ be two zigzag arcs on it. Then for any \emph{oriented intersection} from $\za$ to $\zb$ with weight $\omega$, there is a morphism in $\Ext^\omega(M_\za,M_\zb)$ associated with it. Furthermore, all of such morphisms form a basis of the space $\Ext^\omega(M_\za,M_\zb)$, unless $\za$ and $\zb$ are the same $\bpoint$-arc. In this case, the identity map is the extra basis map of $\Hom(M_\za,M_\zb)$.
\end{proposition}

The idea of the proof of the above proposition is to embed the geometric model for the module category of a gentle algebra into the geometric model of the derived category, in the sense that each zigzag arc on the surface represents an indecomposable module as well as the minimal projective resolution of this module (concerning a so-called projective coordinate). Now, assume that we have a triangle formed by zigzag arcs on the surface. Then it gives a distinguished triangle in the derived category of the gentle algebra. In particular, the sum of the weights of the oriented intersection is one. We write this observation as a lemma which will be cited frequently in the following.

\begin{lemma}\label{lem:sumwt}
Assume that we have a triangle formed by zigzag arcs on $(\cals,\calm,\zD^*)$, and denote by $\mathfrak{p}_1, \mathfrak{p}_2, \mathfrak{p}_3$ the oriented intersections in it.
Then the sum of the weights of the oriented intersections is one:
\[\omega(\mathfrak{p}_1)+\omega(\mathfrak{p}_2)+\omega(\mathfrak{p}_3)=1.\]
\end{lemma}

\section{Maximal rigid modules as admissible $5$-partial triangulations}\label{section:max-rig-mod}
This section studies the rigid modules over a gentle algebra, that is, the modules with trivial once-self-extensions. We realize maximal rigid modules as admissible $5$-partial triangulations and describe their rank.

\subsection{Arc systems and $s$-partial triangulations}\label{section:max-rig-mod1}

\begin{definition}\label{def:par-tri}
A collection of simple zigzag arcs on a coordinated-marked surface $(\cals,\calm, \zD^*)$ is called an \emph{arc system} if the arcs have no interior intersections. Furthermore, it is called \emph{s-partial triangulation}, if the arcs cut the surface into polygons, each of which contains at most one $\rpoint$-point from $\calm_{\rpoint}$, and the number of the edges of each polygon is bounded by $s$.
\end{definition}

We call an edge in a partial triangulation an \emph{internal edge} if it is contained in two polygons, while we call it an \emph{external edge} if it is contained in one polygon. Then an internal edge is an $\bpoint$-arc, and an external edge is a boundary segment with one $\bpoint$-endpoint and one $\rpoint$-endpoint. 
We call a polygon an \emph{internal polygon}, if it is formed by internal edges, while we call it an \emph{external polygon}, if there exist external edges in it. For example, in Figure \ref{fig:ply}, \textbf{F1}, \textbf{F2} and \textbf{F3} are external polygons and \textbf{F4} and \textbf{F5} are internal polygons.

\begin{definition}\label{def:par-tri2}
An \emph{admissible arc system} on $(\cals,\calm,\zD^*)$ is an arc system such that the weight of each oriented intersection (concerning $\zD^*$) of two arcs is different from one.
An \emph{maximal admissible arc system} is an admissible arc system such that there is no other admissible arc system that strictly contains it. 
An \emph{admissible s-partial triangulation} on $(\cals,\calm,\zD^*)$ is an s-partial triangulation which is admissible.
\end{definition}

An admissible arc system will be denoted by $\calp$, and an admissible s-partial triangulation will be denoted by $\calp_s$. Then an admissible $3$-partial triangulation is an ordinary triangulation on $(\cals,\calm)$ consisting of zigzag arcs concerning $\zD^*$, such that the weight of each oriented intersection is different from one. There is a chain of the sets of admissible arc systems and admissible $s$-partial triangulations: 
$$\{\calp_3\}\subset \{\calp_4\}\cdots\subset \{\calp_s\}\cdots\subset \{\calp\}.$$

\begin{proposition}\label{prop:max-m1}
An admissible arc system $\calp=\{\zg_i, 1\leqslant i \leqslant t\}$ on $(\cals,\calm,\zD^*)$ gives rise to a rigid module $M=\bigoplus_{1\leqslant i \leqslant t}M_{\zg_i}$ in $\ma$. 
In particular, this gives a bijection between the set of (maximal resp.) admissible arc systems on $(\cals,\calm,\zD^*)$ and the set of (maximal resp.) rigid modules in $\ma$.  
\end{proposition}
\begin{proof}
Since a band module has non-trivial once self-extension, we only consider the string module, that is, a module $M_\zg$ given by a zigzag arc $\zg$.
Let $\calp$ be an arc system. Then any homomorphism and extension between modules
$M_{\zg_i}$ and $M_{\zg_j}$ arise from oriented intersections at common endpoints, since there is no interior intersection between arcs $\zg_i$ and $\zg_j$. On the other hand, since $\calp$ is admissible, the weight of each oriented intersection of two arcs at a common endpoint is different from one. Therefore, by Proposition \ref{prop:main-extensions}, $\Ext^1(M_{\zg_i}, M_{\zg_j})=0$. Thus, $M=\bigoplus_{1\leqslant i \leqslant t}M_{\zg_i}$ is a rigid module.  

Conversely, if $M=\bigoplus_{1\leqslant i \leqslant t}M_{\zg_i}$ is a rigid module in $\ma$, then each arc $\zg_i$ is a simple arc, and a converse argument shows that $\{\zg_i,1\leqslant i \leqslant t\}$ is an admissible arc system.

Finally, it is not hard to see that the statement for the maximal rigid modules holds.
\end{proof}

It is noted that arc systems are widespread on a given coordinated-marked surface. Naturally, we wonder if there are many admissible arc systems.
We will show that it is easy to construct an admissible arc system from a given arc system by using the flips of arcs introduced in the following.

\begin{definition}\label{def:flip}
Let $\za_1$ and $\za_2$ be two simple (different) zigzag arcs on $(\cals,\calm,\zD^*)$ sharing a common endpoint $p$. Assume that $p$ gives rise to an oriented intersection $\mathfrak{p}$ from $\za_1$ to $\za_2$, and $\omega(\mathfrak{p})=1$. Denote by $\za$ the $\bpoint$-arc obtained by smoothing $\za_1$ and $\za_2$ at $p$. We call $\za$ the \emph{flip} of $\za_1$ and $\za_2$ at $p$.
\end{definition}

Since $\za_1$ and $\za_2$ are both zigzag arcs and the weight of $\mathfrak{p}$ equals one, the flip $\za$ is a zigzag arc, c.f. Figure \ref{fig:flip}. Furthermore, for a given zigzag arc $\zb$ which is not starting at $p$ and in between $\za_1$ and $\za_2$, since $\za$, $\za_1$, and $\za_2$ form a contractible triangle on the surface, there is no interior-oriented intersection between $\za$ and $\zb$ if and only if there is no interior-oriented intersection between $\za_i$, $i=1,2$, and $\zb$. In addition to the common endpoint $p$, we denote by $p_i$ the endpoints of $\za_i$, $i=1,2$, and denote by $\aaa_i$ the oriented intersections between $\za$ and $\za_i$ arising from $p_i$, see Figure \ref{fig:flip}. Then it is clear that the weights of $\aaa_1$ and $\aaa_2$ are zero.
We have the following lemma, which is useful in this section. 

\begin{lemma}\label{lem:flip}
Under the above notation, let $\zb$ be a zigzag arc that intersects $\za$ at $p_1$.

(1) If there is an oriented intersection $\mathfrak{p}_1$ from $\za_1$ to $\zb$, then $\mathfrak{p}'_1=\aaa_1\mathfrak{p}_1$ is an oriented intersection from $\za$ to $\zb$, and the weights of $\mathfrak{p}_1$ and $\mathfrak{p}'_1$ coincide. Conversely, any oriented intersection from $\za$ to $\zb$ arising from $p_1$ is of this form, except when $\zb$ coincides with $\za_1$, for which $\aaa_1$ is an oriented intersection from $\za$ to $\zb$ and there is no oriented intersection from $\za_1$ to $\zb$.

(2) If there is an oriented intersection $\mathfrak{p}'_1$ from $\zb$ to $\za$, then $\mathfrak{p}_1=\mathfrak{p}'_1\aaa_1$ is an oriented intersection from $\zb$ to $\za_1$, and the weights of $\mathfrak{p}_1$ and $\mathfrak{p}'_1$ coincide. Conversely, any oriented intersection from $\zb$ to $\za_1$ arising from $p_1$ is of this form, except when $\zb$ coincides with $\za$, for which $\aaa_1$ is an oriented intersection from $\zb$ to $\za_1$ and there is no oriented intersection from $\zb$ to $\za$.

Similar statements hold for a zigzag arc $\zb$ that intersects $\za$ at $p_2$.
\end{lemma}
\begin{proof}
The proof of the lemma is straightforward, seeing the picture in Figure \ref{fig:flip}.
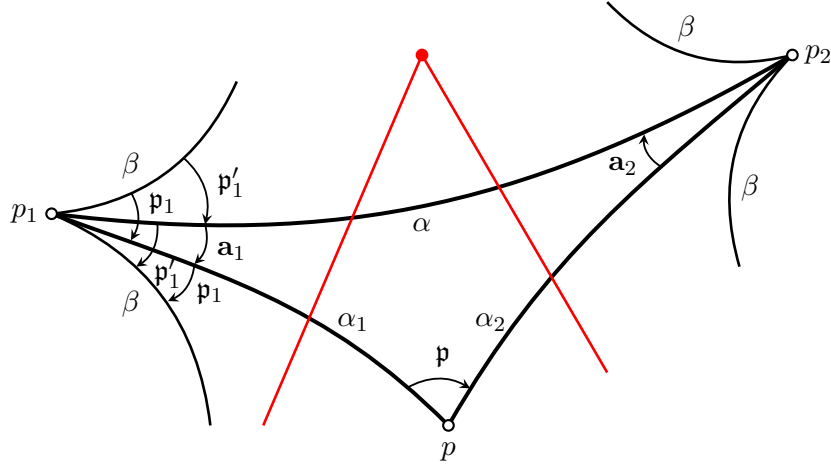
\begin{figure} 
\begin{tikzpicture}[>=stealth,scale=.7]
					
\draw [bend right, line width=1pt] (-7,0) to (-3.5,2.5);
\draw [bend left, line width=1pt] (-7,0) to (-4,-4);
\draw [bend left, line width=1pt] (7,3) to (3.5,4);
\draw [bend right, line width=1pt] (7,3) to (6,-1);
\draw [bend left, line width=.7pt, ->] (-5.5,0.4) to (-5.5,-0.5);
\draw [bend left, line width=.7pt, ->] (-5,-0.2) to (-5.4,-1);
\draw [bend left, line width=.7pt, ->] (-4.5,1.05) to (-4.1,-0.2);
\draw [bend left, line width=.7pt, ->] (-4.1,-0.2) to(-4.3,-0.95) ;
\draw [bend left, line width=.7pt, ->] (-4.3,-0.95) to(-4.8,-1.65) ;
\draw [bend left, line width=.7pt, ->] (-0.3,-3.3) to(.91,-3.3) ;
\draw [bend left, line width=.7pt, ->] (4.5,0.9) to(4.2,1.5) ;

\draw[cyan,line width=1.5pt,black] (-7,0) to[out=-7,in=-150](7,3);
\draw[cyan,line width=1.5pt,black] (-7,0) to[out=-20,in=135](0.5,-4);
\draw[cyan,line width=1.5pt,black] (7,3) to[out=-140,in=60](0.5,-4);
\draw [ line width=1pt,red] (0,3) to(-3,-4) ;
\draw [ line width=1pt,red] (0,3) to(3.5,-3) ;
\draw[thick,fill=white] (-7,0) circle (0.1);\draw[thick,fill=white] (7,3) circle (0.1);
\draw[thick,red, fill=red] (0,3) circle (0.1);
\draw[thick,fill=white] (0.5,-4) circle (0.1);
\draw (0.5,-4.5) node {$p$};
\draw (0.4,-2.7) node {$\mathfrak{p}$};
					
\draw (7.5,3) node {$p_2$};
\draw (5,3.5) node {$\beta$};
\draw (6.2,0.5) node {$\beta$};
\draw (0,-.2) node {$\alpha$};
\draw (3.8,.9) node {$\aaa_{2}$};
					
\draw (-7.5,0) node {$p_1$};
\draw (-5.5,0.9) node {$\beta$};
\draw (-5.5,-1.7) node {$\beta$};
\draw (-4.9,0.2) node {$\mathfrak{p}_1$};
\draw (-3.6,0.6) node {$\mathfrak{p}'_1$};
\draw (-4,-1.5) node {$\mathfrak{p}_1$};
\draw (-4.8,-1.12) node {$\mathfrak{p}'_1$};
\draw (-1.3,-2.05) node {$\alpha_1$};
\draw (-3.6,-.7) node {$\aaa_1$};
\draw (1.3,-2.05) node {$\alpha_2$};
\end{tikzpicture}
	\caption{The arc $\za$ is the flip of $\za_1$ and $\za_2$ at a common endpoint $p$. Then $\za$ is a zigzag arc. For any zigzag arc $\zb$, there is a bijection between oriented intersections $\mathfrak{p}_i$ of $\za_i$ and $\zb$ at $p_i$, and oriented intersections $\mathfrak{p}'_i$ of $\za$ and $\zb$ at $p_i$, and we have $|\mathfrak{p}_i|=|\mathfrak{p}'_i|$, see the precise statements in Lemma \ref{lem:flip}. } 
	\label{fig:flip} 
\end{figure}

\end{proof}
Let $\calp$ be an arc system on $(\cals,\calp,\zD^*)$. Assume that there are two arcs $\za_1$ and $\za_2$ that share a common endpoint $p$ such that the associated oriented intersection $\mathfrak{p}$ has weight one.  Without loss of generality, we further assume $\mathfrak{p}$ is minimal, that is, there is no arc starting from $p$ and in between $\za_1$ and $\za_2$. We replace $\za_1$ or $\za_2$ in $\calp$ by the flip $\za$ of them at $p$, and denote the new set of arcs by $\calp'$. Note that if $\za$ is already in $\calp$, then we just delete $\za_1$ or $\za_2$. By Lemma \ref{lem:flip}, $\calp'$ is an arc system, and the number of oriented intersections with weight one reduces at least one (there may exist other arcs in $\calp$ which have endpoint $p$, except for $\za_1$ and $\za_2$). Since there are only finitely many oriented intersections between the arcs in $\calp$, starting from $\calp$, we finally obtain an admissible arc system by literally replacing the arcs as above. Denote the final admissible arc system by $\calp_{\textbf{adm}}$. Furthermore, if the initial arc system is a partial triangulation, then the final admissible arc system is still a partial triangulation since there is no $\rpoint$-point in the triangle formed by the arcs $\za_i$ and their flip $\za$.

Flip of arcs gives rise to an equivalent relation on the set of arc systems as well as the set of partial triangulations, see the following proposition, whose proof is straightforward.

\begin{proposition-definition}\label{prop-def:flip}
Let $\calp$ and $\calp'$ be two arc systems (partial triangulations, respectively). We call them \emph{flip equivalent}, if the associated admissible arc systems (partial triangulation, respectively) $\calp_{\textbf{adm}}$ and $\calp'_{\textbf{adm}}$ coincide. The flip equivalence gives rise to an equivalent relation on the set of arc systems as well as the set of partial triangulations. 
\end{proposition-definition}

It seems interesting to study more about the properties of the flip equivalence, for example, the relations with derived equivalences.
The rest of this section is devoted to describing maximal admissible arc systems. We will show that they are exactly the $5$-partial triangulations.

\subsection{Reductions for rigid modules by cutting surface}\label{section:max-rig-mod2}
This subsection prepares some reduction tools to prove the main theorems in the next subsection.

Let $(\cals,\calm,\zD^*)$ be a coordinated-marked surface and let $\zg$ be a simple zigzag arc on the surface with endpoints $p_1,p_2 \in \calm_\bpoint\cup \calp_\bpoint$.
In \cite{C24}, the author introduced a new coordinated-marked surface $(\cals_\zg,{\calm_\zg},{\zD^*_\zg})$ obtained by cutting $(\cals,\calm,\zD^*)$ along $\zg$. We briefly recall the construction as follows. The details can be found in \cite[Section 2.2]{C24}.

The surface $\cals_\zg$ is obtained by cutting $\cals$ along $\zg$, where $\zg$ becomes two boundary segments, which are denoted by $\zg'$ and $\zg''$, with endpoints $p_1',p_2'$ and $p_1'',p_2''$ respectively.

The set $\calm_\zg$ is defined as 
$$\calm_\zg=\calm\setminus \{p_1,p_2\}\cup\{p_1',p_2',p_1'',p_2''\}\cup\{q',q''\},$$
where $q'$ and $q''$ are newly added $\rpoint$-points which locate on $\zg'$ and $\zg''$ respectively.
Note that two vertices in $\{p_1',p_2',p_1'',p_2''\}$ may coincide. For example, when $p_i$ is a puncture, we have $p'_i=p''_i$, see the list of all cases in the Appendix of \cite{C24}. 

The set $\zD^*_\zg$ is obtained in the following steps.
Let $\call=\{\ell^*\}$ be the set of arcs in $\zD^*$ that intersect $\zg$. Assume that $\zg$ cuts each $\ell^*$ into several segments. There are two types of endpoints in such a segment, it may be an original endpoint of $\ell^*$ or an interior intersection point with $\zg$. When cutting the surface along $\zg$, we smoothly move the second type of endpoint to the newly added $\rpoint$-point $q'/q''$ along $\zg$. In this way, we obtain new $\rpoint$-arcs on $(\cals_\zg,\calm_\zg)$, which are denoted by $\ell^*_1, \ell^*_2,\cdots,\ell^*_t$. Finally, $\zD^*_\zg$ is defined as the set 
$$\zD^*_\zg=\zD^*\setminus \call\cup\{\ell^*_1,\cdots,\ell^*_t ~| ~\ell^*\in \call\}$$ of $\rpoint$-arcs on $\cals_\zg$, where we identify the arcs which are homotopic with each other.

\begin{example}\label{ex:cutting}
See Figure \ref{fig:ex-cutting} for a concrete example of cutting a coordinated-marked surface.
	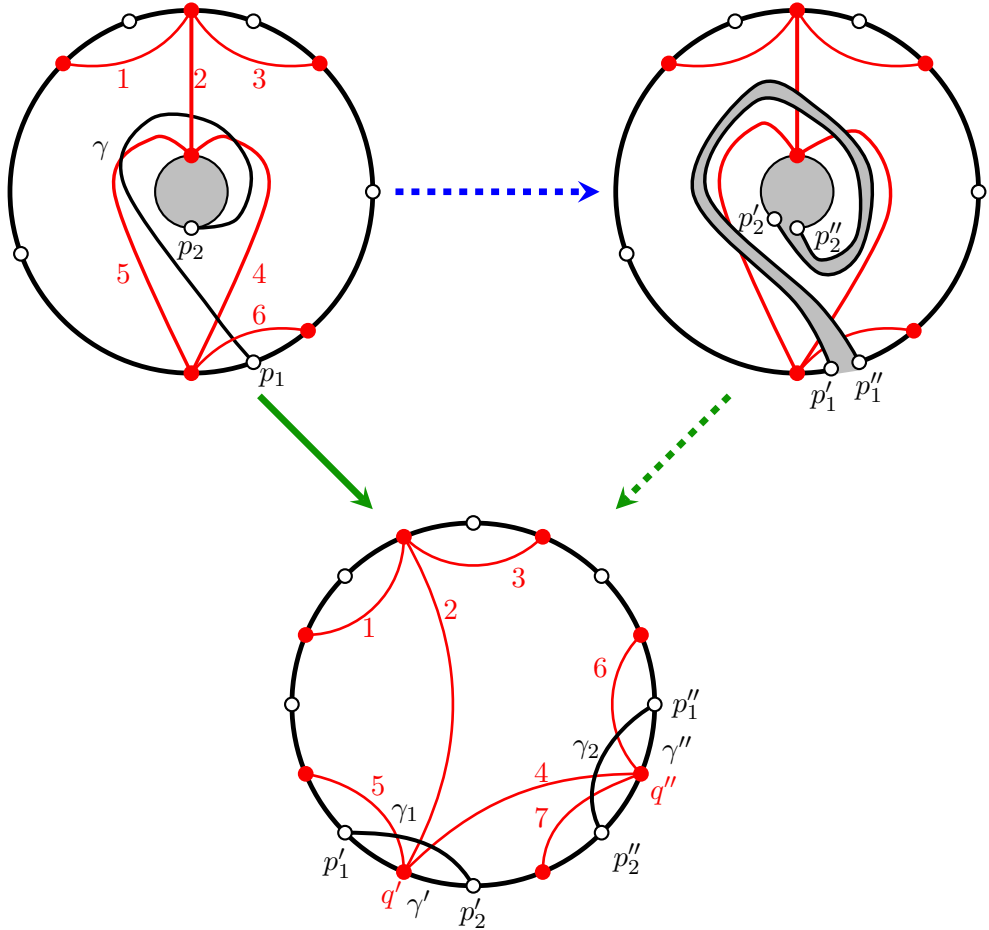
\begin{figure}[ht]
		\begin{center}
\begin{tikzpicture}[>=stealth,scale=0.6]
\draw[line width=1.8pt,fill=white] (0,0) circle (4cm);
				\draw[thick,fill=gray!50] (0,0) circle (0.8cm);
				\path (0:4) coordinate (b1)
				(70:4) coordinate (b2)
				(110:4) coordinate (b3)
				(135:4) coordinate (b4)
				(200:4) coordinate (b5)
				(90:4) coordinate (b6)
				(-90:4) coordinate (b7)
				(-70:4) coordinate (b8)
				(-50:4) coordinate (b9)
				(45:4) coordinate (b10);
				
				\draw[red,line width=1.3pt]plot [smooth,tension=0.6] coordinates {(0,-4) (-1.7,0) (-0.8,1.2) (0,0.8)};
				\draw[red,line width=1.3pt]plot [smooth,tension=0.6] coordinates {(0,-4) (1.7,0) (.8,1.2) (0,0.8)};
				\draw[black,line width=1.3pt]plot [smooth,tension=0.6] coordinates {(b8) (0.5,-2.7) (-1.5,0.2) (-0.9,1.6)(0.5,1.5)(1.3,0.5)(1,-.6)(0,-0.8)};			
				\draw [red, line width=1.5pt] (b6) to (0,0.8);
				\draw[cyan,line width=1pt,red] (b6) to[out=-120,in=-10](b4);
				\draw[cyan,line width=1pt,red] (b6) to[out=-60,in=-170](b10);
				\draw[cyan,line width=1pt,red] (b7) to[out=50,in=170](b9);
				\draw[red] (-1.5,2.5) node {$1$};
				\draw[red] (.2,2.5) node {$2$};
				\draw[red] (1.5,2.5) node {$3$};
				\draw[red] (1.5,-1.8) node {$4$};
				\draw[red] (-1.5,-1.8) node {$5$};
				\draw[red] (1.5,-2.7) node {$6$};
				\draw[black] (-2,0.9) node {$\gamma$};
				\draw[black] (0,-1.3) node {$p_{2}$};
				\draw[black] (1.8,-4.1) node {$p_{1}$};
	\draw[thick,fill=white] (b1) circle (0.15cm)
				(b2) circle (.15cm)
				(b3) circle (.15cm)
				(b5) circle (.15cm)
				(b8) circle (.15cm);

				\draw[thick,red, fill=red] (b4) circle (0.15cm)
				(b6) circle (.15cm)
				(b7) circle (.15cm)
				(b9) circle (.15cm)
				(b10) circle (.15cm);
				
				\draw[thick, red ,fill=red] (0,0.8) circle (0.15cm);
				\draw[thick, black ,fill=white] (0,-0.8) circle (0.15cm);

				\draw[line width=2.5pt,blue, dashed, ->] (4.5, 0) -- (9, 0) ;
				\draw[line width=2.5pt,dark-green,  ->] (1.5, -4.5) -- (4, -7) ;			
\end{tikzpicture}
\begin{tikzpicture}[>=stealth,scale=0.6]
				\draw[line width=1.8pt,fill=white] (0,0) circle (4cm);
				\draw[thick,fill=gray!50] (0,0) circle (0.8cm);

				\path (0:4) coordinate (b1)
				(70:4) coordinate (b2)
				(110:4) coordinate (b3)
				(135:4) coordinate (b4)
				(200:4) coordinate (b5)
				(90:4) coordinate (b6)
				(-90:4) coordinate (b7)
				(-70:4) coordinate (b8)
				(-50:4) coordinate (b9)
				(45:4) coordinate (b10);
				
				\draw[red,line width=1.3pt]plot [smooth,tension=0.6] coordinates {(0,-4) (-1.7,0) (-0.8,1.5) (0,0.8)};
				\draw[red,line width=1.3pt]plot [smooth,tension=0.6] coordinates {(0,-4) (2,-.5) (1.4,1.3) (0,0.8)};

				\draw [red, line width=1.5pt] (b6) to (0,0.8);
				\draw[cyan,line width=1pt,red] (b6) to[out=-120,in=-10](b4);
				\draw[cyan,line width=1pt,red] (b6) to[out=-60,in=-170](b10);
				\draw[cyan,line width=1pt,red] (b7) to[out=50,in=170](b9);

			\draw[black] (-1,-0.6) node {$p_{2}'$};
				\draw[black] (0.7,-1) node {$p_{2}''$};
				\draw[black] (0.6,-4.5) node {$p_{1}'$};
				\draw[black] (1.6,-4.3) node {$p_{1}''$};
				\draw[black,line width=1.3pt, fill=gray!50]plot [smooth,tension=0.6] coordinates {(0.75,-3.9) (0,-2.5) (-2.3,0) (-0.9,2.35)(0.5,1.8)(1.6,0.5)(1.4,-1.3)(0.5,-1.8)(-0.5,-0.6)(0,-0.8)(0.5,-1.5)(1.3,-1)(1.3,0.5)(0.5,1.5)(-0.9,2)(-2,0)(0.5,-2.5) (b8)};
				
				\draw [gray!50, line width=4.5pt] (-0.5,-0.5) to (0,-0.8);
				\draw [gray!50, line width=4.5pt] (0.75,-3.9) to (1.4,-3.8);
	\draw[thick,fill=white] (b1) circle (0.15cm)
				(b2) circle (.15cm)
				(b3) circle (.15cm)
				(b5) circle (.15cm)
				(b8) circle (.15cm);

				\draw[thick,red, fill=red] (b4) circle (0.15cm)
				(b6) circle (.15cm)
				(b7) circle (.15cm)
				(b9) circle (.15cm)
				(b10) circle (.15cm);
				\draw[thick, red ,fill=red] (0,0.8) circle (0.15cm);
				\draw[thick, black ,fill=white] (0,-0.8) circle (0.15cm);
				\draw[thick, black ,fill=white] (-0.5,-0.6) circle (0.15cm);
				\draw[thick, black ,fill=white] (0.75,-3.9) circle (0.15cm);
				\draw[line width=2.5pt,dark-green, dashed, ->] (-1.5, -4.5) -- (-4, -7) ;
		\end{tikzpicture}
		\begin{tikzpicture}[>=stealth,scale=0.6]
				\draw[line width=1.8pt,fill=white] (0,0) circle (4cm);
			
				\path (0:4) coordinate (b1)
				(45:4) coordinate (b2)
				(90:4) coordinate (b3)
				(135:4) coordinate (b4)
				(180:4) coordinate (b5)
				(225:4) coordinate (b6)
				(270:4) coordinate (b7)
				(315:4) coordinate (b8)
				
				(22.5:4) coordinate (r1)
				(67.5:4) coordinate (r2)
				(112.5:4) coordinate (r3)
				(157.5:4) coordinate (r4)
				(202.5:4) coordinate (r5)
				(247.5:4) coordinate (r6)
				(292.5:4) coordinate (r7)
				(337.5:4) coordinate (r8);
				\draw[cyan,line width=1pt,red] (r3) to[out=-90,in=0](r4);
				\draw[cyan,line width=1pt,red] (r3) to[out=-45,in=-135](r2);
				\draw[cyan,line width=1pt,red] (r3) to[out=-60,in=60](r6);
				\draw[cyan,line width=1pt,red] (r6) to[out=90,in=-10](r5);
				\draw[cyan,line width=1pt,red] (r6) to[out=45,in=180](r8);
				\draw[cyan,line width=1pt,red] (r8) to[out=-160,in=90](r7);
				\draw[cyan,line width=1pt,red] (r8) to[out=135,in=-135](r1);
				\draw[cyan,line width=1.5pt,black] (b6) to[out=0,in=120](b7);
				\draw[cyan,line width=1.5pt,black] (b8) to[out=120,in=-150](b1);
				\draw[thick,fill=white] (b1) circle (0.15cm)
				(b2) circle (.15cm)
				(b3) circle (.15cm)
				(b4) circle (.15cm)
				(b5) circle (.15cm)
				(b6) circle (.15cm)
				(b7) circle (.15cm)
				(b8) circle (.15cm);

				\draw[thick,red, fill=red] (r1) circle (0.15cm)
				(r2) circle (0.15cm)
				(r3) circle (0.15cm)
				(r4) circle (0.15cm)
				(r5) circle (0.15cm)
				(r6) circle (0.15cm) 
				(r7) circle (0.15cm)
				(r8) circle (0.15cm);	
				
				\draw[red] (-2.3,1.7) node {$1$};
				\draw[red] (-0.5,2.1) node {$2$};
				\draw[red] (1,2.8) node {$3$};
				\draw[red] (1.5,-1.5) node {$4$};
				\draw[red] (-2.1,-1.8) node {$5$};
				\draw[red] (2.8,0.8) node {$6$};
				\draw[red] (1.5,-2.5) node {$7$};
				\draw[red] (-1.8,-4.2) node {$q'$};
				\draw[red] (4.2,-1.9) node {$q''$};
				
				\draw[black] (-1.5,-2.4) node {$\gamma_1$};
				\draw[black] (-1.2,-4.4) node {$\gamma'$};
				\draw[black] (2.5,-1) node {$\gamma_2$};
				\draw[black] (4.5,-1) node {$\gamma''$};
				\draw[black] (0,-4.6) node {$p_{2}'$};
				\draw[black] (3.4,-3.4) node {$p_{2}''$};
				\draw[black] (-3,-3.5) node {$p_{1}'$};
				\draw[black] (4.7,0) node {$p_{1}''$};
	\end{tikzpicture}
		\end{center}	
  \caption{
 An example of the cutting surface.} 
	\label{fig:ex-cutting}
	\end{figure}
\end{example}

A key observation in \cite{C24} is that $\zD^*_\zg$ is a simple coordinate on $(\cals_\zg,\calm_\zg)$, see \cite[Proposition-Definition 2.6]{C24}.
We call $(\cals_\zg,{\calm_\zg},{\zD^*_\zg})$ the \emph{cutting surface} of $(\cals,\calm,\zD^*)$ along $\zg$.

Denote by $\zg_1$ and $\zg_2$ the $\bpoint$-arcs on $(\cals_\zg,\calm_\zg)$ that form triangles with the new boundary segments $\zg'$ and $\zg''$, respectively. Note that these two triangles contain the two newly added $\rpoint$-points $q'$ and $q''$, and both $\zg_1$ and $\zg_2$ are simple zigzag arcs on $(\cals_\zg,{\calm_\zg},{\zD^*_\zg})$, see Figure \ref{fig:ex-cutting} for an example.

Let $\za$ be an $\bpoint$-arc
on $(\cals,\calm)$. If $\za$ and $\zg$ intersect in the interior of the surface, then it disappears when cutting the surface along $\zg$. If $\za=\zg$, then it induces two $\bpoint$-arcs $\zg_1$ and $\zg_2$ as described above. Now assume that $\za\neq \zg$ does not have interior intersections with $\zg$. Then $\za$ and $\zg$ share common endpoints or are completely disjoint. For both cases, $\za$ induces a (unique) $\bpoint$-arc on $(\cals_\zg,\calm_\zg)$, which is denoted by $\widehat\za$. The associated map will be denoted by $\widehat{\bullet}$.
We introduce the following four sets of arcs:
$$\mathfrak{A}:=\{\text{$\bpoint$-arc $\za$ on $(\cals,\calm,\zD^*)$ which has no interior intersections with $\zg$}\}\setminus \{\zg\};$$ 
$$\mathfrak{Z}:=\{\text{zigzag arc $\za$ on $(\cals,\calm,\zD^*)$ which has no interior intersections with $\zg$}\}\setminus \{\zg\};$$
$$\mathfrak{A}_\zg:=\{\text{$\bpoint$-arc $\widehat{\za}$ on $(\cals_\zg,\calm_\zg,\zD_\zg^*)$}\}\setminus \{\zg_1,\zg_2\};$$ 
$$\mathfrak{Z}_\zg:=\{\text{zigzag arc $\widehat{\za}$ on $(\cals_\zg,\calm_\zg,\zD_\zg^*)$}\}\setminus \{\zg_1,\zg_2\}.$$ 

The following lemma is proved in \cite[Lemma 2.10]{C24}.
\begin{lemma}\label{lemma:corresp.}
The map $\widehat{\bullet}$ establishes an one-to-one correspondence from $\mathfrak{A}$ to $\mathfrak{A}_\zg$, as well as an one-to-one correspondence from $\mathfrak{Z}$ to $\mathfrak{Z}_\zg$. Furthermore, $\widehat{\za}$ is simple if and only if $\za$ is simple.
\end{lemma}

Let $\zG=\{\zg^1,\cdots,\zg^t\}$ be an admissible arc system on $(\cals,\calm,\zD)$. In particular, each $\zg^i$ is a simple $\bpoint$-arc and we can define the cutting surface $(\cals_{\zg^i},\calm_{\zg^i},\zD_{\zg^i})$ obtained by cutting $(\cals,\calm,\zD)$ along $\zg^i$. Inductively, denote by $(\cals_\zG,\calm_\zG,\zD_\zG)$ the surface obtained by cutting
$(\cals,\calm,\zD)$ along the arcs in $\zG$. Note that $(\cals_\zG,\calm_\zG,\zD_\zG)$ is independent of the order of the arcs that cut the surface. Then there are canonical zigzag arcs $\zg^i_1$ and $\zg^i_2$ on $(\cals_\zG,\calm_\zG,\zD_\zG)$ induced by $\zg^i$.

The following proposition provides us with an inductive way to solve problems, which will be used frequently in the rest of this section.
\begin{proposition}\label{prop:red2}
The set $\zG'=\{\zg^i_1,\zg^i_2, 1\leqslant i \leqslant t\}$ is an admissible arc system on $(\cals_\zG,\calm_\zG,\zD_\zG)$. Furthermore, there is a one-to-one correspondence between the set of (maximal resp.) admissible arc systems on $(\cals,\calm,\zD)$ which contain $\zG$ and the set of (maximal resp.) admissible arc systems on $(\cals_\zG,\calm_\zG,\zD_\zG)$ which contain $\zG'$.
\end{proposition}
\begin{proof}
We only need to prove the case where $\zG=\{\zg\}$. The general case can be proved inductively. Let $\calp$ be an admissible arc system that contains $\zg$. Denote by $\widehat{\calp}=\{\widehat{\za}  \text{~for~} \za \in \calp\}\cup\{\zg_1,\zg_2\}$. Then, according to Lemma \ref{lemma:corresp.}, each arc in $\widehat{\calp}$ is a simple zigzag arc on $(\cals_\zg,\calm_\zg,\zD^*_\zg)$. Furthermore, there is no interior intersection between the arcs in $\widehat{\calp}$, since $\calp$ is an arc system, and no new intersections appear when cutting the surface. Therefore, $\widehat{\calp}$ is an arc system. Now we prove that it is admissible. 

Assume that there is a (non-interior) oriented intersection between two $\bpoint$-arcs in $\widehat{\calp}$. Again, since when cutting the surface, no new intersection of arcs appears, such an oriented intersection is induced from an oriented intersection of arcs in $\calp$. We denote them by $\widehat{\mathfrak{p}}$ and $\mathfrak{p}$, respectively. We will prove that $\omega(\widehat{\mathfrak{p}})=\omega(\mathfrak{p})$. Then $\widehat{\calp}$ is admissible, since $\calp$ is.
Assume that $\mathfrak{p}$ arises from a  $\bpoint$-point $p$.
We have two cases depending on the position of $p$. 

Case I. The $\bpoint$-point $p$ is an endpoint of $\zg$. We have two sub-cases. When $p$ is a boundary point, it induces two $\bpoint$-points $p'$ and $p''$, where $p'$ is an endpoint of $\zg_1$ and $p''$ is an endpoint of $\zg_2$. Without loss of generality, we assume that $p'$ is that which gives rise to $\widehat{\mathfrak{p}}$, which is denoted by $\widehat{p}$. When $p$ is a puncture, it induces only one boundary $\bpoint$-point $\widehat{p}$, which is a common endpoint of $\zg_1$ and $\zg_2$. Assume that $\mathfrak{p}$ is from $\za$ to $\zb$, and $\widehat{\mathfrak{p}}$ is from $\widehat{\za}$ to $\widehat{\zb}$. For both cases, $\zg$ is not in between $\za$ and $\zb$, otherwise, there is no oriented intersection $\widehat{\mathfrak{p}}$ induced by $\mathfrak{p}$. Then the pictures in Figure \ref{fig:ind1} illustrate that $\omega(\widehat{\mathfrak{p}})=\omega(\mathfrak{p}).$ Note that these pictures are enough to explain things, where $\za$ and $\zb$ may coincide with $\zg$, and we treat the endpoint of a loop as two distinguished points. 
\begin{figure} 
\begin{tikzpicture}[>=stealth,scale=0.6]
					
					\draw [line width=3pt, gray!80] (-5.8,-0.1) to (5.8,-0.1);
					\draw [line width=1.5pt ] (-5.8,0) to (5.8,0);
					\draw [line width=1pt,red ] (-2,0) to (-4,2);
					\draw [line width=1pt,red ] (-4,2) to (-4,4);
					\draw [line width=1pt,red,dashed ] (-2,5.5) to (-4,4);
					\draw [line width=1pt,red ] (2,0) to (4,2);
					\draw [line width=1pt,red ] (4,2) to (4,4);
					\draw [line width=1pt,red, dashed ] (2,5.5) to (4,4);
					\draw [line width=1pt,red ] (2,5.5) to (-2,5.5);
                    
					\draw [line width=1pt ] (0,0) to (0,6.5);				
					\draw [bend right, line width=1pt] (0,0) to (-5.5,3);
					\draw [bend left, line width=1pt] (0,0) to (5.5,3);
					\draw [bend left, line width=1pt,->] (-0.85,1.2) to (0,1.5);
					
					\draw[thick, black ,fill=white] (0,0) circle (0.1);
					\draw[thick, red ,fill=red] (-2,0) circle (0.1);
					\draw[thick, red ,fill=red] (-4,2) circle (0.1);
					\draw[thick, red ,fill=red] (-4,4) circle (0.1);
					\draw[thick, red ,fill=red] (-2,5.5) circle (0.1);
					\draw[thick, red ,fill=red] (2,0) circle (0.1);
					\draw[thick, red ,fill=red] (4,2) circle (0.1);
					\draw[thick, red ,fill=red] (4,4) circle (0.1);
					\draw[thick, red ,fill=red] (2,5.5) circle (0.1);

					\draw (-3,3) node {$\alpha$};
					\draw (0.5,4) node {$\beta$};
					\draw (4.8,3.3) node {$\gamma$};
					\draw (-0.5,1.9) node {$\mathfrak{p}$};
					\draw (2,3.5) node[red] {$\bbp$};
					\draw (0,-0.6) node {$p$};
					
				\end{tikzpicture}
				\begin{tikzpicture}[>=stealth,scale=0.6]
					
					\draw [line width=3pt, gray!80] (-5.8,-0.1) to (5.8,-0.1);
					\draw [line width=1.5pt ] (-5.8,0) to (5.8,0);
					\draw [line width=1pt ] (0,0) to (0,6.5);
					\draw [line width=1pt,red ] (-2,0) to (-4,2);
					\draw [line width=1pt,red ] (-4,2) to (-4,4);
					\draw [line width=1pt,red,dashed ] (-2,5.5) to (-4,4);

					\draw [line width=1pt,red ] (4,4) to (4,0);
					\draw [line width=1pt,red, dashed ] (2,5.5) to (4,4);
					\draw [line width=1pt,red ] (2,5.5) to (-2,5.5);
					
					\draw [bend right, line width=1pt] (0,0) to (-5.5,3);
					\draw [bend left, line width=1pt] (0,0) to (5.5,3);
					\draw [bend left, line width=1pt,->] (-0.85,1.2) to (0,1.5);
					
					\draw[thick, black, fill=white ] (0,0) circle (0.1);
					\draw[thick, red ,fill=red] (-2,0) circle (0.1);
					\draw[thick, red ,fill=red] (-4,2) circle (0.1);
					\draw[thick, red ,fill=red] (-4,4) circle (0.1);
					\draw[thick, red ,fill=red] (-2,5.5) circle (0.1);
					\draw[thick, red ,fill=red] (4,0) circle (0.1);
					\draw[thick, red ,fill=red] (4,4) circle (0.1);
					\draw[thick, red ,fill=red] (2,5.5) circle (0.1);

					\draw (-3,3.2) node {$\widehat{\alpha}$};
					\draw (0.5,4) node {$\widehat{\beta}$};
					\draw (4.8,3.3) node {$\gamma_{1}$};
					\draw (-0.5,1.9) node {$\widehat{\mathfrak{p}}$};
					\draw (2,3.5) node[red] {$\mathrm{\bbp_{\gamma}}$};
					\draw (0,-0.6) node {$\widehat{p}$};
					\draw (4,-0.6) node[red] {$q'$};
					\draw (5,-0.6) node {$\gamma'$};
					
				\end{tikzpicture}		
				\begin{tikzpicture}[>=stealth,scale=0.5]

					\draw [line width=1pt ] (-5.8,0) to (5.8,0);
					\draw [line width=1pt ] (0,0) to (0,-5);
					\draw [line width=1pt,red,dashed ] (-4,3) to (4,3);
					\draw [line width=1pt,red ] (-4,3) to (-4,-1.5);
					\draw [line width=1pt,red ] (4,3) to (4,-1.5);
					\draw [line width=1pt,red,dashed ] (4,-1.5) to (2,-4);
					\draw [line width=1pt,red,dashed ] (-4,-1.5) to (-2,-4);
					\draw [line width=1pt,red ] (2,-4) to (-2,-4);

					\draw [bend left, line width=1pt] (-1,0) to (0,0.6)[bend left, line width=1pt,->](0,0.6)to(1,0);
					
					\draw[thick, black ,fill=white] (0,0) circle (0.1);
					\draw[thick, red ,fill=red] (-4,3) circle (0.1);
					\draw[thick, red ,fill=red] (-4,-1.5) circle (0.1);
					\draw[thick, red ,fill=red] (2,-4) circle (0.1);
					\draw[thick, red ,fill=red] (-2,-4) circle (0.1);
					\draw[thick, red ,fill=red] (4,3) circle (0.1);
					\draw[thick, red ,fill=red] (4,-1.5) circle (0.1);
	
					\draw (-2.5,-0.5) node {$\alpha$};
					\draw (2.5,-0.5) node {$\beta$};
					\draw (0.5,-3) node {$\gamma$};
					\draw (0,1) node {$\mathfrak{p}$};
					\draw (0.5,-0.5) node {$p$};
					\draw[red] (0,2) node {$\bbp$};
					\draw (-7,2) node {};
				\end{tikzpicture}
				\begin{tikzpicture}[>=stealth,scale=0.5]	
					\draw [line width=3pt, gray!80] (-7,-0.1) to (7,-0.1);
					\draw [line width=1.5pt ] (-7,0) to (7,0);
					\draw [line width=1pt,red ] (-2,0) to (-4.5,3);
					\draw [line width=1pt,red,dashed  ] (-4.5,3) to (-4.5,5.5);
					\draw [line width=1pt,red] (-4.5,5.5) to (-2,8);
					\draw [line width=1pt,red ] (2,0) to (4.5,3);
					\draw [line width=1pt,red,dashed  ] (4.5,3) to (4.5,5.5);
					\draw [line width=1pt,red] (4.5,5.5) to (2,8);
					\draw [line width=1pt,red,dashed] (2,8) to (-2,8);
					
					\draw [bend right, line width=1pt] (0,0) to (-5.5,2);
					\draw [bend left, line width=1pt] (0,0) to (5.5,2);
					\draw [bend right, line width=1pt] (0,0) to (-5.5,7.5);
					\draw [bend left, line width=1pt] (0,0) to (5.5,7.5);
					\draw [bend left, line width=1pt,->] (-0.8,3) to (0.8,3);
					
					\draw[thick, black ,fill=white] (0,0) circle (0.1);
					\draw[thick, red ,fill=red] (-2,0) circle (0.1);
					\draw[thick, red ,fill=red] (-4.5,3) circle (0.1);
					\draw[thick, red ,fill=red] (-4.5,5.5) circle (0.1);
					\draw[thick, red ,fill=red] (-2,8) circle (0.1);
					\draw[thick, red ,fill=red] (2,0) circle (0.15);
					\draw[thick, red ,fill=red] (4.5,3) circle (0.1);
					\draw[thick, red ,fill=red] (4.5,5.5) circle (0.1);
					\draw[thick, red ,fill=red] (2,8) circle (0.1);
					
					\draw (-3,5.3) node {$\widehat{\alpha}$};
					\draw (3,5.3) node {$\widehat{\beta}$};
					\draw (4.8,1.5) node {$\gamma_{1}$};
					\draw (-4.8,1.5) node {$\gamma_{2}$};
					\draw (0,3.8) node {$\widehat{\mathfrak{p}}$};
					\draw (0,6.5) node[red] {$\mathrm{\bbp_{\gamma}}$};
					\draw (0,-0.7) node {$\widehat{p}$};
					\draw (2,-0.6) node[red] {$q'$};
					\draw (-2,-0.6) node[red] {$q''$};
					\draw (5,-0.6) node {$\gamma'$};
					\draw (-5,-0.6) node {$\gamma''$};	
                    \draw (-8,2) node {};
				\end{tikzpicture}
	\caption{Three zigzag arcs $\za, \zb$ and $\zg$ share a common endpoint $p$. An oriented intersection $\widehat{\mathfrak{p}} $ from $\widehat{\za}$ to $\widehat{\zb}$ is induced by an oriented intersection $\mathfrak{p} $ from $\za$ to $\zb$. For both cases, $\omega(\mathfrak{p})=\omega(\widehat{\mathfrak{p}})$. The two pictures on the left demonstrate the polygons $\bbp$ on the original surface, and the two pictures on the right demonstrate the induced polygons $\bbp_\zg$ on the cutting surface.} 
	\label{fig:ind1} 
\end{figure}
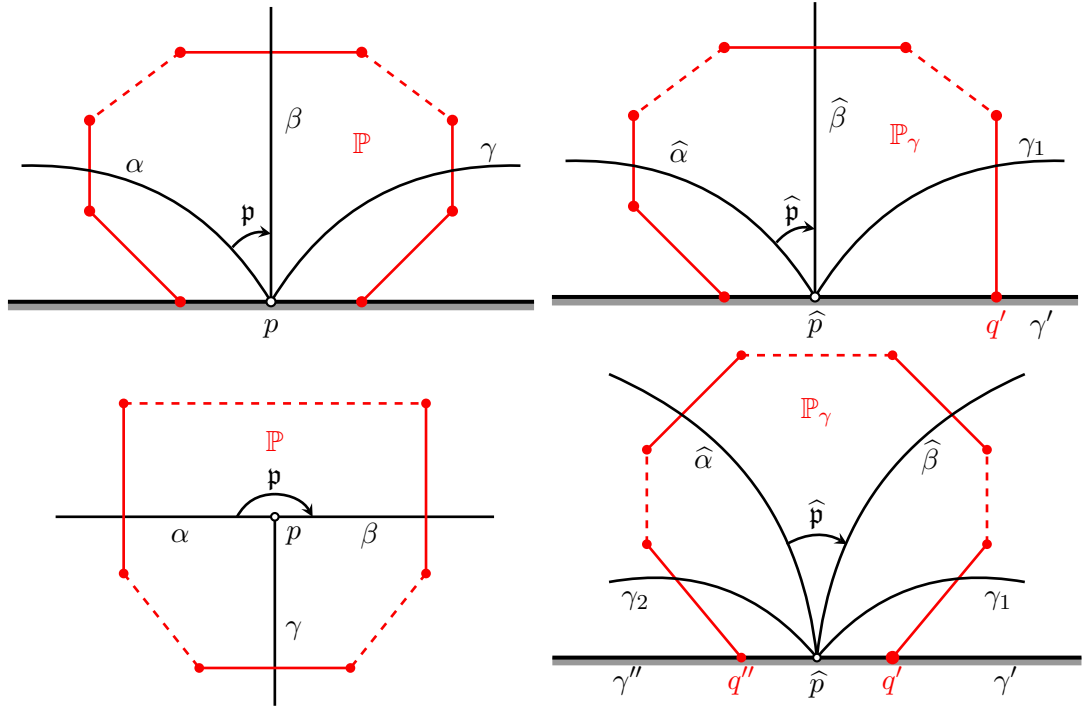

Case II. The $\bpoint$-point $p$ differs from the endpoints of $\zg$. In this case, $p$ induces a unique $\bpoint$-point $\widehat{p}$ on the cutting surface. Furthermore, the polygon $\bbp_\zg$ that contains $\widehat{p}$ is induced from a polygon $\bbp$ that contains $p$, see \cite[Corollary 2.7]{C24}. These polygons are homotopic and there is a one-to-one correspondence between the minimal oriented intersections in $\bbp_\zg$ and $\bbp$. Up to this homotopy, the local configuration of the arcs associated with $\mathfrak{p}$ does not change when cutting the surface. Therefore, we have $\omega(\widehat{\mathfrak{p}})=\omega(\mathfrak{p})$, noticing that the weight of an oriented intersection is the number of minimal oriented intersections of the simple coordinate in between the intersection. 

Conversely, let $\widehat{\calp}$ be an admissible arc system on $(\cals_\zg,\calm_\zg,\zD^*_\zg)$ which contains $\zg_1$ and $\zg_2$. The converse argument as above derives that $\calp$ is an admissible arc system on $(\cals,\calm,\zD^*)$. The only exception that needs to be addressed is the situation in which some oriented intersection $\mathfrak{p}$ in $\calp$ may disappear after cutting the surface. This happens when $p$ is an endpoint of $\zg$ and $\mathfrak{p}: \za\longrightarrow \zb$ factor through $\zg$. More precisely, $\mathfrak{p}=\mathfrak{p}_1\mathfrak{p}_2$ for $\mathfrak{p}_1:\za\longrightarrow \zg$ and $\mathfrak{p}_2:\zg\longrightarrow \zb$, see the left two pictures in Figure \ref{fig:ind2}. In this case, the weight $\omega(\mathfrak{p})$ is different from one, since $\omega(\mathfrak{p})=\omega(\mathfrak{p}_1)+\omega(\mathfrak{p}_2)$, $\omega(\mathfrak{p}_i)\geqslant 0$ and $\omega(\mathfrak{p}_i)=\omega(\widehat{\mathfrak{p}}_i)\neq 1$ for $i=1,2.$  

The statement for the maximal admissible arc systems directly follows from the statement for the admissible arc systems.
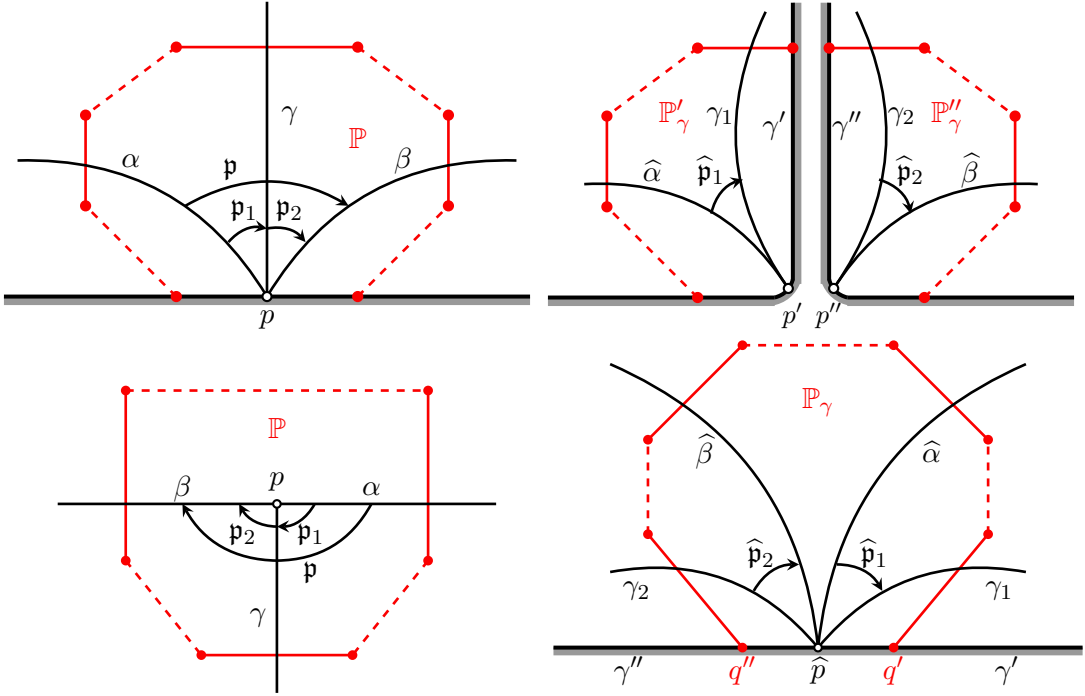
\begin{figure} 
			\begin{tikzpicture}[>=stealth,scale=0.6]
				
				\draw [line width=3pt, gray!80] (-5.8,-0.1) to (5.8,-0.1);
				\draw [line width=1.5pt ] (-5.8,0) to (5.8,0);
				\draw [line width=1pt,red, dashed ] (-2,0) to (-4,2);
				\draw [line width=1pt,red ] (-4,2) to (-4,4);
				\draw [line width=1pt,red,dashed ] (-2,5.5) to (-4,4);
				\draw [line width=1pt,red, dashed ] (2,0) to (4,2);
				\draw [line width=1pt,red ] (4,2) to (4,4);
				\draw [line width=1pt,red, dashed ] (2,5.5) to (4,4);
				\draw [line width=1pt,red ] (2,5.5) to (-2,5.5);
                
				\draw [line width=1pt ] (0,0) to (0,6.5);	
				\draw [bend right, line width=1pt] (0,0) to (-5.5,3);
				\draw [bend left, line width=1pt] (0,0) to (5.5,3);
				\draw [bend left, line width=1pt,->] (-0.85,1.2) to (0,1.5);
				\draw [bend left, line width=1pt,->] (0,1.5) to (0.85,1.2);
				\draw [bend left, line width=1pt,->] (-1.8,2) to (1.8,2);
				
				\draw[thick, black ,fill=white] (0,0) circle (0.1);
				\draw[thick, red ,fill=red] (-2,0) circle (0.1);
				\draw[thick, red ,fill=red] (-4,2) circle (0.1);
				\draw[thick, red ,fill=red] (-4,4) circle (0.1);
				\draw[thick, red ,fill=red] (-2,5.5) circle (0.1);
				\draw[thick, red ,fill=red] (2,0) circle (0.1);
				\draw[thick, red ,fill=red] (4,2) circle (0.1);
				\draw[thick, red ,fill=red] (4,4) circle (0.1);
				\draw[thick, red ,fill=red] (2,5.5) circle (0.1);

				\draw (-3,3) node {$\alpha$};
				\draw (0.5,4) node {$\gamma$};
				\draw (3,3) node {$\beta$};
				\draw (-0.8,2.8) node {$\mathfrak{p}$};
				\draw (-0.5,1.9) node {$\mathfrak{p}_{1}$};
				\draw (0.5,1.9) node {$\mathfrak{p}_{2}$};
				\draw (2,3.5) node[red] {$\bbp$};
				\draw (0,-0.5) node {$p$};
				
			\end{tikzpicture}
			\begin{tikzpicture}[>=stealth,scale=0.6]
				\draw [line width=3pt, gray!80] (-5.8,-0.1) to (-0.8,-0.1);
				\draw [bend right, line width=3pt, gray!80] (-0.9,-0.1) to (-0.3,0.4);
				\draw [line width=3pt, gray!80 ](-0.3,0.3) to (-0.3,6.5);
				\draw [line width=1.5pt ] (-5.8,0) to (-0.8,0);
				\draw[bend right, line width=1.5pt] (-0.8,0) to (-0.4,0.4);
				\draw [line width=1.5pt ](-0.4,0.38) to (-0.4,6.5);
				
				\draw [line width=3pt, gray!80] (5.8,-0.1) to (0.8,-0.1);
				\draw [bend left, line width=3pt, gray!80] (0.9,-0.1) to (0.3,0.4);
				\draw [line width=3pt, gray!80 ](0.3,0.3) to (0.3,6.5);
				\draw [line width=1.5pt ] (5.8,0) to (0.8,0);
				\draw[bend left, line width=1.5pt] (0.8,0) to (0.4,0.4);
				\draw [line width=1.5pt ](0.4,0.38) to (0.4,6.5);

				\draw [line width=1pt,red, dashed ] (-2.5,0) to (-4.5,2);
				\draw [line width=1pt,red ] (-4.5,2) to (-4.5,4);
				\draw [line width=1pt,red,dashed ] (-2.5,5.5) to (-4.5,4);
				\draw [line width=1pt,red, dashed ] (2.5,0) to (4.5,2);
				\draw [line width=1pt,red ] (4.5,2) to (4.5,4);
				\draw [line width=1pt,red, dashed ] (2.5,5.5) to (4.5,4);
				\draw [line width=1pt,red ] (2.5,5.5) to (0.4,5.5);
				\draw [line width=1pt,red ] (-2.5,5.5) to (-0.4,5.5);

				\draw [bend left, line width=1pt] (0.5,0.2) to (5,2.5);
				\draw [bend right, line width=1pt] (0.5,0.2) to (1,6.3);
				\draw [bend right, line width=1pt] (-0.5,0.2) to (-5,2.5);
				\draw [bend left, line width=1pt] (-0.5,0.2) to (-1,6.3);                
			    \draw [bend left, line width=1pt, ->] (-2.2,1.85) to (-1.5,2.6);
				\draw [bend left, line width=1pt, ->](1.5,2.6)  to(2.2,1.85) ;
				
				\draw[thick, black ,fill=white] (0.5,0.2) circle (0.1);
				\draw[thick, black ,fill=white] (-0.5,0.2) circle (0.1);
				\draw[thick, red ,fill=red] (-2.5,0) circle (0.1);
				\draw[thick, red ,fill=red] (-4.5,2) circle (0.1);
				\draw[thick, red ,fill=red] (-4.5,4) circle (0.1);
				\draw[thick, red ,fill=red] (-2.5,5.5) circle (0.1);
				\draw[thick, red ,fill=red] (2.5,0) circle (0.1);
				\draw[thick, red ,fill=red] (4.5,2) circle (0.1);
				\draw[thick, red ,fill=red] (4.5,4) circle (0.1);
				\draw[thick, red ,fill=red] (2.5,5.5) circle (0.1);
				\draw[thick, red ,fill=red] (0.4,5.5) circle (0.1);
				\draw[thick, red ,fill=red] (-0.4,5.5) circle (0.1);

				\draw (-3.5,2.8) node {$\widehat{\alpha}$};
				\draw (3.5,2.8) node {$\widehat{\beta}$};
				\draw (-2,4) node {$\gamma_{1}$};
				\draw[red] (-3,4) node {$\bbp'_\zg$};
				\draw[red] (3,4) node {$\bbp''_\zg$};
				\draw (2,4) node {$\gamma_{2}$};
				\draw (-0.8,3.8) node {$\gamma'$};
				\draw (0.8,3.8) node {$\gamma''$};
				\draw (-2.2,2.8) node {$\widehat{\mathfrak{p}}_{1}$};
				\draw (2.2,2.8) node {$\widehat{\mathfrak{p}}_{2}$};
				\draw (-0.4,-0.4) node {\small$p'$};
				\draw (0.4,-0.4) node {\small$p''$};

			\end{tikzpicture}	
			
			\begin{tikzpicture}[>=stealth,scale=0.5]

				\draw [line width=1pt,red,dashed ] (-4,3) to (4,3);
				\draw [line width=1pt,red ] (-4,3) to (-4,-1.5);
				\draw [line width=1pt,red ] (4,3) to (4,-1.5);
				\draw [line width=1pt,red,dashed ] (4,-1.5) to (2,-4);
				\draw [line width=1pt,red,dashed ] (-4,-1.5) to (-2,-4);
				\draw [line width=1pt,red ] (2,-4) to (-2,-4);
				
				\draw [line width=1pt ] (-5.8,0) to (5.8,0);
				\draw [line width=1pt ] (0,0) to (0,-5);
				\draw [bend left, line width=1pt,->] (1,0) to (0,-0.6);
				\draw [bend left, line width=1pt, ->](0,-0.6)to(-1,0);
				\draw [bend left, line width=1pt] (2.5,0) to (0,-1.5);
				\draw [bend left, line width=1pt, ->](0,-1.5)to(-2.5,0);
				
				\draw[thick, black ,fill=white] (0,0) circle (0.1);
				\draw[thick, red ,fill=red] (-4,3) circle (0.1);
				\draw[thick, red ,fill=red] (-4,-1.5) circle (0.1);
				\draw[thick, red ,fill=red] (2,-4) circle (0.1);
				\draw[thick, red ,fill=red] (-2,-4) circle (0.1);
				\draw[thick, red ,fill=red] (4,3) circle (0.1);
				\draw[thick, red ,fill=red] (4,-1.5) circle (0.1);
				\draw[red] (0,2) node {$\bbp$};			
				\draw (2.5,0.4) node {$\alpha$};
				\draw (-2.5,0.4) node {$\beta$};
				\draw (-0.5,-3) node {$\gamma$};
				\draw (-0.9,-0.8) node {$\mathfrak{p}_{2}$};
				\draw (0.9,-0.8) node {$\mathfrak{p}_{1}$};
				\draw (0,0.6) node {$p$};
				\draw (0.9,-1.8) node {$\mathfrak{p}$};
				\draw (-7.3,-1.8) node {};
				\draw (6.7,-1.8) node {};
			\end{tikzpicture}
			\begin{tikzpicture}[>=stealth,scale=0.5]
				
				\draw [line width=3pt, gray!80] (-7,-0.1) to (7,-0.1);
				\draw [line width=1.5pt ] (-7,0) to (7,0);
				\draw [line width=1pt,red ] (-2,0) to (-4.5,3);
				\draw [line width=1pt,red,dashed  ] (-4.5,3) to (-4.5,5.5);
				\draw [line width=1pt,red] (-4.5,5.5) to (-2,8);
				\draw [line width=1pt,red ] (2,0) to (4.5,3);
				\draw [line width=1pt,red,dashed  ] (4.5,3) to (4.5,5.5);
				\draw [line width=1pt,red] (4.5,5.5) to (2,8);
				\draw [line width=1pt,red,dashed ] (2,8) to (-2,8);
				
				\draw [bend right, line width=1pt] (0,0) to (-5.5,2);
				\draw [bend left, line width=1pt] (0,0) to (5.5,2);
				\draw [bend right, line width=1pt] (0,0) to (-5.5,7.5);
				\draw [bend left, line width=1pt] (0,0) to (5.5,7.5);
				\draw [bend left, line width=1pt,->] (-1.7,1.5) to (-0.5,2.2);
				\draw [bend left, line width=1pt,->] (0.5,2.2) to (1.7,1.5);
				
				\draw[thick, black ,fill=white] (0,0) circle (0.1);
				\draw[thick, red ,fill=red] (-2,0) circle (0.1);
				\draw[thick, red ,fill=red] (-4.5,3) circle (0.1);
				\draw[thick, red ,fill=red] (-4.5,5.5) circle (0.1);
				\draw[thick, red ,fill=red] (-2,8) circle (0.1);
				\draw[thick, red ,fill=red] (2,0) circle (0.1);
				\draw[thick, red ,fill=red] (4.5,3) circle (0.1);
				\draw[thick, red ,fill=red] (4.5,5.5) circle (0.1);
				\draw[thick, red ,fill=red] (2,8) circle (0.1);
				
				\draw (3,5.2) node {$\widehat{\alpha}$};
				\draw (-3,5.2) node {$\widehat{\beta}$};
				\draw (4.8,1.5) node {$\gamma_{1}$};
				\draw (-4.8,1.5) node {$\gamma_{2}$};
				\draw (1.5,2.5) node {$\widehat{\mathfrak{p}}_{1}$};
				\draw (-1.5,2.5) node {$\widehat{\mathfrak{p}}_{2}$};
				\draw (0,6.5) node[red] {$\mathrm{\bbp_{\gamma}}$};
				\draw (0,-0.6) node {$\widehat{p}$};
				\draw (2,-0.6) node[red] {$q'$};
				\draw (-2,-0.6) node[red] {$q''$};
				\draw (5,-0.6) node {$\gamma'$};
				\draw (-5,-0.6) node {$\gamma''$};
			\end{tikzpicture}
	\caption{An oriented intersection $\mathfrak{p}$ from $\za$ to $\zb$ disappears after cutting the surface. This happens when $p$ is an endpoint of $\zg$ and $\mathfrak{p}: \za\longrightarrow \zb$ factor through $\zg$. Then $\omega(\widehat{\mathfrak{p}}_i)\neq 1$ for $i=1,2$ implies $\omega(\mathfrak{p}_i)\neq 1$, and thus $\omega(\mathfrak{p})=\omega(\mathfrak{p}_1)+\omega(\mathfrak{p}_2)\neq 1$.} 
	\label{fig:ind2} 
\end{figure}
\end{proof}

The following lemma is the first application of the above proposition.
\begin{lemma}\label{lem:completion}
Any admissible arc system $\calp$ on a coordinated-marked surface $(\cals,\calm,\zD^*)$ can be completed as an admissible partial triangulation.
\end{lemma}
\begin{proof}
We may always assume that each arc $\zg$ in $\calp$ belongs to an external triangle with a unique $\rpoint$-point, that is, $\zg$ is the arc in the triangle of type $\textrm{\textbf{F1}}$ in Figure \ref{fig:ply}. Otherwise, we consider the cutting surface $(\cals_\zg,\calm_\zg,\zD^*_\zg)$ and complete the arc system $\widehat{\calp}=\{\widehat{\za}, \za \in \calp\}\cup \{\zg_1,\zg_2\}$ as an admissible partial triangulation, and then lift it to the original surface by Proposition \ref{prop:red2}. 

Now we show that under above assumption, $\calp$ can be completed as an admissible partial triangulation.
Denote by $\calp'$ the union of $\calp$ and the projective dissection $\zD_P$.
Denote by $\za_1$ and $\za_2$ the arcs in $\calp$ and $\zD_P$, respectively. Since $\za_1$ is isotopic to a boundary segment that contains only one $\rpoint$-point, there is no interior intersection between $\za_1$ and $\za_2$. Furthermore, since the projective dissection is a partial triangulation, $\calp'$ is also a partial triangulation. But $\calp'$ may not be admissible, that is, there may exist an oriented intersection of two arcs whose weight equals one. In this case, one of the two arcs belongs to $\calp$, and the other belongs to $\zD_P$, since both $\calp$ and $\zD_P$ are themselves admissible. 
Then by using an argument similar to the paragraph above Proposition-Definition \ref{prop-def:flip}, we get an admissible partial triangulation $\calp'_{\textbf{adm}}$, where when we replace the arc by the flip of the two arcs, we replace the one in $\zD_P$, such that the final admissible partial triangulation $\calp'_{\textbf{adm}}$ still contains the original arc system $\calp$.
\end{proof}

\subsection{Maximal rigid modules as admissible $5$-partial triangulations}\label{section:max-rig-mod3}

In this subsection, we characterize the maximal admissible arc systems and show that they are exactly the $5$-partial triangulations. We start with some basic lemmas. 
\begin{lemma}\label{lem:disk1}
Assume that $(\cals,\calm)$ is a disk with $m\geqslant 5$ $\bpoint/\rpoint$-marked points on the boundary. 
Let $\zG=\{\zg^1,\cdots,\zg^{m-1}\}$ be an arc system such that each $\zg^i$ forms an external triangle which contains one $\rpoint$-point. If $\zG$ is admissible with respect to a simple coordinate $\zD^*$, then $\zG$ is not a maximal admissible arc system concerning $\zD^*$.
\end{lemma}
\begin{proof}
We clockwise labeled the arcs $\zg^i$ on the disk, see the picture in Figure \ref{fig:5}. Let $\za_1$ be the (unique) $\bpoint$-arc that forms a $m$ -gon with the arcs $\zg^1,\cdots,\zg^{m-1}$. Then $\za_1$ forms an external triangle which contains one $\rpoint$-point. Therefore, $\za_1$ factors through a fan of $\rpoint$-arcs in $\zD^*$, and it is a zigzag arc for $\zD^*$. Denote by $p_1$ and $p_2$ the endpoints of $\za_1$, and by $\bbp_1$ and $\bbp_2$ the polygons of $\zD^*$ which contain $p_1$ and $p_2$ respectively. 

\begin{figure}
		{
\begin{tikzpicture}[>=stealth,scale=.7]
				\draw[line width=1.5pt]  (0,0) ellipse[x radius=7, y radius=4];
	
				\draw[cyan,line width=1pt,red] (-3,3.6) to[out=-60,in=60] (-3.8,-3.35) ;
				\draw[cyan,line width=1pt,red,dashed ] (-3,3.6) to[out=-30,in=100] (-1,0) ;
				\draw[cyan,line width=1pt,red,dashed ] (3,3.6) to[out=-150,in=80] (1,0) ;
				\draw[line width=1pt,red ] (-1,0) to (0,-4) to (1,0) ;
				\draw[cyan,line width=1pt,red ] (3,3.6) to[out=-100,in=130] (3.8,-3.35) ;
			     
				\draw[cyan,line width=1pt,blue ] (-2,-3.85) to[out=30,in=150] (2,-3.85) ;
				\draw[cyan,line width=1pt,blue ] (-5.5,-2.5) to[out=20,in=140] (2,-3.85) ;
				\draw[cyan,line width=1pt,blue ] (-5.5,-2.5) to[out=30,in=150] (5.5,-2.5) ;
				\draw[cyan,line width=1pt,blue ] (-2,-3.85) to[out=40,in=160] (5.5,-2.5) ;

				\draw[cyan,line width=1pt,black ] (-6.75,1) to[out=-30,in=30] (-6.75,-1) ;
				\draw[cyan,line width=1pt,black ] (-6.75,-1) to[out=0,in=80] (-5.5,-2.5) ;
				\draw[cyan,line width=1pt,black ] (-5.5,-2.5) to[out=10,in=110] (-2,-3.85) ;
				\draw[cyan,line width=1pt,black ] (6.75,1) to[out=-150,in=150] (6.75,-1) ;
				\draw[cyan,line width=1pt,black ] (6.75,-1) to[out=180,in=100] (5.5,-2.5) ;
				\draw[cyan,line width=1pt,black ] (5.5,-2.5) to[out=170,in=70] (2,-3.85) ;

				\draw[thick,fill=white] (-5.5,-2.5) circle (0.1);
				\draw[thick,fill=white] (5.5,-2.5)  circle (0.1);
				\draw[thick,fill=white] (-6.75,1) circle (0.1);
				\draw[thick,fill=white] (6.75,1)  circle (0.1);
				\draw[thick,fill=white] (-6.75,-1) circle (0.1);
				\draw[thick,fill=white] (6.75,-1)  circle (0.1);
				\draw[thick,fill=white] (-2,-3.85) circle (0.1);
				\draw[thick,fill=white] (2,-3.85)  circle (0.1);

				\draw[thick,red, fill=red] (0,-4) circle (0.1);
				\draw[thick,red,fill=red] (-3,3.6) circle (0.1);	
				\draw[thick,red,fill=red] (3,3.6) circle (0.1);	
				\draw[thick,red,fill=red] (-3.8,-3.35) circle (0.1);	
				\draw[thick,red,fill=red] (3.8,-3.35) circle (0.1);	
                \draw[red] (0,0) node {$\cdots$};		
				\draw (-1.8,0) node[red] {$\bbp_1$};
				\draw (1.8,0) node[red] {$\bbp_2$};
				\draw (-1.9,-4.3) node {$p_1$};
				\draw (1.9,-4.3) node {$p_2$};
				\draw (-2.4,-2.6) node {$\gamma^{1}$};
				\draw (-5.4,-1.4) node {$\gamma^{2}$};
				\draw (-5.8,0) node {$\gamma^{3}$};
				\draw (2.2,-2.7) node {$\gamma^{m-1}$};
				\draw (5.3,-1.3) node {$\gamma^{m-2}$};
				\draw (5.5,0) node {$\gamma^{m-3}$};
	            \draw (0.6,-3.6) node[blue]{$\alpha_{1}$};
	            \draw (-1.25,-2.9) node[blue]{$\alpha_{2}$};
	            \draw (1.2,-2.9) node[blue]{$\alpha_{3}$};
	            \draw (0,-1.2) node[blue]{$\alpha_{4}$};
	            
		\end{tikzpicture}
	}
	\caption{For an admissible arc system $\zG=\{\zg^1,\cdots,\zg^{m-1}\}$ on a disk with $m$ $\rpoint$-points, where each $\zg^i$ forms an external triangle, it is always possible to add an extra arc $\za_i$ to get a larger admissible arc system. The choice of $\za_i$ depends on whether the polygons $\bbp_1$ and $\bbp_2$ are triangles or not.}\label{fig:5}
	\end{figure}
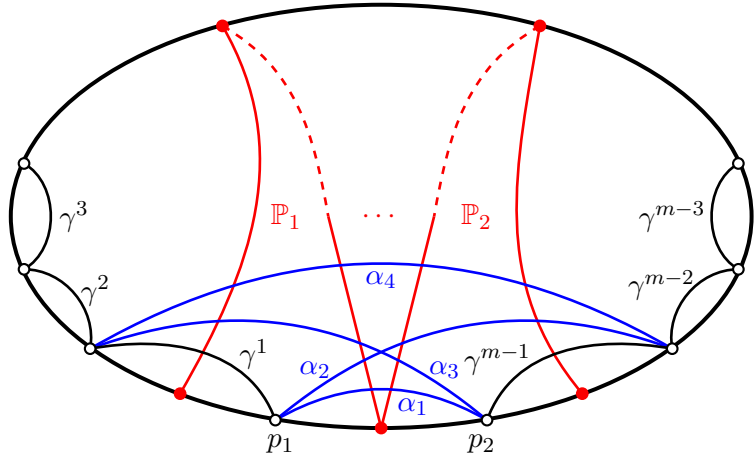

We have several cases depending on whether $\bbp_1$ and $\bbp_2$ are triangles.
If both $\bbp_1$ and $\bbp_2$ are not triangles, then the weights of the oriented intersections between $\za_1$ and the arcs $\zg^1$ and $\zg^{m-1}$ arising from $p_1$ and $p_2$ are different from one. On the other hand, there are no other intersections between $\za_1$ and the arcs $\zg^i, 1\leqslant i \leqslant m-1$. Therefore, $\zG\cup\{\za_1\}$ is still an admissible arc system and thus $\zG$ is not maximal. If $\bbp_1$ is not a triangle but $\bbp_2$ is a triangle, then we consider an arc $\za_2$ which is the flip of $\za_1$ and $\zg^{m-1}$ at $p_2$. Then, according to Lemma \ref{lem:flip}, $\za_2$ is a zigzag arc, and $\zG\cup\{\za_2\}$ is still an admissible arc system. Thus, $\zG$ is not maximal. 
For the other two cases, the proof is similar. 

Note that when both $\bbp_1$ and $\bbp_2$ are triangles, we consider the arc $\za_4$ in the picture, which is obtained by smoothing $\za_1$ with $\zg^1$ and $\zg^{m-1}$ at $p_1$ and $p_2$ respectively. Since $m\geqslant 5$, $\za_4$ is different from any arc in $\zG$, therefore $\zG\cup\{\za_4\}\neq \zG$, and $\zG$ is not maximal. 
\end{proof}

\begin{proposition}\label{lem:ply2}
Let $(\cals,\calm,\zD^*)$ be a coordinated-marked surface. Then each $m$-gon with $m \geqslant 6$ in a maximal admissible arc system is an internal polygon, that is, it does not contain a $\rpoint$-point. 
\end{proposition}
\begin{proof}
Suppose that there is an external $m$-gon $\bbp$ in a maximal admissible arc system $\calp$ with $m\geqslant 6$. Then there is a $\rpoint$-point in $\bbp$. We denote $\zG=\{\zg^i, 1\leqslant i \leqslant m-2\}$ the edges of $\bbp$ which are the arcs in $\calp$, and by $\mathsf{b}$ and $\mathsf{b}'$ the extra edges of $\bbp$ which are boundary segments. Then on the cutting surface $(\cals_\zG,\calm_\zG,\zD^*_\zG)$, there is a connected component which contains $\mathsf{b}$ and $\mathsf{b}'$. Denote this component by $(\cals',\calm',\zD'^*)$, where $\cals'$ is a disk and there are $(m-1)$ $\rpoint$-points in $\calm'$.
By Proposition \ref{prop:red2}, $\zG$ induces a maximal admissible arc system on $(\cals',\calm',\zD'^*)$, which we denote by $\zG'=\{\zg^i_1, 1\leqslant i \leqslant m-2\}$. Note that each $\zg^i_1$ forms an external triangle that contains one $\rpoint$ point. On the other hand, the number of $\rpoint$-points on the boundary of the disk $\cals'$ is equal to $m-1$, which is not smaller than $5$. Therefore, the conditions in Lemma \ref{lem:disk1} are met, and there is an extra arc $\zg'$ such that $\zG'\cup \{\zg'\}$ is an admissible arc system on $(\cals',\calm',\zD'^*)$. Denote by $\zg$ the lift of $\zg'$ on $(\cals,\calm,\zD^*)$. Then, by Proposition \ref{prop:red2}, $\calp\cup \{\zg\}$ is an admissible arc system on $(\cals,\calm,\zD^*)$, which contradicts the assumption that $\calp$ is maximal.
\end{proof}

\begin{lemma}\label{lem:disk2}
Assume that $(\cals,\calm)$ is a disk with $m\geqslant 6$ $\bpoint/\rpoint$-marked points on the boundary. 
Let $\zG=\{\zg^1,\cdots,\zg^{m}\}$ be a set of $\bpoint$-arcs such that each $\zg^i$ forms an external triangle that contains one $\rpoint$-point. If $\zG$ is an admissible arc system for a simple coordinate $\zD^*$, then $\zG$ is not a maximal admissible arc system for $\zD^*$.
\end{lemma}
\begin{proof}

We have two cases. For each case, we will find an extra $\bpoint$-arc $\za$ such that $\zG\cup\{\za\}$ is still an admissible arc system. 

Case I. Each arc in $\zD^*$ forms a triangle with boundary segments. Note that there are $(m-1)$ $\rpoint$-arcs in $\zD^*$. These arcs form a $(m+1)$-gon with two boundary segments. Denote by $p$ the $\bpoint$-point in the polygon. We choose $\za$ to be an $\bpoint$-arc which starts from $p$ and divides the polygon formed by the arcs in $\zG$ into two polygons each of which has at least four edges, noticing that since $m\geqslant 6$, we can always find such an arc, see the picture on the left in Figure \ref{fig:6base1}. Then $\zG\cup\{\za\}$ is still an admissible arc system concerning $\zD^*$, and thus $\zG$ is not maximal.

\begin{figure}
	
\begin{tikzpicture}[>=stealth,scale=0.4]
\draw[line width=1.5pt,fill=white] (0,0) circle (7cm);
				
\draw [blue, line width=1pt] (0,7) to (0,-7);
\draw [bend right, line width=1pt] (0,7) to (4.3,5.5);
\draw [bend left, line width=1pt] (0,7) to (-4.3,5.5);
\draw [bend left, line width=1pt] (0,-7) to (4.3,-5.5);
\draw [bend right, line width=1pt] (0,-7) to (-4.3,-5.5);
\draw [bend right, line width=1pt, dotted] (-4.3,-5.5) to (-4.3,5.5);
\draw [bend right, line width=1pt,dotted] (4.3,5.5) to (4.3,-5.5);
\draw [bend right,red, line width=1pt] (-2.5,6.55) to (2.5,6.55);
\draw [bend left, dashed,red ,  line width=1pt] (-2.5,6.55) to (-7,0);
\draw [bend right, dashed,red ,  line width=1pt] (2.5,6.55) to (7,0);
\draw [bend right, dashed,red ,  line width=1pt] (-2.5,-6.55) to (-7,0);
\draw [bend left, dashed,red ,  line width=1pt] (2.5,-6.55) to (7,0);
				
\draw[thick,fill=white] (0,7) circle (0.15);
\draw[thick,fill=white] (0,-7)  circle (0.15);
\draw[thick,fill=white] (4.3,5.5) circle (0.15);
\draw[thick,fill=white] (-4.3,5.5)  circle (0.15);
\draw[thick,fill=white] (4.3,-5.5) circle (0.15);
\draw[thick,fill=white] (-4.3,-5.5)  circle (0.15);
\draw[thick,red, fill=red] (-7,0)  circle (0.15);
\draw[thick,red, fill=red] (7,0)  circle (0.15);
\draw[thick,red, fill=red] (2.5,6.55)  circle (0.15);
\draw[thick,red, fill=red] (-2.5,-6.55)  circle (0.15);
\draw[thick,red, fill=red] (-2.5,6.55)   circle (0.15);
\draw[thick,red, fill=red] (2.5,-6.55)   circle (0.15);
					
\draw (0.5,0) node[blue] {$\za$};
\draw (0,-7.8) node{$p$};
\end{tikzpicture}
\begin{tikzpicture}[>=stealth,scale=0.4]
\draw[line width=1.5pt,fill=white] (0,0) circle (7cm);
				\draw [blue, line width=1pt] (0,7) to (0,-7);
				\draw [bend right, line width=1pt] (0,7) to (4.3,5.5);
				\draw [bend left, line width=1pt] (0,7) to (-4.3,5.5);
				\draw [bend left, line width=1pt] (0,-7) to (4.3,-5.5);
				\draw [bend right, line width=1pt] (0,-7) to (-4.3,-5.5);
				\draw [red ,  line width=1pt] (3.5,1)  to (-2.45,6.45) to (2.45,-6.45) to (-3.5,-1);
				\draw [red ,dotted,   line width=1pt] (-1,3.5)  to (-0.4,4.2) ;
				\draw [red ,dotted,   line width=1pt] (1,-3.5)  to (0.4,-4.2) ;

			   \draw [bend left, line width=.7pt,->] (0,5.5) to (-0.7,6.2);
				\draw [bend left, line width=.7pt,->] (0.7,6.2) to (0,5.5);
				\draw [bend left, line width=.7pt,->] (-0.7,-6.2) to (0,-5.5);
				\draw [bend left, line width=.7pt,->] (0,-5.5) to (0.7,-6.2);

				\draw[thick,fill=white] (0,7) circle (0.15);
				\draw[thick,fill=white] (0,-7)  circle (0.15);
				\draw[thick,fill=white] (4.3,5.5) circle (0.15);
				\draw[thick,fill=white] (-4.3,5.5)  circle (0.15);
				\draw[thick,fill=white] (4.3,-5.5) circle (0.15);
				\draw[thick,fill=white] (-4.3,-5.5)  circle (0.15);
				\draw[thick,red, fill=red] (2.5,6.55)  circle (0.15);
				\draw[thick,red, fill=red] (-2.5,-6.55)  circle (0.15);
				\draw[thick,red, fill=red] (-2.5,6.55)   circle (0.15);
				\draw[thick,red, fill=red] (2.5,-6.55)   circle (0.15);
\draw (0,7.8) node{$p_{1}$};
\draw (0,-7.8) node{$p_{2}$};
				\draw (-3.5,4.8) node{$\gamma^{1}$};
				\draw (3.5,4.8) node{$\gamma^{2}$};
\draw (-3.5,-4.8) node{$\gamma^{3}$};
\draw (3.5,-4.8) node{$\gamma^{4}$};
\draw (1,5.2) node{$\mathfrak{p}_{2}$};
\draw (-1,5.2) node{$\mathfrak{p}_{1}$};
\draw (1,-5.2) node{$\mathfrak{p}_{4}$};
\draw (-1,-5.2) node{$\mathfrak{p}_{3}$};
\draw (-1.5,1.5) node[red]{$\ell^{*}$};
\draw (1,1.5) node[blue]{$\za$};
\draw (3.5,2) node[red]{$\ell^{*}_{1}$};
\draw (-2,-1.5) node[red]{$\ell^{*}_{2}$};
\end{tikzpicture}
\caption{For an admissible arc system $\zG=\{\zg^1,\cdots,\zg^{m}\}$ on a disk which has $m$ $\rpoint$-points with $m\geqslant 6$, if each $\zg^i$ forms an external triangle, then it is always possible to add an extra arc $\za$ to obtain a larger admissible arc system. For the picture on the left, each $\rpoint$-arc in the coordinate $\zD^*$ is isotopic to a boundary segment that has a $\bpoint$-point, while for the picture on the right, there exists an $\rpoint$-arc $\ell^*$ that is not isotopic to a boundary segment containing a $\bpoint$-point.}
\label{fig:6base1}
\end{figure}
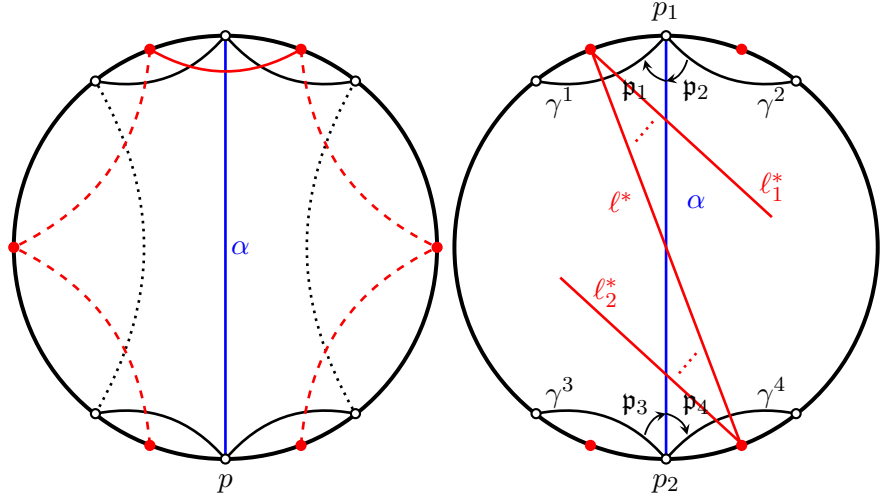

Case II. At least one arc in $\zD^*$ does not form a triangle with boundary segments. We denote this arc by $\ell^*$. Let $\za$ be an $\bpoint$-arc obtained from $\ell^*$ by clockwise rotating both endpoints to the next $\rpoint$-points, see the picture on the right in Figure \ref{fig:6base1}. Then $\za$ is a zigzag arc which gives rise to an indecomposable projective module, and $\za$ differs from any arc in $\zG$. Denote by $p_1$ and $p_2$ the endpoints of $\za$, and we label $\zg^i, 1\leqslant i\leqslant 4$, the arcs in $\zG$ with endpoints $p_1$ and $p_2$, see the picture on the right in Figure \ref{fig:6base1}.
To show that $
\zG\cup\{\za\}$ is admissible concerning $\zD^*$, it only needs to show that the weights of the four oriented intersections $\mathfrak{p}_i, 1\leqslant i \leqslant 4$, respectively between $\za$ and $\zg^i, 1\leqslant i \leqslant 4$, are all different from one. 
It is clear that the weights of $\mathfrak{p}_1$ and $\mathfrak{p}_4$ are zero. On the other hand, since the weight of the oriented intersection $\mathfrak{p}_2\mathfrak{p}_1$ from $\zg^2$ to $\zg^1$ is different from one, the weight of $\mathfrak{p}_2$ is not one. A similar argument shows that the weight of $\mathfrak{p}_3$ is also not one.
Therefore, $\zG\cup\{\za\}$ is still an admissible arc system and $\zG$ is not maximal with respect to $\zD^*$.
\end{proof}

\begin{lemma}\label{lem:ply}
Let $(\cals,\calm,\zD^*)$ be a coordinated-marked surface. Then there is no internal triangle in any admissible arc system. 
\end{lemma}
\begin{proof}
Let $\calp$ be an admissible arc system.
Assume that there is an internal triangle formed by arcs in $\calp$. Denote by $\mathfrak{p}_1, \mathfrak{p}_2$, and $\mathfrak{p}_3$ the oriented intersections between the arcs in the triangle. Then, by Lemma \ref{lem:sumwt}, we have $\omega(\mathfrak{p}_1)+\omega(\mathfrak{p}_2)+\omega(\mathfrak{p}_3)=1$. On the other hand, each weight is non-negative. Therefore, there is exactly one oriented intersection whose weight equals one. This contradicts the assumption that $\calp$ is admissible. 
\end{proof}

We have the following description for $5$-partial triangulations, which can be derived directly from the above result.
\begin{corollary}
An admissible $5$-partial triangulation on $(\cals,\calm,\zD^*)$ is formed by polygons of the types $\textrm{\textbf{Fi}}, 1 \leqslant \textrm{\textbf{i}} \leqslant 5,$ in Figure \ref{fig:ply}.
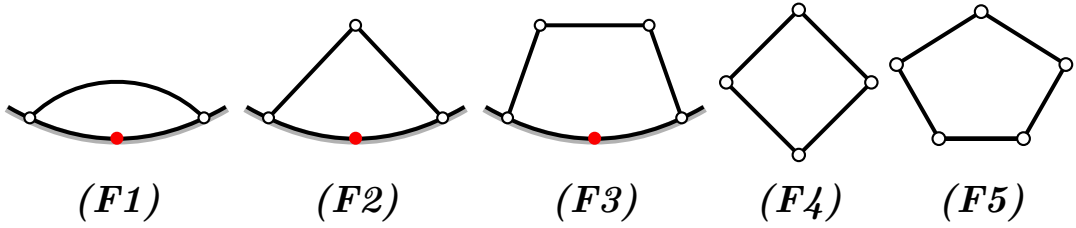
\begin{figure} 
\begin{tikzpicture}[>=stealth,scale=0.36]
\draw [bend right, line width=3pt, gray!60] (-4.05,0.95) to (4.05,0.95);
\draw [bend right, line width=1.5pt] (-4,1) to (4,1);
\draw[cyan,line width=1.5pt,black] (-3.2,0.6) to[out=45,in=135](3.2,0.6);
\draw[thick,fill=white] (-3.2,0.6) circle (0.2);
\draw[thick,fill=white] (3.2,0.6) circle (0.2);
\draw[thick,fill=red, red ] (0,-0.15) circle (0.2);
\draw (0,-2.5) node {\Large \textbf{(F1)} };
\end{tikzpicture}
\begin{tikzpicture}[>=stealth,scale=0.36]
\draw [bend right, line width=3pt, gray!60] (-4.05,0.95) to (4.05,0.95);
\draw [bend right, line width=1.5pt] (-4,1) to (4,1);
\draw[line width=1.5pt,black] (0,4) to(-3.2,0.6);
\draw[line width=1.5pt,black] (0,4) to(3.2,0.6);
\draw[thick,fill=white] (0,4) circle (0.2);
\draw[thick,fill=white] (-3.2,0.6) circle (0.2);
\draw[thick,fill=white] (3.2,0.6) circle (0.2);
\draw[thick,fill=red, red ] (0,-0.15) circle (0.2);
\draw (0,-2.5) node {\Large \textbf{(F2)} };
\end{tikzpicture}
\begin{tikzpicture}[>=stealth,scale=0.36]
\draw [bend right, line width=3pt, gray!60] (-4.05,0.95) to (4.05,0.95);
\draw [bend right, line width=1.5pt] (-4,1) to (4,1);
\draw[cyan,line width=1.5pt,black] (-3.2,0.6) to (-2,4) to (2,4) to(3.2,0.6);
\draw[thick,fill=white] (-2,4) circle (0.2);\draw[thick,fill=white] (2,4) circle (0.2);\draw[thick,fill=white] (-3.2,0.6) circle (0.2);\draw[thick,fill=white] (3.2,0.6) circle (0.2);
\draw[thick,fill=red, red ] (0,-0.15) circle (0.2);
\draw (0,-2.5) node {\Large \textbf{(F3)} };
\end{tikzpicture}
\begin{tikzpicture}[>=stealth,scale=0.48]
\draw [line width=1.5pt] (0,0) to (2,2) to(0,4) to(-2,2)to(0,0) ;
\draw[thick,fill=white] (0,0) circle (0.17);
\draw[thick,fill=white] (0,4) circle (0.17);
\draw[thick,fill=white] (2,2) circle (0.17);
\draw[thick,fill=white] (-2,2) circle (0.17);
\draw (0,-1.3) node {\Large \textbf{(F4)} };
\end{tikzpicture}
\begin{tikzpicture}[>=stealth,scale=0.28]
			
\draw[line width=1.5pt] (-2,0) to (2,0) to (4,3.5)to (0,6)to (-4,3.5)to (-2,0)to (2,0);
\draw[thick,fill=white] (-2,0) circle (0.3);
\draw[thick,fill=white] (2,0) circle (0.3);
\draw[thick,fill=white] (4,3.5) circle (0.3);
\draw[thick,fill=white] (-4,3.5) circle (0.3);
\draw[thick,fill=white] (0,6) circle (0.3);
\draw (0,-3) node {\Large \textbf{(F5)} };
\end{tikzpicture}
\caption{Five kinds of polygons in an admissible 5-partial triangulation, where we label them by $\textbf{Fi}, 1\leqslant \textbf{i} \leqslant 5$.} 
	\label{fig:ply} 
\end{figure}
\end{corollary}

We call two $5$-partial triangulations equivalent if the only possible different polygons of them are external $5$-gons, that is, the polygons of type $\textbf{F3}$ in Figure \ref{fig:ply}. More precisely, if a $5$-partial triangulation $\calp$ can be obtained from another $5$-partial triangulation $\calp'$ by adding an arc in an external $5$-gon of type $\textbf{F3}$ to get an internal $4$-gon of type $\textbf{F4}$ and an external triangle of type $\textbf{F1}$, or vise vars, then we call $\calp$ and $\calp'$ are equivalent. 
Denote by $\{[\calp_5]\}$ the set of the equivalent classes of admissible $5$-partial triangulations. 

\begin{lemma}\label{lem:5partmax}
There is a unique maximal element in each equivalent class of admissible $5$-partial triangulations. Furthermore, this element is a maximal admissible arc system. In particular, if there is an admissible $4$-partial triangulation in an equivalent class of admissible $5$-partial triangulations, then it is the maximal element, which is the unique admissible $4$-partial triangulation in this equivalent class.
\end{lemma}
\begin{proof}
The first statement is clear, where the maximal element contains any element in the same equivalent class.
We denote this maximal element by $\calp$.
In particular, in $\calp$, one cannot add an extra arc to any external $5$-gon of type $\textbf{F3}$ (if there exists such a $5$-gon) to get an admissible arc system.

Now we show the second statement, that is, $\calp$ is a maximal admissible arc system. Otherwise, suppose that we can add an extra arc $\zg$ to $\calp$ to obtain an admissible partial triangulation $\calp'=\calp\cup \{\zg\}$. Then there is a $m$-gon, denoted by $\bbp$, in $\calp$ such that $\zg$ separates $\bbp$ into two polygons $\bbp'$ and $\bbp''$ in $\calp'$. Since $\calp$ is a $5$-partial triangulation, we have $3\leqslant m \leqslant 5$, and $\bbp$ is one of the types $\textbf{F4}$ or $\textbf{F5}$ in Figure \ref{fig:ply}. But then one of $\bbp'$ and $\bbp''$ must be an internal triangle formed by arcs in $\calp'$, noticing that $\calp$ is a maximal element in the equivalent class. This contradicts the assumption that $\calp'$ is admissible since by Lemma \ref{lem:ply} there is no internal triangle in an admissible arc system. Thus, $\calp$ is a maximal admissible arc system.

Since each admissible $4$-partial triangulation must be a maximal admissible arc system, the last statement holds.
\end{proof}

Denote by $\mathbf{\mr}$ the set of isomorphism classes of maximal rigid modules in $\ma$. The above lemma tells us that there is a map:
\[\mathbb{M}: \{[\calp_5]\}\longrightarrow \mr,\]
that maps an equivalent class $[\calp_5]$ of admissible $5$-partial triangulations to a maximal rigid module $\mathbb{M}([\calp_5]):=\oplus M_\zg$, where $\zg$ extends all the arcs in the (unique) maximal element of $[\calp_5]$.

\begin{theorem}\label{thm:max-m2}
The map $\mathbb{M}$ establishes a bijection from the set of the equivalent classes of admissible $5$-partial triangulations on $(\cals,\calm,\zD^*)$ to the set of the maximal rigid modules in $\ma$.
\end{theorem}
\begin{proof}
To show that $\mathbb{M}$ is a bijection, we have to show that any admissible arc system can be completed as an admissible $5$-partial triangulation.
By Lemma \ref{lem:completion}, each admissible arc system can be completed as an admissible partial triangulation. Thus, we only need to show that if there is a $m$-gon $\bbp$ in an admissible partial triangulation $\calp$ with $m\geqslant 6$, then it is not maximal. 

By Proposition \ref{lem:ply2}, $\bbp$ is an internal polygon formed by arcs in $\calp$, where we denote by $\zG=\{\zg^i, 1\leqslant i \leqslant m\}$ the set of arcs in $\bbp$. Then on the cutting surface $(\cals_\zG,\calm_\zG,\zD_\zG)$, there is a connected component with boundary segments $\mathsf{b}_i, 1\leqslant i \leqslant m$, where each $\mathsf{b}_i$ is induced by $\zg^i$. Denote this component by $(\cals',\calm',\zD'^*)$, where $\cals'$ is a disk and there are $m$ $\rpoint$-points in $\calm'$. Denote by $\zg^i_1$ the $\bpoint$-arc that forms a triangle with the segment $\mathsf{b}_i$. Then by Proposition \ref{prop:red2}, $\zG'=\{\zg_1^i, 1\leqslant i \leqslant m\}$ is an admissible arc system on $(\cals',\calm',\zD'^*)$, which is not maximal by Lemma \ref{lem:disk2}. Then after lifting the arc system on the original surface, $\calp$ is not maximal.
\end{proof}

\begin{remark}
An almost rigid module over a finite-dimensional algebra is introduced in \cite{BMGS23}, which is a generalization of classical rigid modules. The original motivation in \cite{BMGS23} was to give an algebraic interpretation of Catalan combinatorics, where the authors show that there is a bijection between the maximal
almost rigid $A$-modules and triangulations of a disk with $(n+1)$ $\bpoint/\rpoint$-marked points, if $A$ is a hereditary algebra of type $A_n$. More recently,
it is proved in \cite{BCGS24} that maximal almost rigid modules over a general gentle algebra are in bijection with so-called permissible triangulations on the associated surface. Note that there may exist internal triangles in such a permissible triangulation. 

On the other hand, $3$-partial triangulations considered in this paper coincide with the classical triangulations on the surface. However, suppose that a coordinated-marked surface has an admissible $3$-partial triangulation. Then each polygon of it must be an external triangle since there is no internal triangle in a partial triangulation according to the Lemma \ref{lem:ply}. Therefore, the surface is, in fact, trivial, that is, a disk with two marked $\bpoint$-point, or a disk with one $\bpoint$-point and one puncture. 
\end{remark}

\subsection{The ranks of maximal rigid modules}

In this subsection, we compute the rank of a maximal rigid module, that is, the number of the indecomposable non-isomorphic direct summands of $M$, or equivalently, the number of internal edges in the associated $5$-partial triangulation.
We start by introducing some notation. 

Let $(\cals,\calm,\zD^*)$ be a coordinated-marked surface. Denote by $g$, $b$, $m$, and $p$ the genus, the number of boundary components, the number of $\bpoint/\rpoint$-points, and the number of punctures of the surface. Denote by $A$ the gentle algebra associated with $(\cals,\calm,\zD^*)$, whose rank is as follows, see for example in \cite{APS23}, 
$$n=2g-2+b+p+m.$$

Let $\calp_5$ be a $5$-partial triangulation on the surface. Denote by $e$, $f$, and $v$ the number of edges, faces (polygons), and vertices in $\calp_5$. Since the set of the endpoints of the edges in $\calp_5$ coincides with $\calm=\calm_{\rpoint}\cup\calm_\bpoint\cup\calp_\bpoint$, we have 
$$v=2m+p.$$ 
On the other hand, we have $$f=\sum_{1\leqslant i \leqslant 5}f_i,$$ where $f_i$ is the number of faces of type $\textbf{Fi}$, respectively, depicted in Figure \ref{fig:ply}. 
Let $e_1$ and $e_2$ be the number of internal and external edges in $\calp_5$, respectively. Then we have $$e=e_1+e_2\text{, and}$$ 
$$e_2=2m,$$
where the last equality holds since each external edge has adjacent $\bpoint/\rpoint$-points as endpoints.

\begin{theorem}\label{thm:max-m3}
Let $M$ be the maximal rigid module in $\ma$ associated with a (maximal) $5$-partial triangulation $\calp_5$. Under the above notation, the rank of $M$ equals 
$$e_1=n+f_4+f_5.$$
\end{theorem}
\begin{proof}
The Euler characteristic of the marked surface is $$\chi=2-2g-b.$$
On the other hand, $\calp_5$ is a CW-complex whose Euler characteristic is 
\[f-e+v,\]
which equals the Euler characteristic of the marked surface. Therefore we have
\[f-e+v=2-2g-b.\]
Thus, the rank of $M$ equals
\[e_1=n+f_4+f_5,\]
using equalities
\[e=e_1+e_2=e_1+2m, v=2m+p, f=m+f_4+f_5,\]
where the last equality holds since a polygon of $\calp_5$ has (exactly) one $\rpoint$-point if and only if it is of the type $\textbf{Fi}$ for $i=1,2,3$, and thus $f_1+f_2+f_3=m.$
\end{proof}

\begin{example}\label{ex:4pt}
For any gentle algebra $A$ associated with a coordinated-marked surface $(\cals,\calm,\zD^*)$, we construct a $4$-partial triangulation $\calp$ by adding several arcs in the injective dissection $\zD_I=t^{-1}(\zD^*)$. Let $\bbp$ be a $(\nu+2)$-gon of $\zD_I$ that has $\nu$ arcs and two boundary segments. Denote by $\za_i, 1\leqslant i \leqslant \nu,$ the arcs in $\bbp$, where we label them clockwise, and denote by $p_{i-1}$ and $p_i$ the endpoints of $\za_i$.
Let $\zb_j, 1\leqslant j \leqslant s,$ be an $\bpoint$-arc in $\bbp$ with endpoints $p_0$ and $p_{2j+1}$, where $s=[{\frac{\nu-1}{2}}]$. Then $\zb_j's$ separate $\bbp$ into $s$ internal $4$-gons of type $\textbf{F4}$ and one external triangle of type $\textbf{F1}$ if $\nu$ is odd, while they separate $\bbp$ into $s$ internal $4$-gons of type $\textbf{F4}$ and one external $4$-gon of type $\textbf{F2}$ if $\nu$ is even. Denote by $\zG$ the set of arcs obtained from $\zD_I$ by adding arcs $\zb_j's$ to each polygon $\bbp$. Then it is not hard to see that $\calp$ is an admissible $4$-partial triangulation, which is a maximal admissible arc system by Lemma \ref{lem:5partmax}.
Furthermore, by the above argument, we know that the rank of $\calp$ equals
$$n+\sum_{\bbp_i}[{\frac{\nu_i-1}{2}}],$$
where $\bbp_i$ extends all the polygons of $\zD_I$. In Proposition \ref{prop:geoinjcl}, we will see that $\calp$ corresponds to the $\tau_2$-orbit of injective modules in $\ma$, see the picture in Figure \ref{fig:18}, where $\zb_j's$ are labeled by the blue arcs $\zg_p^{2j+1}$.
\end{example}

The following result describes the surface of a gentle algebra so that the rank of a maximal rigid module over the algebra equals the rank of the algebra.

\begin{corollary}\label{thm:ply}
Let $(\cals,\calm,\zD^*)$ be a coordinated-marked surface associated with a gentle algebra $A$. Let $n$ be the rank of $A.$
Then the rank of any maximal rigid module in $\ma$ is not smaller than $n$.
Furthermore, any maximal rigid module in $\ma$ has rank $n$ if and only if $A$ is a hereditary algebra of type $A$ or of affine type $A$.
\end{corollary}
\begin{proof}
The first statement can be directly derived from Theorem \ref{thm:max-m2}.
We prove the second statement.
Let $M$ be the maximal rigid module associated with the $4$-partial triangulation $\calp$ constructed in Example \ref{ex:4pt}.
Since the rank of $M$ is $n$, using the formula of the rank of $\calp$ given in the example, we get $\nu_i=1$ or $2$ for each polygon $\bbp_i$ of $\zD_I$. 
Then each polygon in $\zD^*$ is also a triangle or a $4$-gon, and thus the associated gentle algebra is hereditary by the definition of a gentle algebra arising from $\zD^*$ given in Definition \ref{definition:gentle algebra from dissection}.
Conversely, if $A$ is a gentle algebra of type $A$ or affine type $A$, then the maximal rigid modules coincide with the (classical) tilting modules. Therefore, the rank of any maximal rigid module equals the rank of the algebra since this holds for any tilting module.
\end{proof}


\begin{remark}
Skew-gentle algebras and string algebras are two kinds of generalizations of gentle algebras. 
The geometric models are established in \cite{HZZ23} and \cite{BC24}, respectively. It also seems interesting to realize maximal rigid modules over them as certain partial triangulations on the surface.
\end{remark}

\section{Higher-homology for the module categories of gentle algebras}\label{section:higher-hom-mod}

\subsection{Basic concepts in higher Auslander-Reiten theory}\label{subsection:Higher Auslander-Reiten}
In this subsection, we recall some basic concepts in the higher Auslander-Reiten theory. We refer the reader to \cite{ARS,ASS} for the basic terminologies in the representation theory of finite-dimensional algebras appearing in the following, and to \cite{I07b,I11} and survey papers \cite{I08, JK19} for more details on the higher Auslander-Reiten theory.

\begin{definition}\label{define n-cluster tilting}
Let $A$ be a finite-dimensional algebra, and let $n\ge1$. 
We call a module $M$ in $\ma$  \emph{$n$-rigid} if $\Ext^i(M,M)=0$ for any $1\leqslant i\leqslant n-1$.
We call $M$ \emph{$n$-cluster tilting} in a extension closed subcategory $\calc$ of $\ma$ if
	\begin{eqnarray*}
		\add M&=&\{X\in \calc\ |\ \Ext^i(X,\calc)=0\ (1\leqslant i\leqslant n-1)\}\\
		&=&\{X\in \calc\ |\ \Ext^i(\calc,X)=0\ (1\leqslant i\leqslant n-1)\}.
	\end{eqnarray*}
\end{definition}
We denote by
\[D:=\Hom_k(-,k):\ma\leftrightarrow \ma^{op},\]
 the $k$-duality and by
\[\nu:=D\Hom_A(-,A): \ma\to \ma\]
the Nakayama functor of $A$.
The Auslander-Bridger transpose duality and the syzygy functor are, respectively, denoted by
\[\tr:{\urm}\leftrightarrow {\urmo}\ \mbox{ and }\ 
\Omega:{\urm} \rightarrow {\urm}\]
where ${\urm}$ and ${\urmo}$ are the stable categories of $\ma$ and $\ma^{op}$ with respect to the projective modules.

For $n\ge1$, the \emph{$n$-Auslander-Reiten translations} is defined by
\[\tau_n:=D\tr\Omega^{n-1}:{\urm}\to {\urm}.\]
Note that $M\in \ma$ satisfies $\tau_nM=0$ if and only if $\pd M<n$.
The \emph{$\tau_n$-closure} of the injective module $DA$ is defined by
\[\mathfrak{M}=\mathfrak{M}_n(DA):=\add\{\tau_n^i(DA)\ |\ i\geqslant 0\}\subset \ma.\]

We introduce the following subcategories:
	\begin{eqnarray*}
		\mathfrak{I}(\mathfrak{M})&:=&\add DA,\\
		\mathfrak{P}(\mathfrak{M})&:=&\{X\in\mathfrak{M}\ |\ \pd X_A<n\}=\{X\in\mathfrak{M}\ |\ \tau_nX=0\},\\
		\mathfrak{M}_I&:=&\{X\in\mathfrak{M}\ |\ X\ \mbox{ has no non-zero summands in }\ \mathfrak{I}(\calm)\},\\
		\mathfrak{M}_P&:=&\{X\in\mathfrak{M}\ |\ X\ \mbox{ has no non-zero summands in }\ \mathfrak{P}(\mathfrak{M})\}.
	\end{eqnarray*}

\begin{definition}\label{complete}
Let $A$ be a finite-dimensional algebra.
	\begin{enumerate}
				\item We call $A$ \emph{$\tau_n$-finite} if $\gd A\le n$
		and $\tau_n^\ell(DA)=0$ holds for sufficiently large $\ell$.
		In this case, we have $\tau_n^\ell=0$.
		\item We call $A$ \emph{$n$-complete} if $\gd A\le n$ and
		the following conditions are satisfied.
		\begin{itemize}
			\item[(A$_n$)] There exists a tilting $A$-module $T$ satisfying $\mathfrak{P}(\mathfrak{M})=\add T$,
			\item[(B$_n$)] $\mathfrak{M}$ is an $n$-cluster tilting subcategory of $T^{\perp}=\{X\in \ma\ |\ \Ext_A^i(T,X)=0\ (i>0)\}$, noticing that since $T$ is tilting, $T^{\perp}$ is an extension closed subcategory
			of $\ma$. 
			\item[(C$_n$)] $\Ext_A^i(\mathfrak{M}_P,A)=0$ for any $1\leqslant i\leqslant n-1$.
		\end{itemize}
		We call $A$ \emph{absolutely $n$-complete} if $\mathfrak{P}(\mathfrak{M})=\add A$.
	\end{enumerate}
\end{definition}

Inspired by the above definition, we introduce the following concept: for an algebra $A$ with a finite global dimension, we call the minimal positive number $\ell$ such that $\tau_n^\ell=0$ the \emph{$\tau_n$-dimension} of $A$. In particular, the $\tau_n$-dimension of $A$ is infinite if there exists no such positive number.
Then an algebra is $\tau_n$-finite if and only if its $\tau_n$-dimension is finite.
 
Note that a $n$-complete algebra is $\tau_n$-finite, thus $\mathfrak{M}$ has an additive
generator $M$.
We call the endomorphism algebra $\End_A(M)$ of $M$ in $\ma$ the \emph{cone} of $A$.
One of the main theorem (Theorem 1.14) in \cite{I11} says that for any $n\ge1$, the cone of a $n$-complete algebra is $(n+1)$-complete.

\subsection{Geometric realizations for $\Omega^m$ and $\tau_m$}\label{section:geo-taun}
In the following, we consider a gentle algebra $A$ with a finite global dimension. Thus, there is no puncture on the associated coordinated-marked surface $(\cals,\calm,\zD^*)$. 

Let $\za$ be a zigzag arc with an endpoint $p$, where $p$ belongs to a polygon $\bbp=\{\ell^*_1,\ell^*_2,\cdots,\ell^*_s\}$ of $\zD^*$ with the arcs labeled clockwise, see Figure \ref{figure:weight and co-weight}.
Assume that starting from $p$, $\ell^*_t$ is the first arc that $\za$ intersects, for some $1\leqslant t \leqslant s$. We call $s-t$ the \emph{weight} of $\za$ at $p$, which is denoted by $w_p(\za)$. Under this notion, the weight $\omega(\mathfrak{p})$ of an oriented intersection $\mathfrak{p}$ from $\za$ to $\zb$ based on the endpoint $p$ defined in Definition \ref{definition: weight} equals $w_p(\za)-w_p(\zb)$.
	\begin{figure}
		\begin{center}
			\begin{tikzpicture}[>=stealth,scale=1]
				\draw[red,very thick] (1,0)--(3,1)--(3,3)--(1,4);
				\draw[red,very thick,dashed] (-1,4)--(1,4);
				\draw[red,very thick] (-1,0)--(-3,1)--(-3,3)--(-1,4);
				
				\draw[line width=1pt] (-2.5,4)--(0,0)--(2.5,4);
				\draw [ gray!60, line width=3pt] (-2,-.07)--(2,-.07);		
					\draw[line width=1.5pt] (-2,0)--(2,0);			\draw[thick,fill=white] (0,0) circle (0.07);
				\draw (0,-.5) node {$p$};
				\node[red] at (-1,3.4) {$\ell^*_{s-\omega_p(\za)}$};
				\node[red] at (-1.7,4) {$\ell^*_{t}$};
				\node[red] at (1.2,3.4) {$\ell^*_{s-\omega_p(\zb)}$};
				\node at (-1.6,2) {$\za$};
				\node at (1.5,2) {$\zb$};
				\node[red] at (-2.3,1) {$\ell^*_{1}$};
				\node[red] at (2.3,1) {$\ell^*_{s}$};
				\node[red] at (0,3) {$\bbp$};
				
				\draw[thick,bend left,->](-.2,.3)to(.2,.3);
				\node [] at (0,.6) {\bf$\mathfrak{p}$};	
				\draw[red,thick,fill=red] (1,0) circle (0.07);
				\draw[red,thick,fill=red] (-1,0) circle (0.07);
				\draw[red,thick,fill=red] (3,1) circle (0.07);
				\draw[red,thick,fill=red] (3,3) circle (0.07);
				\draw[red,thick,fill=red] (1,4) circle (0.07);
				\draw[red,thick,fill=red] (1,0) circle (0.07);
				\draw[red,thick,fill=red] (-3,1) circle (0.07);
				\draw[red,thick,fill=red] (-3,3) circle (0.07);
				\draw[red,thick,fill=red] (-1,4) circle (0.07);
			\end{tikzpicture}
		\end{center}
		\begin{center}
			\caption{The polygon $\bbp$ has $s$ $\rpoint$-arcs $\ell^*_i, 1\leqslant i \leqslant s$. For a zigzag arc $\za$ with endpoint $p$ which intersects $\ell^*_t$, the weight  $w_p(\za)$ of $\za$ at $p$ equals $s-t$. Then the weight $\omega(\mathfrak{p})$ of an oriented intersection $\mathfrak{p}$ from $\za$ to $\zb$ equals $w_p(\za)-w_p(\zb)$.}\label{figure:weight and co-weight}
		\end{center}
	\end{figure}
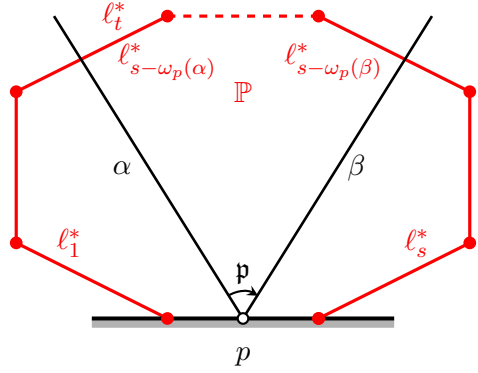
 
\begin{definition}
For a $(s+2)$-gon $\bbp=\{\ell^*_1,\ell^*_2,\cdots,\ell^*_s\}$ of $\zD^*$ which contains a $\bpoint$-point $p$, we define the arcs $_p^t\zg$ and $\zg_p^t$ as depicted in Figure \ref{fig:twotypearcs}, if $1\leqslant t\leqslant s$; or define them as the point $p$, which is viewed as a trivial arc, if $t > s$.

Denote by $_p^tM$ and $M_p^t$ the $A$-modules associated with $_p^t\zg$ and $\zg_p^t$ respectively. In particular, if the arcs are trivial, then $_p^tM$ and $M_p^t$ are zero modules.
\begin{figure} 
\centering 
\begin{tikzpicture}[scale=0.7]
			\draw [line width=1.5pt] (-8,0) to (-0.5,0);
			\draw [gray!60, line width=3pt] (-8,-.1)--(-0.5,-.1);	
	     	\draw[gray!60, line width=3pt] (-4.1,8.2) to[out=-90, in=85] (-4.13,7.4) to[out=-95,in=45](-5.55,5.2)to[out=-135,in=10](-7.51,4.4);
	     	\draw[cyan, line width=1.5pt,black] (-4,8.2) to[out=-90, in=85] (-4.05,7.3) to[out=-95,in=45](-5.5,5.1)to[out=-135,in=10](-7.5,4.3);  
            
			\draw [line width=1pt,red] (-6,0) to (-7.5,3);	
            \draw [line width=1pt,red] (-5.56,5) to (-7.5,3);
			\draw [line width=1pt,red] (-5.56,5) to (-2.5,5);
			\draw [line width=1pt,red] (-1,3) to (-2.5,5);
			\draw [line width=1pt,red] (-1,3) to (-2.5,0);
			\draw [line width=1pt,red] (-5.56,5) to (-2,8);
			\draw [line width=1pt,red] (-5.56,5) to (-1.7,6.8);
			\draw [line width=1pt,red,dotted] (-3,7) to (-2.6,6.4);
     	
	     	\draw [line width=1pt] (-4.05,0) to (-4.05,7.3);	
 	            
			\draw[cyan, line width=1pt,black] (-4.05,0) to[out=120, in=-20] (-8,3.8) ;		           
			\draw[thick,black, fill=white] (-4.05,0) circle (0.1);
			\draw[thick,black, fill=white] (-4.05,7.3) circle (0.1);
			\draw[thick,red, fill=red ] (-6,0) circle (0.08);
			\draw[thick,red, fill=red ] (-7.5,3) circle (0.08);
			\draw[thick,red, fill=red ] (-5.6,5) circle (0.08);
			\draw[thick,red, fill=red ] (-2.5,5) circle (0.08);
			\draw[thick,red, fill=red ] (-1,3) circle (0.1);
			\draw[thick,red, fill=red ] (-2.5,0) circle (0.08);

            \draw (-5.3,2.5) node {$\gamma$};            
            \draw (-3.4,2.5) node {$^{t}_{p}\gamma$};
			  \draw (-7.2,1.5) node[red] {$\ell^{*}_{1}$};
			\draw (-3.3,4.6) node[red] {$\ell^{*}_{t}$};
     		\draw (-5.7,4) node[red] {$\ell^{*}_{s-\omega}$};
			\draw (-1.3,1.5) node[red] {$\ell^{*}_{s}$};
            \draw (-4,-.5) node {$p$};		
            \draw (4,-.5) node {$p$};	
            \draw[red] (-2.2,2.5) node {$\bbp$};		
            \draw[red] (5.6,2.5) node {$\bbp$};	
            \draw (4,-.5) node {$p$};

			\draw [line width=1.5pt] (8,0) to (0.5,0);
			\draw [gray!60, line width=3pt] (8,-.1)--(0.5,-.1);	
	     	\draw[gray!60, line width=3pt] (4.1,8.2) to[out=-90, in=95] (4.15,7.4) to[out=-85,in=135](5.55,5.2)to[out=-45,in=170](7.51,4.4);
	     	\draw[cyan, line width=1.5pt,black] (4,8.2) to[out=-90, in=95] (4.07,7.3) to[out=-85,in=135](5.5,5.1)to[out=-45,in=170](7.5,4.3);
            
			\draw [line width=1pt,red] (6,0) to (7.5,3);	
			\draw [line width=1pt,red] (5.56,5) to (7.5,3);
			\draw [line width=1pt,red] (5.56,5) to (2.5,5);
			\draw [line width=1pt,red] (1,3) to (2.5,5);
			\draw [line width=1pt,red] (1,3) to (2.5,0);
			\draw [line width=1pt,red] (5.56,5) to (2,8);
			\draw [line width=1pt,red] (5.56,5) to (1.7,6.8);
			\draw [line width=1pt,red,dotted] (3,7) to (2.8,6.2);
			
			\draw [line width=1pt] (4.05,0) to (4.05,7.3);	
            
			\draw[cyan, line width=1pt,black] (4.05,0) to[out=120, in=-20] (1,3.8) ;		
            
			\draw[thick,black, fill=white] (4.05,0) circle (0.1);
			\draw[thick,black, fill=white] (4.05,7.3) circle (0.1);
			\draw[thick,red, fill=red ] (6,0) circle (0.08);
			\draw[thick,red, fill=red ] (7.5,3) circle (0.08);
			\draw[thick,red, fill=red ] (5.6,5) circle (0.08);
			\draw[thick,red, fill=red ] (2.5,5) circle (0.08);
			\draw[thick,red, fill=red ] (1,3) circle (0.08);
			\draw[thick,red, fill=red ] (2.5,0) circle (0.08);

			\draw (4.4,2.5) node {$\gamma^{t}_{p}$};
			\draw (7.2,1.5) node[red] {$\ell^{*}_{s}$};
			\draw (3.3,4.6) node[red] {$\ell^{*}_{t}$};
 			\draw (2.5,4) node[red] {$\ell^{*}_{s-\omega}$};
			\draw (1.3,1.5) node[red] {$\ell^{*}_{1}$};
	        \draw (2,2.5) node {$\gamma$};\end{tikzpicture}
	\caption{The definition of arcs $_p^t\zg$ and $\zg_p^t$, which give rise to modules $_p^tM$ and $M_p^t$ respectively. Assume that the weight of $\zg$ at $p$ is $\omega$, then Lemma \ref{lem:geost} says that the module $^{s-\omega+m}_pM$ is a direct summand of $\Omega^mM_\zg$ for $m\geqslant 1$, and the module $M^{s-\omega+m}_p$ is a direct summand of $\tau_mM_\zg$ for $m\geqslant 2$.} 
	\label{fig:twotypearcs} 
	\end{figure}
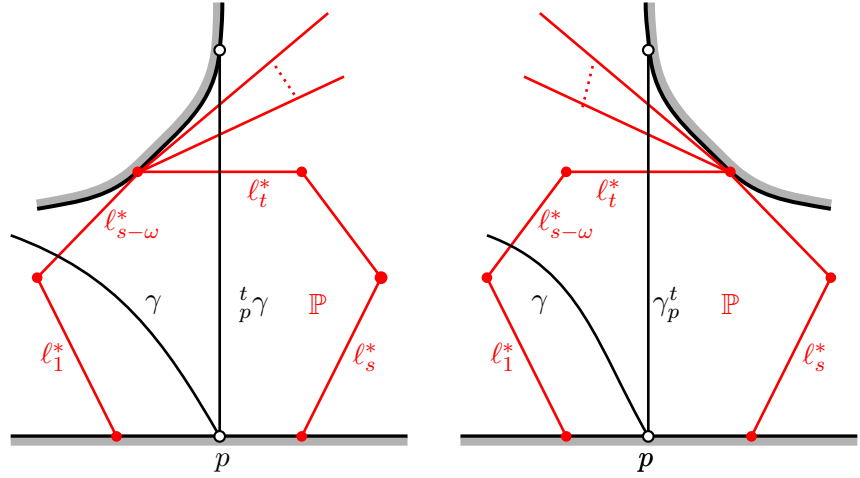
\end{definition}

\begin{lemma}\label{lem:geost}
Let $M$ be a module associated with a zigzag arc $\zg$. Denote by $p_i, i=1,2,$ the endpoints of $\zg$ which are contained in $(s_i+2)$-gon $\bbp_i$ respectively. Assume that the weights of $\zg$ at $p_i$ are respectively $\omega_i$. Then 
\begin{enumerate}
	\item $\Omega^m M=P\oplus {_{p_1}^{s_1-\omega_1+m}M}\oplus {_{p_2}^{s_2-\omega_2+m}M}$ for some projective module $P$, where $P=0$ when $m\geqslant 2$ and $P$ may be zero when $m=1$;
	\item $\tau_m M={M_{p_1}^{s_1-\omega_1+m}}\oplus {M_{p_2}^{s_2-\omega_2+m}}$ for any $m\geqslant 2$. 
\end{enumerate}
\end{lemma}
\begin{proof}
The description of $\Omega^m$ is given in \cite{HS05,BS21}. After translating the results on the surface, we get the geometric realization for $\Omega^m$ as in the statement. On the other hand, the Auslander-Reiten translation $\tau$ is described in \cite{BR87}, and the geometric explanation is given in \cite{BC21}. Then we get the geometric characterization of $\tau_m$ for $ m\geqslant 2$, by combining the above results together. 
\end{proof}

The following corollary describes the string combinatorics for $\Omega^m$ and $\tau_m$, which can be proved straightforwardly by the above lemma, where $\overline{\sigma}$ is the inverse string for a string $\sigma$.
\begin{corollary}\label{cor;string}
Let $\sigma$ be a string and let $M_\sigma$ be the associated string module. If $m\geqslant 2$, then $\Omega^m(M_\sigma)=M_{\sigma^m}\oplus M_{\overline{\sigma}^m}$ and $\tau_m(M_\sigma)=M_{\sigma_m}\oplus M_{\overline{\sigma}_m}$, see $\sigma^m$ and $\sigma_m$ in Figure \ref{fig:cor}.
\begin{figure} 
	\centering 
			\begin{tikzpicture}[>=stealth, scale=.9]
				
				\path 
				(-8,0) coordinate (1)
				(-4,0) coordinate (2)
				(-3.5,-1)  coordinate (3)
				(-3,-2)  coordinate (4)
				(-2.5,-3) coordinate (5)
				(-2,-4) coordinate (6)
				(-1.5,-5) coordinate (7)
				(-1,-6) coordinate (8)
				(-0.5,-7) coordinate (9)
				(0,-8)  coordinate (10)

				(-2,0)  coordinate (11)
				(2,0) coordinate (12)
				(2.5,-1) coordinate (13)
				(3,-2) coordinate (14)
				(3.5,-3) coordinate (15)
				(4-6,-4) coordinate (16)
				(4.5-6,-3)  coordinate (17)
				(5-6,-2)  coordinate (18)
				(5.5-6,-1) coordinate (19)
				(6-6,0) coordinate (20);

				\draw 
				(1) node {}
				(2) node { }
				(3) node {$\bullet$}
				(4) node {$\bullet$}
				(5) node {$\bullet$}
				(6) node {$\bullet$}
				(7) node {$\bullet$}
				(8) node {$\bullet$}
				(9) node {$\bullet$}
				(11) node { }
				(12) node { }

				(16) node {$\bullet$}
				(17) node {$\bullet$}
				(18) node {$\bullet$}
				(19) node {$\bullet$};
			
				 \draw[decorate,line width=1pt,decoration={snake,amplitude=1mm,segment length=3mm}] (1)--(2);
				\draw [thick,->] (2)--(-3.55,-0.95);
				\draw [thick,->] (-3.5,-1.05)--(-3.05,-1.95);
				\draw [thick] (-3,-2.05)--(-2.55,-2.9);
				\draw [thick,->] (-2.45,-3.1)--(-2.05,-3.9);
				\draw [thick,->] (-1.95,-4.1)--(-1.55,-4.9);
				\draw [thick,dashed] (-1.45,-5.1)--(-1.05,-5.9);
				\draw [thick,->] (-0.95,-6.1)--(-0.55,-6.9);
		 
				\draw [thick,->] (4.45-6,-3.1)--(4.05-6,-3.9);
				\draw [thick, dashed] (4.98-6,-2.05)--(4.55-6,-2.9);
				\draw [thick,->] (5.45-6,-1.1)--(5.05-6,-1.9);

\draw (-6,-0.5) node {$\sigma$};
				\draw (-4.1,-0.5) node {$a_{1}$};
				\draw (-3.6,-1.5) node {$a_{2}$};
				\draw (-2.65,-3.4) node {$a_{m}$};
				\draw (-3.5,-5.8) node {$\sigma^{m}$};
				\draw (-3.7,-6.3) node {(longest)};

        		\draw (1.5,-1.7) node {$\sigma_{m}$};
				\draw (1.7,-2.3) node {(longest)};
                
				\draw[dotted,thick,bend left](-3.65,-0.5) to (-3.25,-1.45);	
				\draw[dotted,thick,bend left](-3.2,-1.55) to (-2.7,-2.5);	
				\draw[dotted,thick,bend left](-2.65,-2.6) to (-2.2,-3.5);	
			
				\draw[decorate, decoration={brace, amplitude=5mm, mirror}] (-2.2,-4) -- (-2.2,-7);

				\draw[decorate, decoration={brace, amplitude=5mm}] (0.2,-1) -- (0.2,-4);
			\end{tikzpicture}
	\caption{The string combinatorics of $\Omega^m$ and $\tau_m$ for $m\geqslant 2$, where the waved part is the string $\sigma$, and $\sigma^m$ and $\sigma_m$ are respectively the strings give rise to the direct summand of $\Omega^m(M_{\sigma})$ and $\tau_m(M_{\sigma})$ respectively. To get the other direct summand, we consider the inverse string $\overline{\sigma}$, and the associated strings ${\overline{\sigma}}^m$ and ${\overline{\sigma}}_m$, see Corollary \ref{cor;string}.} 
	\label{fig:cor} 
\end{figure}
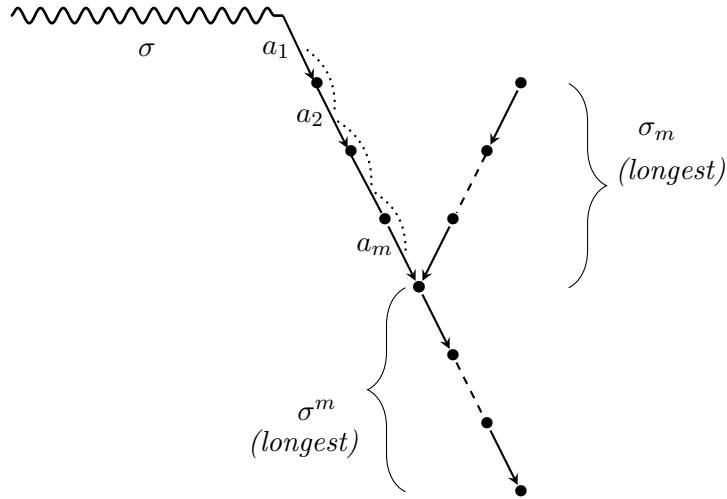
\end{corollary}

\subsection{A geometric realization for $\tau_m$-closure of injectives}\label{section:tau-inj}

Let $A$ be a gentle algebra associated with a surface model $(\cals,\calm,\zD^*)$. We assume that the global dimension of $A$ is equal to $n$ in this subsection.
For each $2\leqslant m \leqslant n$, we will describe the $\tau_m$-orbit of injective modules over $A$. The main result in this subsection is as follows:

\begin{proposition}\label{prop:geoinjcl}
For $2\leqslant m\leqslant n$, the $\tau_m$-orbit of the injective modules in $\ma$ gives rise to an admissible $(m+2)$-partial triangulation. Therefore, the direct sum of the modules in the orbit is a rigid module. In particular, the $\tau_2$-orbit of the injective modules gives rise to an admissible $4$-partial triangulation, for which the associated module is a maximal rigid module.
\end{proposition}
\begin{proof}
We prove the case where $m=2$, and the proofs for the other cases are similar.
Let $\bbp$ be a $(\nu+2)$-gon of $\zD^*$ that contains a $\bpoint$-point $p$. Then $t^{-1}(\bbp)$ is a $(\nu+2)$-gon of the injective dissection $\zD_I=t^{-1}(\zD^*)$, which is a set of $\bpoint$-arcs corresponding to indecomposable injective modules. Denote by $\ell^*$ the first arc of $\bbp$ in a clockwise orientation, and denote by $\zg$ the anti-twist of $\ell^*$, that is, $\zg=t^{-1}(\ell^*)$. Let $\zg_1,\zg_2,\cdots,\zg_s$ be the remained arcs in $\zD_I$ with $p$ as an endpoint, see the picture in Figure \ref{fig:18}.

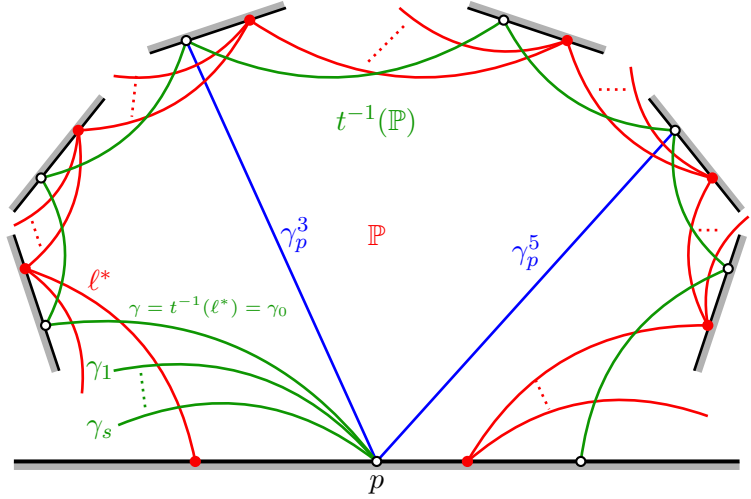
\begin{figure}
\centering
\begin{tikzpicture}[>=stealth,scale=.6]
\draw [ line width=3pt, gray!60] (8,-0.1) to (-8,-0.1);
\draw [ line width=1.5pt] (8,0) to (-8,0);
			
\draw [ line width=3pt, gray!60] (-2.04,10.07) to (-5.04,9.07);
\draw [ line width=3pt, gray!60] (2.04,10.07) to (5.04,9.07);
\draw [ line width=3pt, gray!60] (-6.05,8.04) to (-8.05,5.54);
\draw [ line width=3pt, gray!60] (6.05,8.04) to (8.05,5.54);
			\draw [ line width=3pt, gray!60] (-8.08,4.98) to (-7.08,1.98);
			\draw [ line width=3pt, gray!60] (8.08,4.98) to (7.08,1.98);
		    
			\draw [ line width=1pt] (-2,10) to (-5,9);
			\draw [ line width=1pt] (2,10) to (5,9);
			\draw [ line width=1pt] (-6,8) to (-8,5.5);
			\draw [ line width=1pt] (6,8) to (8,5.5);
			\draw [ line width=1pt] (-8,5) to (-7,2);
			\draw [ line width=1pt] (8,5) to (7,2);

			\draw [ line width=1pt, blue ] (-4.2,9.28) to (0,0) to (6.58,7.3);

			\draw [  line width=1pt, red,dotted ] (-7.6,5.3) to (-7.4,4.7) ;
			\draw [  line width=1pt, red,dotted ] (-5.3,8.5) to (-5.4,7.6) ;
			\draw [  line width=1pt, red,dotted ] (0.6,9.6) to (-0.2,8.8) ;
			\draw [  line width=1pt, red,dotted ] (4.9,8.2) to (5.5,8.2) ;
			\draw [  line width=1pt, red,dotted ] (7.1,5.1) to (7.6,5.1) ;
			\draw [  line width=1pt, red,dotted ] (3.5,1.8) to (3.8,1.15) ;
			\draw [  line width=1pt, dark-green,dotted ] (-5.2,1.95) to (-5.1,1.15) ;

			\draw [bend right,  line width=1pt, red ] (-4,0) to (-7.75,4.25) ;
			\draw [bend left,  line width=1pt, red ] (2,0) to (7.3,3) ;
			\draw [bend left,  line width=1pt, red ] (2,0) to (7.3,1) ;
			\draw [bend left,  line width=1pt, red ] (7.3,3) to  (7.4,6.25) ;
			\draw [bend left,  line width=1pt, red ] (7.3,3) to  (8.2,5.4) ;
			\draw [bend left,  line width=1pt, red ] (7.4,6.25) to  (4.2,9.28);
			\draw [bend left,  line width=1pt, red] (7.4,6.25) to  (5.6,8.7);
			\draw [bend left,  line width=1pt, red] (4.2,9.28) to (-2.8,9.73) ;
			\draw [bend left,  line width=1pt, red] (-2.8,9.73) to  (-6.58,7.3);
			\draw [bend left,  line width=1pt, red] (-6.58,7.3) to (-7.75,4.25) ;
			\draw [bend left,  line width=1pt, red] (4.2,9.28) to (0.5,10) ;
			\draw [bend left,  line width=1pt, red] (-2.8,9.73) to  (-5.8,8.5);
			\draw [bend left,  line width=1pt, red] (-6.58,7.3) to (-8,5.2) ;
			\draw [bend left,  line width=1pt, red] (-7.75,4.25) to (-6.5,1.5) ;

			\draw [bend left,  line width=1pt, dark-green ] (7.75,4.25) to (6.58,7.3) ;
			\draw [bend right,  line width=1pt, dark-green ] (0,0) to (-7.3,3) ;
			\draw [bend right,  line width=1pt, dark-green ] (0,0) to (-5.8,2) ;
			\draw [bend right,  line width=1pt, dark-green ] (0,0) to (-5.7,0.8) ;
			\draw [bend left,  line width=1pt, dark-green ] (4.5,0) to (7.75,4.25) ;
			\draw [bend right,  line width=1pt, dark-green ] (2.8,9.73) to (6.58,7.3);
			\draw [bend left,  line width=1pt, dark-green ] (2.8,9.73) to (-4.2,9.28) ;
			\draw [bend left,  line width=1pt, dark-green ] (-4.2,9.28) to (-7.4,6.25) ;
			\draw [bend left,  line width=1pt, dark-green ] (-7.4,6.25) to (-7.3,3) ;

			\draw[thick,fill=white] (0,0) circle (0.1);
			\draw[thick,fill=white] (4.5,0)  circle (0.1);
			\draw[thick,red, fill=red] (-4,0)  circle (0.1);
			\draw[thick,red, fill=red] (2,0)  circle (0.1);

			\draw[thick,fill=white] (-4.2,9.28) circle (0.1);
			\draw[thick,red, fill=red] (-2.8,9.73)  circle (0.1);
			\draw[thick,red, fill=red] (4.2,9.28) circle (0.1);
			\draw[thick,fill=white] (2.8,9.73)  circle (0.1);
			
			\draw[thick,fill=white] (-7.4,6.25) circle (0.1);
			\draw[thick,red, fill=red] (-6.58,7.3)  circle (0.1);
			\draw[thick,red, fill=red] (7.4,6.25) circle (0.1);
			\draw[thick,fill=white] (6.58,7.3)  circle (0.1);
			
			\draw[thick,red, fill=red] (-7.75,4.25) circle (0.1);
			\draw[thick,fill=white] (-7.3,3)  circle (0.1);
			\draw[thick,fill=white] (7.75,4.25) circle (0.1);
			\draw[thick,red, fill=red] (7.3,3)  circle (0.1);

			\draw (0,7.5) node[dark-green] {$t^{-1}(\bbp)$};
			\draw (-6.1,2) node[dark-green] {$\gamma_{1}$};
			\draw (-6.1,0.7) node[dark-green] {$\gamma_{s}$};
			\draw (-3.7,3.4) node[dark-green] {\tiny$\gamma = t^{-1}(\ell^{*}) = \gamma_{0}$};
			\draw (-6.1,4) node[red] {$\ell^{*}$};
\draw (0,5) node[red] {$\bbp$};
\draw (0,-.5) node {$p$};
			\draw (-1.8,5) node[blue] {$\gamma^{3}_{p}$};
			\draw (3.3,4.6) node[blue] {$\gamma^{5}_{p}$};
\end{tikzpicture}\caption{A local configuration of arcs in the $\tau_2$-orbit of injective arcs, that is, the arcs associated with injective modules. The injective arcs are colored green and the newly appearing arcs $\zg_p^{t}$ are colored blue.}
\label{fig:18}
\end{figure}
For convenience, we write $\zg=\zg_0$. Then ${(\zg_i)}_p^t=\zg_p^t$ for any $1\leqslant i \leqslant s$. Furthermore, by Lemma \ref{lem:geost}, for any $1\leqslant t \leqslant [{\frac{\nu-1}{2}}]$, $\zg_p^{2t+1}$ is an arc gives rise to one of the direct summands of $\tau_2^{t}M_{\zg_i}$, $0\leqslant i \leqslant s$. Note that these arcs separate $t^{-1}(\bbp)$ as $4$ -gons and possibly one $3$-gon. The potential $3$-gon exists only when $\nu$ is even. In this case, the $3$-gon is formed by $\zg_p^\mu$ with $\mu={\frac{\nu-2}{2}}$ and two boundary segments. Note that the set $t^{-1}(\bbp)\cup \{\zg_p^{2t+1}, 1\leqslant t \leqslant [{\frac{\nu-1}{2}}]\}$ is an admissible arc system.
Furthermore, since the weight of any arc in $\zD_I$ is equal to one, it must be of the form $\zg_i$ for some $0\leqslant i \leqslant s$ and some polygon $t^{-1}(\bbp)$.
When $\bbp$ extends all the polygons of $\zD^*$, we get all the arcs in the $\tau_2$-orbit of injective arcs, which form an admissible $4$-partial triangulation on $(\cals,\calm,\zD^*)$. 

The final admissible $4$-partial triangulation is maximal by Lemma \ref{lem:5partmax}, and thus gives rise to a maximal rigid module by Theorem \ref{thm:max-m2}.
\end{proof}

\subsection{A classification of $\tau_n$-finite gentle algebras}\label{section:taun-finite}
In this subsection, we classify the gentle algebras which are $\tau_n$-finite.
Assume that $A$ is a gentle algebra associated with a surface model $(\cals,\calm,\zD^*)$, whose global dimension is $n$.

\begin{definition}\label{def:tseq}
A \emph{$\tau_n$-sequence} is a walk $\sigma=\mathfrak{r}_1\overline{\mathfrak{t}}_1\mathfrak{r}_2\overline{\mathfrak{t}}_2\cdots\mathfrak{r}_m\overline{\mathfrak{t}}_m$, where $\mathfrak{r}_i$ is a direct walk consisting of $n$ arrows with relations at each vertex, and $\mathfrak{t}_i$ is a direct string for each $1\leqslant i \leqslant m$. 
A $\tau_n$-sequence is a \emph{$\tau_n$-cycle} if it is a cycle, that is, the starting vertex coincides with the ending vertex. If $\sigma$ is not a $\tau_n$-cycle, then we call $m$ the \emph{$\tau_n$-length} of $\sigma$, otherwise, the $\tau_n$-length is infinite.
The maximal $\tau_n$-length of the $\tau_n$-sequences in $A$ is called the \emph{$\tau_n$-length} of $A$. In particular, the $\tau_n$-length of $A$ is infinite if it contains a $\tau_n$-cycle.
\end{definition}

In the Definition, ${\mathfrak{t}}_i, 1\leqslant i \leqslant m,$ is allowed to be trivial.
The picture in Figure \ref{figlist10} gives a geometric explanation of a $\tau_n$-sequence $\sigma$, where each $\mathfrak{r}_i$ corresponds to a $(n+3)$-gon $\bbp_i$, and each $\overline{\mathfrak{t}}_i$ corresponds to a fan of $\rpoint$-arcs.
Denote by $\ell^*$ the start vertex of $\sigma$, which is the first arc in the figure, and denote by $\zg$ the anti-twist of $\ell^*$, which gives rise to an injective module $I=M_\zg$. Let $p_i$ be the $\bpoint$-point in $\bbp_i$ for each $1\leqslant i \leqslant m$. The arc $\zg_{p_i}^{n+1}$ gives rise to a module $M_{p_i}^{n+1}$, which is a direct summand of $\tau_n^i I$ by Lemma \ref{lem:geost}. Furthermore, a straightforward argument using these observations derives the following theorem.

\begin{figure}
\begin{tikzpicture}[>=stealth,scale=0.7]
\draw [line width=3pt,gray!60] (8.5,-0.05) to (-8.5,-0.05);
\draw [line width=1pt] (-8.5,0) to (8.5,0);
			
\draw [dark-green, line width=1pt] (-8.5,2) to (-6,0);
\draw [bend left, line width=1pt] (-6,0) to (-2,0);
\draw [bend left, line width=1pt] (4,0) to (8,0);
			
\draw [red,dotted, line width=1pt] (-4.4,2.1) to (-3.5,2.1);
\draw [red,dotted, line width=1pt] (5.6,2.1) to (6.4,2.1);
\draw [red,dotted, line width=1pt] (0.5,2.1) to (1.5,2.1);
\draw [bend right, red,line width=.6pt,->] (-3.7,1.3) to (-4.3,1.3);
\draw [bend right, red,line width=.6pt,->] (6.3,1.3) to(5.7,1.3) ;
			
\draw [red ,line width=1pt] (-8,0) to (-7.3,3);
\draw [red ,line width=1pt] (-4,0) to (-3.3,3);
\draw [red ,line width=1pt] (2,0) to (2.7,3);
			
\draw [red ,line width=1pt] (-4,0) to (-4.7,3);
\draw [red ,line width=1pt] (0,0) to (-0.7,3);
\draw [red ,line width=1pt] (5.3,3) to (6,0)to (6.7,3);
			
\draw [red ,dashed, line width=1pt](-7.3,3) to (-6,4.5) to (-4.7,3);
\draw [red ,dashed, line width=1pt] (-3.3,3) to(-2,4.5) to (-0.7,3);
\draw [red ,dashed, line width=1pt] (2.7,3) to(4,4.5) to (5.3,3);

\draw [bend right, red,line width=.6pt,->] (-7.4,2.5) to (-6.9,3.35);
\draw [bend right, red,line width=.6pt,->] (-3.4,2.5) to (-2.9,3.35);
\draw [bend right, red,line width=.6pt,->] (2.6,2.5) to (3.1,3.35);
\draw [bend right, red,line width=.6pt,->] (-6.4,4) to (-5.6,4);
\draw [bend right, red,line width=.6pt,->] (-2.4,4) to (-1.6,4);
\draw [bend right, red,line width=.6pt,->] (3.6,4) to (4.4,4);
\draw [bend right, red,line width=.6pt,->] (-5.1,3.4) to (-4.6,2.5);
\draw [bend right, red,line width=.6pt,->] (-1.1,3.4) to (-0.6,2.5);
\draw [bend right, red,line width=.6pt,->] (4.9,3.4) to (5.4,2.5);
            
            \draw[thick, fill=white] (-6,0)   circle (0.08);
			\draw[thick, fill=white] (-2,0)   circle (0.08);
			\draw[thick, fill=white] (4,0)   circle (0.08);
			\draw[thick, fill=white] (8,0)   circle (0.08);
			
			\draw[thick,red, fill=red] (-8,0)  circle (0.07);
			\draw[thick,red, fill=red] (-4,0)  circle (0.07);
			\draw[thick,red, fill=red] (0,0)   circle (0.07);
			\draw[thick,red, fill=red] (2,0)   circle (0.07);
			\draw[thick,red, fill=red] (6,0)   circle (0.07);
			
			\draw[thick,red, fill=red] (-7.3,3)  circle (0.07);
			\draw[thick,red, fill=red] (-4.7,3)  circle (0.07);
			\draw[thick,red, fill=red] (-6,4.5)   circle (0.07);
			\draw[thick,red, fill=red] (-3.3,3)  circle (0.07);
			\draw[thick,red, fill=red] (-0.7,3)  circle (0.07);
			\draw[thick,red, fill=red] (-2,4.5)   circle (0.07);
			\draw[thick,red, fill=red] (2.7,3)  circle (0.07);
			\draw[thick,red, fill=red] (5.3,3)  circle (0.07);
			\draw[thick,red, fill=red] (4,4.5)   circle (0.07);

\draw[red] (-4,3) node{$\overline{\mathfrak{t}}_{1}$};
\draw[red] (6,3) node{$\overline{\mathfrak{t}}_{m}$};
\draw (-6,2.5) node[red]{$\bbp_{1}$};
\draw (-2,2.5) node[red]{$\bbp_{2}$};
\draw (4,2.5) node[red]{$\bbp_{m}$};

\draw (-6,-.33) node{$p_{1}$};
\draw (-2,-.35) node{$p_{2}$};
\draw (4,-.35) node{$p_{m}$};

            \draw (-6.7,1) node[dark-green]{$\gamma$};
			\draw (-8.1,0.8) node[red]{$\ell^{*}$};
			\draw (-2.8,.8) node{$\gamma^{n}_{p_{1}}$};
			\draw (7.5,.8) node{$\gamma^{n}_{p_{m}}$};
\end{tikzpicture}
\centering 
\caption{A geometric explanation of a $\tau_n$-sequence $\sigma=\mathfrak{r}_1\overline{\mathfrak{t}}_1\mathfrak{r}_2\overline{\mathfrak{t}}_2\cdots\mathfrak{r}_m\overline{\mathfrak{t}}_m$, where each $\mathfrak{r}_i$ is a walk with full relations corresponds to a polygon $\bbp_i$, and each $\mathfrak{t}_i$ is a string corresponds to a fan of oriented intersections arising from a common endpoint.}\label{figlist10}
	\end{figure}
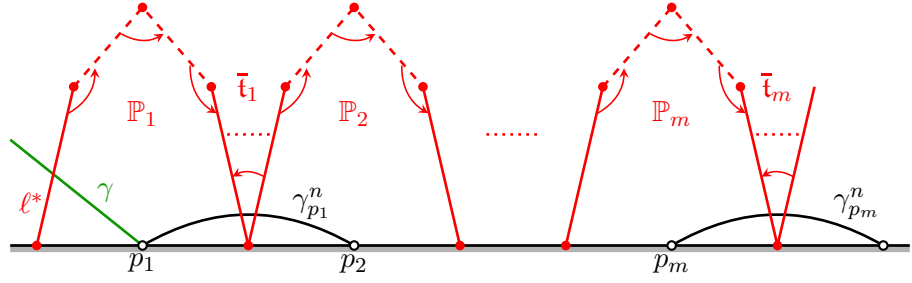
    
\begin{theorem}\label{thm:taufinite}
The $\tau_n$-dimension of a gentle algebra $A$ equals the $\tau_n$-length of $A$. In particular, it is $\tau_n$-finite if and only if there is no $\tau_n$-cycle in $A$.
\end{theorem}

\begin{example}
See the examples in Figure \ref{fig:ex} of $\tau_n$-finite and $\tau_n$-infinite gentle algebras for $n=2$, where the first two algebras are $\tau_2$-infinite and the last three algebras are $\tau_2$-finite.

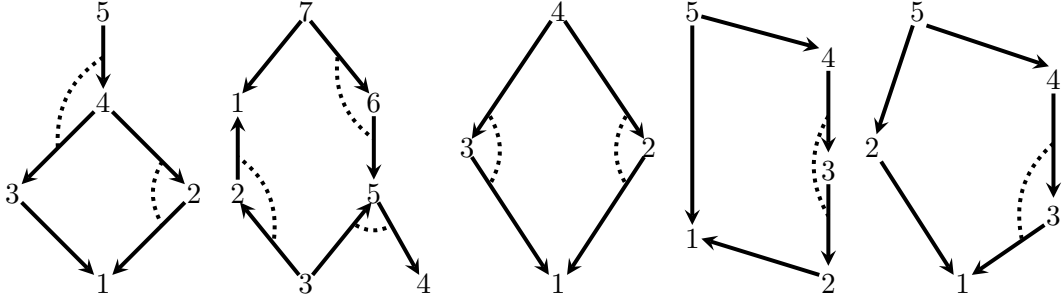
\begin{figure}
\begin{tikzpicture}[>=stealth,scale=0.6]
			
			\draw [line width=1.5pt,->] (0,5.7) to (0,4.3);
			\draw [line width=1.5pt,->] (-0.2,3.8) to (-1.8,2.2);
			\draw [line width=1.5pt,->] (0.2,3.8) to (1.8,2.2);
			\draw [line width=1.5pt,->] (1.8,1.8) to (0.2,0.2);
			\draw [line width=1.5pt,->] (-1.8,1.8) to (-0.2,0.2);
			\draw [bend right, line width=1.5pt,dotted] (0,5) to (-1,3);
			\draw [bend right, line width=1.5pt,dotted] (1.3,2.7) to (1.3,1.3);

			\draw (0,6) node{$5$};
			\draw (0,4) node{$4$};
			\draw (-2,2) node{$3$};
			\draw (2,2) node{$2$};
			\draw (0,0) node{$1$};
			
		\end{tikzpicture}
		\begin{tikzpicture}[>=stealth,scale=0.6]
			
			\draw [line width=1.5pt,->] (0.1,5.8) to (1.4,4.2);
			\draw [line width=1.5pt,->] (-0.1,5.8) to (-1.4,4.2);
			\draw [line width=1.5pt,->] (-1.5,2.3) to (-1.5,3.7);
			\draw [line width=1.5pt,->]  (1.5,3.7) to(1.5,2.3);
			\draw [line width=1.5pt,->]  (0.15,0.2) to(1.45,1.8);
			\draw [line width=1.5pt,->]  (-0.15,0.2) to(-1.45,1.8);
			\draw [line width=1.5pt,->]  (1.6,1.8) to(2.5,0.2);

			\draw [bend right, line width=1.5pt,dotted] (-0.7,1) to (-1.5,2.8);
			\draw [bend right, line width=1.5pt,dotted] (0.7,5) to (1.5,3.2);
			\draw [bend right, line width=1.5pt,dotted] (1.1,1.3) to (1.9,1.3);

			\draw (0,6) node{$7$};
			\draw (1.5,4) node{$6$};
			\draw (-1.5,4) node{$1$};
			\draw (-1.5,2) node{$2$};
			\draw (1.5,2) node{$5$};
			\draw (0,0) node{$3$};
			\draw (2.6,0) node{$4$};
		\end{tikzpicture}
		\begin{tikzpicture}[>=stealth,scale=0.6]
			
			\draw [line width=1.5pt,->] (-0.15,5.8) to (-1.9,3.2);
			\draw [line width=1.5pt,->] (0.15,5.8) to (1.9,3.2);
			\draw [line width=1.5pt,->] (-1.9,2.8) to (-0.2,0.2);
			\draw [line width=1.5pt,->] (1.9,2.8) to (0.2,0.2);
			\draw [bend right, line width=1.5pt,dotted] (1.5,3.7) to (1.5,2.2);
	    	\draw [bend left, line width=1.5pt,dotted] (-1.5,3.7) to (-1.5,2.2);

			\draw (0,6) node{$4$};
			\draw (-2,3) node{$3$};
			\draw (2,3) node{$2$};
			\draw (0,0) node{$1$};
			
		\end{tikzpicture}
		\begin{tikzpicture}[>=stealth,scale=0.6]
			
			\draw [line width=1.5pt,->] (-2.8,5.9) to (-0.2,5.2);
			\draw [line width=1.5pt,->] (-3,5.7) to (-3,1.3);
			\draw [line width=1.5pt,->] (0,4.7) to (0,2.9);
			\draw [line width=1.5pt,->] (0,2.2) to (0,0.4);
			\draw [line width=1.5pt,->] (-0.2,0.2) to (-2.8,1);
			
			\draw [bend right, line width=1.5pt,dotted] (0,3.7) to (0,1.5);

			\draw (-3,6) node{$5$};
			\draw (-3,1) node{$1$};
			\draw (0,0) node{$2$};
			\draw (0,2.5) node{$3$};
			\draw (0,5) node{$4$};
		\end{tikzpicture}
		\begin{tikzpicture}[>=stealth,scale=0.6]
			
			\draw [line width=1.5pt,->] (-1.1,5.7) to (-1.9,3.3);
			\draw [line width=1.5pt,->] (-0.8,5.7) to (1.9,4.8);
			\draw [line width=1.5pt,->] (2,4.2) to (2,1.9);
			\draw [line width=1.5pt,->] (1.8,1.3) to (0.2,0.2);
			\draw [line width=1.5pt,->] (-1.8,2.7) to (-0.2,0.2);
			\draw [bend right, line width=1.5pt,dotted] (2,3.1) to (1.3,1);
			
			\draw (-1,6) node{$5$};
			\draw (2,4.5) node{$4$};
			\draw (2,1.5) node{$3$};
			\draw (-2,3) node{$2$};
			\draw (0,0) node{$1$};
\end{tikzpicture}
\centering
\caption{Five gentle algebras with global dimension two. The first two algebras are $\tau_2$-infinite and the last three are $\tau_2$-finite.}\label{fig:ex}
\end{figure}    
\end{example}

\subsection{A classification of $n$-complete gentle algebras}\label{section:n-complete}

Let $A$ be a gentle algebra associated with a quiver $Q$. We call the number of arrows adjacent to a vertex $\ell^*$ in $Q$ the \emph{degree} of $\ell^*$. We have the following classification of $n$-complete gentle algebras.
\begin{theorem}\label{thm:ncomp}
A gentle algebra $A$ with $\gd A=n$ is $n$-complete if and only if for any walk $\sigma=\aaa_1\cdots\aaa_{n}$ in $A$ with full relations, the degree of the source is one.
\end{theorem}
The proof of the theorem is separated into several lemmas.
\begin{lemma}\label{lem:comp1}
Let $A$ be a gentle algebra with $\gd A=n$.
Assume that there are $s$ walks $\sigma_j=\aaa^j_1\cdots\aaa^j_{n}, 1\leqslant j \leqslant s$, in $A$ with full relations.
If $A$ is $n$-complete, then the degree of the source of $\sigma_j$ is one.
\end{lemma}
\begin{proof}
Note that each walk $\sigma_j$ corresponds to a $(n+3)$-gon $\bbp_j$ in $\zD^*$, which contains a $\bpoint$-point $p_j$, see the picture in Figure \ref{fig:ncom1}. 
By Lemma \ref{lem:geost}, the $\tau_n$-closure of injective modules is
\[\mathfrak{M}:=\add\{\tau_n^i(DA)\ |\ i\ge0\}=\add\{DA,M_{p_j}^{n+1}, 1\leqslant j \leqslant s \}.\]
Denote by $\ell_j^*$ the source of $\sigma_j$, and by $I_j$ the injective module corresponds to $\ell_j^*$. Then $M_{p_j}^{n+1}$ is a direct summand of $\tau_n I_j$. In particular, $\tau_n I_j$ is nonzero and therefore $I_j$ belongs to $\mathfrak{M}_P$.
\begin{figure} 
\centering 
\begin{tikzpicture}[>=stealth,scale=.6]
\draw [ line width=3pt, gray!60] (8,-0.1) to (-5.5,-0.1);
\draw [ line width=1.5pt] (8,0) to (-5.5,0);
			
\draw [ line width=3pt, gray!60] (-2.04,10.07) to (-5.04,9.07);
			\draw [ line width=3pt, gray!60] (2.04,10.07) to (5.04,9.07);
			\draw [ line width=3pt, gray!60] (-6.05,8.04) to (-8.05,5.54);
			\draw [ line width=3pt, gray!60] (6.05,8.04) to (8.05,5.54);
			\draw [ line width=3pt, gray!60] (-7.58,3.48) to (-6.58,0.49);
			\draw [ line width=3pt, gray!60] (8.08,4.98) to (7.08,1.98);
			\draw [bend right,  line width=3pt, gray!60] (-6.6,0.5) to (-5.48,-0.11);
			
			\draw [ line width=1.5pt] (-2,10) to (-5,9);
			\draw [ line width=1.5pt] (2,10) to (5,9);
			\draw [ line width=1.5pt] (-6,8) to (-8,5.5);
			\draw [ line width=1.5pt] (6,8) to (8,5.5);
			\draw [ line width=1.5pt] (-7.5,3.5) to (-6.5,0.5);
			\draw [ line width=1.5pt] (8,5) to (7,2);
			\draw [ bend right,  line width=1.5pt] (-6.5,0.5) to (-5.5,0);
			
		    \draw [bend left,  line width=1.5pt] (0,0) to (4.5,0) ;
			
			\draw [  line width=1pt, red,dotted ] (-6.8,4) to (-7.4,4.7) ;
			\draw [  line width=1pt, red,dotted ] (-5.3,8.5) to (-5.4,7.6) ;
			\draw [  line width=1pt, red,dotted ] (0.6,9.6) to (-0.2,8.8) ;
			\draw [  line width=1pt, red,dotted ] (4.9,8.2) to (5.5,8.2) ;
			\draw [  line width=1pt, red,dotted ] (7.1,5.1) to (7.6,5.1) ;
			\draw [  line width=1pt, red,dotted ] (3.5,1.8) to (3.8,1.15) ;

			\draw [bend right,  line width=1pt, red ] (-4,0) to (-7.2,2.7) ;
			\draw [bend left,  line width=1pt, red ] (2,0) to (7.3,3) ;
			\draw [bend left,  line width=1pt, red ] (2,0) to (7.3,1) ;
			\draw [bend left,  line width=1pt, red ] (7.3,3) to  (7.4,6.25) ;
			\draw [bend left,  line width=1pt, red ] (7.3,3) to  (8.2,5.4) ;
			\draw [bend left,  line width=1pt, red ] (7.4,6.25) to  (4.2,9.28);
			\draw [bend left,  line width=1pt, red] (7.4,6.25) to  (5.6,8.7);
			\draw [bend left,  line width=1pt, red] (4.2,9.28) to (-2.8,9.73) ;
			\draw [bend left,  line width=1pt, red] (-2.8,9.73) to  (-6.58,7.3);
			\draw [bend left,  line width=1pt, red] (-6.58,7.3) to (-7.2,2.7) ;
			\draw [bend left,  line width=1pt, red] (4.2,9.28) to (0.5,10) ;
			\draw [bend left,  line width=1pt, red] (-2.8,9.73) to  (-5.8,8.5);
			\draw [bend left,  line width=1pt, red] (-6.58,7.3) to (-8,4.5) ;

			\draw [bend left,  line width=1pt, dark-green ] (7.75,4.25) to (6.58,7.3) ;
			\draw [bend right,  line width=1pt, dark-green ] (0,0) to (-6.65,1) ;
			\draw [bend left,  line width=1pt, dark-green ] (4.5,0) to (7.75,4.25) ;
			\draw [bend right,  line width=1pt, dark-green ] (2.8,9.73) to (6.58,7.3);
			\draw [bend left,  line width=1pt, dark-green ] (2.8,9.73) to (-4.2,9.28) ;
			\draw [bend left,  line width=1pt, dark-green ] (-4.2,9.28) to (-7.4,6.25) ;
			\draw [bend left,  line width=1pt, dark-green ] (-7.4,6.25) to (-6.65,1) ;

			\draw[thick,fill=white] (0,0) circle (0.1);
			\draw[thick,fill=white] (4.5,0)  circle (0.1);
			\draw[thick,red, fill=red] (-4,0)  circle (0.1);
			\draw[thick,red, fill=red] (2,0)  circle (0.1);

			\draw[thick,fill=white] (-4.2,9.28) circle (0.1);
			\draw[thick,red, fill=red] (-2.8,9.73)  circle (0.1);
			\draw[thick,red, fill=red] (4.2,9.28) circle (0.1);
			\draw[thick,fill=white] (2.8,9.73)  circle (0.1);
			
			\draw[thick,fill=white] (-7.4,6.25) circle (0.1);
			\draw[thick,red, fill=red] (-6.58,7.3)  circle (0.1);
			\draw[thick,red, fill=red] (7.4,6.25) circle (0.1);
			\draw[thick,fill=white] (6.58,7.3)  circle (0.1);
			
			\draw[thick,red, fill=red] (-7.2,2.7) circle (0.1);
			\draw[thick,fill=white] (-6.65,1)  circle (0.1);
			\draw[thick,fill=white] (7.75,4.25) circle (0.1);
			\draw[thick,red, fill=red] (7.3,3)  circle (0.1);

			\draw (1.5,1.1) node {$\gamma^{n+1}_{p_{j}}$};
			\draw (0,-0.6) node {$p_{j}$};
			\draw (-2.7,1.8) node[dark-green] {$t^{-1}(\ell^{*}_{j})$};
			\draw (0,6.5) node[dark-green] {$t^{-1}(\bbp_{j})$};
			\draw (0,5) node[red] {$\bbp_{j}$};
			\draw (-5.1,2.2) node[red] {$\ell^{*}_{j}$};
\end{tikzpicture}
\caption{The polygon $\bbp_j$ corresponds to a walk $\sigma_j=\aaa^j_1\cdots\aaa^j_{n}$. Note that if the algebra $A$ is $n$-complete, then by Lemma \ref{lem:comp1}, the source $\ell_j^*$ of $\sigma_j$ has degree one. Therefore, $\ell^*_j$ is isotopic to a boundary segment with one $\bpoint$-point.} 
	\label{fig:ncom1} 
\end{figure}
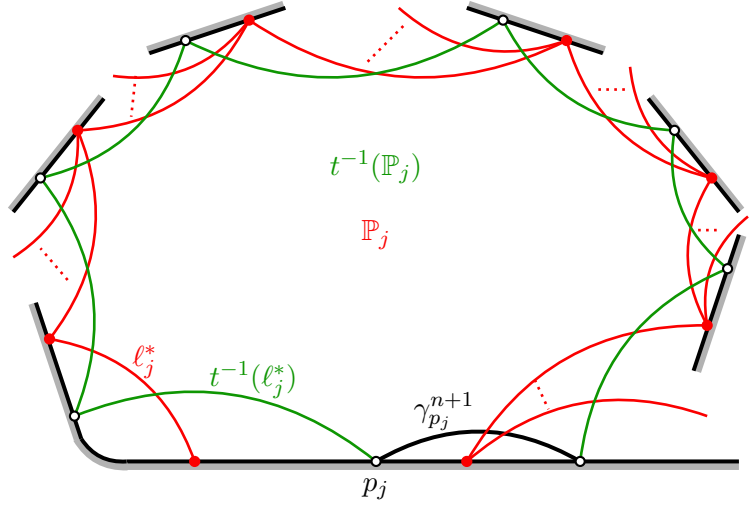

Suppose that an arrow $\aaa$ ends at $\ell^*_j$ for some $1 \leqslant j \leqslant s$. Note that we have $\aaa\aaa_1\neq 0$, otherwise $\aaa\sigma_j$ is a walk with full relations, and thus $\gd(A)\geqslant n+1$, which is a contradiction. But then, there is a non-trivial extension between $I_j$ and $P_j$, where $P_j$ is the indecomposable projective module associated with $\ell^*_j$. That is, $\Ext^1(I_j,P_j)\neq 0$, which contradicts the third condition $(C_n)$ in Definition \ref{complete} of a $n$-complete algebra.
We have shown that there does not exist an arrow ending at the source of $\sigma_j$.

Suppose that there exists $\sigma_k, 1\leqslant k \leqslant s$, such that except for $\aaa^k_1$, there is an extra arrow $\bbb$ starting at the source $\ell^*_k$. Denote by $\ell^*$ the endpoint of $\bbb$, and $I$ the injective module associated with $\ell^*$. 
Then $t^{-1}(\ell^*_k)$ and $t^{-1}(\ell^*)$ share a common endpoint, that is the unique $\bpoint$-point in $\bbp_k$. In particular, $\tau_nI_k$ and $\tau_nI$ share a common direct summand $M_{p_k}^{n+1}$ and thus $I\in \mathfrak{M}_P$.
Therefore $$\mathfrak{P}(\mathfrak{M})=\mathfrak{M}\setminus\mathfrak{M}_P\subseteq \add\{DA,M_{p_j}^n, 1\leqslant j \leqslant s \}\setminus \{I, I_j, 1\leqslant j \leqslant s\}.$$
On the other hand, by the argument in the first paragraph, no arrow goes to the source of $\sigma_j$, therefore $\ell^*\neq \ell^*_j$ for any $1\leqslant j \leqslant s$. Thus, the cardinality of the right set is strictly smaller than $n$, and $\mathfrak{P}(\mathfrak{M})$ cannot be given by a tilting module since the cardinality of a tilting module in $\ma$ is $n$. This contradicts the first condition $(A_n)$ in the definition of $n$-complete algebras. 

In summary, for each walk $\sigma_j$, the degree of the source is one.
\end{proof}
\begin{lemma}\label{lem:3.12}
Let $A$ be a gentle algebra with $\gd A=n$.
Assume that there are $s$ walks $\sigma_j=\aaa^j_1\cdots\aaa^j_{n}, 1\leqslant j \leqslant s$, in $A$ with full relations, and the degree of each source $\ell^*_j$ is one.
Denote by $p_j$ the $\bpoint$-point in the $(n+3)$-gon $\bbp_j$ of $\zD^*$ corresponding to $\sigma_j$. Then the sets of arcs correspond to $\mathfrak{M}$, $\mathfrak{P}(\mathfrak{M})$ and $T^\perp$ are respectively 
\[\calp_{\mathfrak{M}}=\zD_I\cup \{\zg_{p_j}^{n+1}, 1\leqslant j \leqslant s\},\]
\[\zD_T=\zD_I\setminus\{t^{-1}(\ell^*_j), 1\leqslant j \leqslant s\}\cup \{\zg_{p_j}^{n+1}, 1\leqslant j \leqslant s\},\]
\[\calp_{T^\perp}=\zD_I\cup\{\zg_{p_j}^m, 1\leqslant j \leqslant s, 1\leqslant m \leqslant n+1\},\]
where $T$ is a tilting module with $\mathfrak{P}(\mathfrak{M})=\add T$.
\end{lemma}
\begin{proof}
Denote by $I_j$ the injective module associated with $\ell^*_j$. Since the degree of $\ell_j^*$ is one, we have
$\tau_n I_j=M_{p_j}^{n+1}$ for each $1\leqslant j \leqslant s$, and $\tau_nI_t=0$ for any $I_t$ different from $I_j$. Therefore we have 
$$\mathfrak{M}=\add\{DA,M_{p_j}^{n+1}, 1\leqslant j \leqslant s\},$$
and \[\calp_{\mathfrak{M}}=\zD_I\cup \{\zg_{p_j}^{n+1}, 1\leqslant j \leqslant s\},\]
which is $(n+1)$ partial triangulation.

Let $\mathfrak{P}(\mathfrak{M})=\add T$. Then 
$$T=(\oplus I_t)\oplus(\oplus_{j=1}^sM_{p_j}^{n+1}),$$ where $I_t$ extends all indecomposable injective modules that are different from $I_j, 1\leqslant j \leqslant s$. 
Therefore
$$\zD_T=\zD_I\setminus\{t^{-1}(\ell^*_j), 1\leqslant j \leqslant s\}\cup \{\zg_{p_j}^{n+1}, 1\leqslant j \leqslant s\},$$
which is an admissible dissection.
Furthermore, it is not difficult to check that the weight of each oriented intersection between two arcs in $\zD_T$ is equal to zero, and it is a tilting dissection introduced in \cite{C24}, which gives rise to a tilting module $T$.

Now we prove 
$$\calp_{T^\perp}=\zD_I\cup\{\zg_{p_j}^m, 1\leqslant j \leqslant s, 1\leqslant m \leqslant n+1\}.$$ Note that $\zD_I$ and $\{\zg_{p_j}^m, 1\leqslant j \leqslant s, 1\leqslant m \leqslant n+1\}$ have common elements, where $\zg_{p_j}^1=t^{-1}(\ell^*_j)\in \zD_I$. Let $M_\za\in T^\perp$. Suppose that there is an interior intersection between $\za$ and some $\zb\in \zD_T$. Since each $\zg^{n+1}_{p_j}\in\zD_T$ forms an external triangle that contains no $\bpoint$-point, $\za$ has no interior intersection with it. Therefore, $\zb$ belongs to $\zD_I$. Note that this interior intersection gives rise to a non-zero once extension between $M_\za$ and $M_\zb$. Since $M_\zb$ is an injective module and $\Ext^1(M_\za,M_\zb)=0$, this extension belongs to $\Ext^1(M_\zb,M_\za)$. However, we have $\Ext^1(M_\zb,M_\za)=0$, since $M_\za\in T^\perp$ and $M_\zb\in \add T$.

We have shown that any arc $\za$ in $\calp_{T^\perp}$ has no interior intersection with arcs in $\zD_T$. Thus, each arc $\za$ in $\calp_{T^\perp}$ belongs to some polygon $\bbp$ of $\zD_T$. Assume that $\bbp$ is a $(\nu+3)$-gon. There are two types of such polygons, which are called polygons of type (I) and of type (II), see Figure \ref{fig:ncom2}. The polygon of type (I) is a $(\nu+3)$-gon with $\nu\leqslant n-1$, and the polygon of type (II) is a $(\nu+3)$-gon with $\nu=n$, that is, $\bbp_j$. Note that for both cases, any zigzag arc in $\bbp$ is of the form $\zg_p^{m}$ for $1 \leqslant m \leqslant \nu+1$, where $p$ is the (unique) $\bpoint$-point in $\bbp$.

Since in a polygon of type (I), the arc $\zg_p^1=t^{-1}(\ell^*)$ belongs to $\zD_T$, and the weight of the oriented intersection from $\zg_p^1$ to $\zg_p^m$ equals $m-1\leqslant n$. Thus the space $\Ext^{m-1}(M_p^1,M_p^m)$ is non-zero, and $\zg_p^m$ does not belong to $\calp_{T^\perp}$ for any $m\geqslant 2$. On the other hand, for a polygon of type (II), it can be checked directly that $\zg_{p_j}^m$ belongs to $\calp_{T^\perp}$ for any $1 \leqslant j \leqslant s$ and $1\leqslant m \leqslant n+1$. Therefore we have
$$\calp_{T^\perp}=\zD_I\cup\{\zg_{p_j}^m, 1\leqslant j \leqslant s, 1\leqslant m \leqslant n+1\},$$
which is an admissible arc system.
\begin{figure} 
	\centering 
\begin{tikzpicture}[>=stealth,scale=.6]
\draw [ line width=3pt, gray!60] (8,-0.1) to (-8,-0.1);
			\draw [ line width=1.5pt] (8,0) to (-8,0);
			
			\draw [ line width=3pt, gray!60] (-2.04,10.07) to (-5.04,9.07);
			\draw [ line width=3pt, gray!60] (2.04,10.07) to (5.04,9.07);
			\draw [ line width=3pt, gray!60] (-6.05,8.04) to (-8.05,5.54);
			\draw [ line width=3pt, gray!60] (6.05,8.04) to (8.05,5.54);
			\draw [ line width=3pt, gray!60] (-8.08,4.98) to (-7.08,1.98);
			\draw [ line width=3pt, gray!60] (8.08,4.98) to (7.08,1.98);
		    
			\draw [ line width=1.5pt] (-2,10) to (-5,9);
			\draw [ line width=1.5pt] (2,10) to (5,9);
			\draw [ line width=1.5pt] (-6,8) to (-8,5.5);
			\draw [ line width=1.5pt] (6,8) to (8,5.5);
			\draw [ line width=1.5pt] (-8,5) to (-7,2);
			\draw [ line width=1.5pt] (8,5) to (7,2);

			\draw [  line width=1pt, red,dotted ] (-7.6,5.3) to (-7.4,4.7) ;
			\draw [  line width=1pt, red,dotted ] (-5.3,8.5) to (-5.4,7.6) ;
			\draw [  line width=1pt, red,dotted ] (0.6,9.6) to (-0.2,8.8) ;
			\draw [  line width=1pt, red,dotted ] (4.9,8.2) to (5.5,8.2) ;
			\draw [  line width=1pt, red,dotted ] (7.1,5.1) to (7.6,5.1) ;
			\draw [  line width=1pt, red,dotted ] (3.5,1.8) to (3.8,1.15) ;
			\draw [  line width=1pt, red,dotted ] (-5.2,1.95) to (-6,1.915) ;

			\draw [bend left,   line width=1pt ] (0,0) to (4.5,0) ;
			\draw[cyan,line width=1pt,black] (0,0) to[out=45,in=-170](7.75,4.25);
			\draw[cyan,line width=1pt,black] (0,0) to[out=65,in=-150](6.58,7.3);
			\draw [line width=1pt ] (0,0) to (2.8,9.73) ;
			\draw[cyan,line width=1pt,black] (0,0) to[out=90,in=-35](-4.2,9.28) ;
			\draw[cyan,line width=1pt,black] (0,0) to[out=105,in=-15](-7.4,6.25);
		
			\draw [bend right,  line width=1pt, red ] (-4,0) to (-7.75,4.25) ;
			\draw [bend left,  line width=1pt, red ] (2,0) to (7.3,3) ;
			\draw [bend left,  line width=1pt, red ] (2,0) to (7.3,1) ;
			\draw [bend left,  line width=1pt, red ] (7.3,3) to  (7.4,6.25) ;
			\draw [bend left,  line width=1pt, red ] (7.3,3) to  (8.2,5.4) ;
			\draw [bend left,  line width=1pt, red ] (7.4,6.25) to  (4.2,9.28);
			\draw [bend left,  line width=1pt, red] (7.4,6.25) to  (5.6,8.7);
			\draw [bend left,  line width=1pt, red] (4.2,9.28) to (-2.8,9.73) ;
			\draw [bend left,  line width=1pt, red] (-2.8,9.73) to  (-6.58,7.3);
			\draw [bend left,  line width=1pt, red] (-6.58,7.3) to (-7.75,4.25) ;
			\draw [bend left,  line width=1pt, red] (4.2,9.28) to (0.5,10) ;
			\draw [bend left,  line width=1pt, red] (-2.8,9.73) to  (-5.8,8.5);
			\draw [bend left,  line width=1pt, red] (-6.58,7.3) to (-8,5.2) ;
			\draw [bend left,  line width=1pt, red] (-7.75,4.25) to (-6.5,1.5) ;

			\draw [bend left,  line width=1.5pt, blue ] (7.75,4.25) to (6.58,7.3) ;
			\draw [bend right,  line width=1pt, blue ] (0,0) to (-7.3,3) ;
			\draw [bend left,  line width=1pt, blue ] (4.5,0) to (7.75,4.25) ;
			\draw [bend right,  line width=1pt, blue] (2.8,9.73) to (6.58,7.3);
			\draw [bend left,  line width=1pt, blue ] (2.8,9.73) to (-4.2,9.28) ;
			\draw [bend left,  line width=1pt, blue ] (-4.2,9.28) to (-7.4,6.25) ;
			\draw [bend left,  line width=1pt, blue ] (-7.4,6.25) to (-7.3,3) ;

			\draw[thick,fill=white] (0,0) circle (0.1);
			\draw[thick,fill=white] (4.5,0)  circle (0.1);
			\draw[thick,red, fill=red] (-4,0)  circle (0.1);
			\draw[thick,red, fill=red] (2,0)  circle (0.1);

			\draw[thick,fill=white] (-4.2,9.28) circle (0.1);
			\draw[thick,red, fill=red] (-2.8,9.73)  circle (0.1);
			\draw[thick,red, fill=red] (4.2,9.28) circle (0.1);
			\draw[thick,fill=white] (2.8,9.73)  circle (0.1);
			
			\draw[thick,fill=white] (-7.4,6.25) circle (0.1);
			\draw[thick,red, fill=red] (-6.58,7.3)  circle (0.1);
			\draw[thick,red, fill=red] (7.4,6.25) circle (0.1);
			\draw[thick,fill=white] (6.58,7.3)  circle (0.1);
			
			\draw[thick,red, fill=red] (-7.75,4.25) circle (0.1);
			\draw[thick,fill=white] (-7.3,3)  circle (0.1);
			\draw[thick,fill=white] (7.75,4.25) circle (0.1);
			\draw[thick,red, fill=red] (7.3,3)  circle (0.1);

\draw (-4.7,1) node[red] {$\ell^{*}$};
\draw (-3.3,1.8) node[blue] {$ t^{-1}(\ell^{*})$};
\draw (1.1,5.5) node {$ \gamma^{m}_{p}$};
\draw (2.2,1.1) node {$ \gamma^{\nu+1}_{p}$};
\draw (0,-0.5) node {$ p$};
\end{tikzpicture}
\begin{tikzpicture}[>=stealth,scale=.6]
\draw [ line width=3pt, gray!60] (8,-0.1) to (-5.5,-0.1);
			\draw [ line width=1.5pt] (8,0) to (-5.5,0);
			
			\draw [ line width=3pt, gray!60] (-2.04,10.07) to (-5.04,9.07);
			\draw [ line width=3pt, gray!60] (2.04,10.07) to (5.04,9.07);
			\draw [ line width=3pt, gray!60] (-6.05,8.04) to (-8.05,5.54);
			\draw [ line width=3pt, gray!60] (6.05,8.04) to (8.05,5.54);
			\draw [ line width=3pt, gray!60] (-7.58,3.48) to (-6.58,0.49);
			\draw [ line width=3pt, gray!60] (8.08,4.98) to (7.08,1.98);
			\draw [bend right,  line width=3pt, gray!60] (-6.6,0.5) to (-5.48,-0.11);
			
			\draw [ line width=1.5pt] (-2,10) to (-5,9);
			\draw [ line width=1.5pt] (2,10) to (5,9);
			\draw [ line width=1.5pt] (-6,8) to (-8,5.5);
			\draw [ line width=1.5pt] (6,8) to (8,5.5);
			\draw [ line width=1.5pt] (-7.5,3.5) to (-6.5,0.5);
			\draw [ line width=1.5pt] (8,5) to (7,2);
			\draw [ bend right,  line width=1.5pt] (-6.5,0.5) to (-5.5,0);

			\draw [  line width=1pt, red,dotted ] (-6.8,4) to (-7.4,4.7) ;
			\draw [  line width=1pt, red,dotted ] (-5.3,8.5) to (-5.4,7.6) ;
			\draw [  line width=1pt, red,dotted ] (0.6,9.6) to (-0.2,8.8) ;
			\draw [  line width=1pt, red,dotted ] (4.9,8.2) to (5.5,8.2) ;
			\draw [  line width=1pt, red,dotted ] (7.1,5.1) to (7.6,5.1) ;
			\draw [  line width=1pt, red,dotted ] (3.5,1.8) to (3.8,1.15) ;

			\draw [bend right,  line width=1pt, red ] (-4,0) to (-7.2,2.7) ;
			\draw [bend left,  line width=1pt, red ] (2,0) to (7.3,3) ;
			\draw [bend left,  line width=1pt, red ] (2,0) to (7.3,1) ;
			\draw [bend left,  line width=1pt, red ] (7.3,3) to  (7.4,6.25) ;
			\draw [bend left,  line width=1pt, red ] (7.3,3) to  (8.2,5.4) ;
			\draw [bend left,  line width=1pt, red ] (7.4,6.25) to  (4.2,9.28);
			\draw [bend left,  line width=1pt, red] (7.4,6.25) to  (5.6,8.7);
			\draw [bend left,  line width=1pt, red] (4.2,9.28) to (-2.8,9.73) ;
			\draw [bend left,  line width=1pt, red] (-2.8,9.73) to  (-6.58,7.3);
			\draw [bend left,  line width=1pt, red] (-6.58,7.3) to (-7.2,2.7) ;
			\draw [bend left,  line width=1pt, red] (4.2,9.28) to (0.5,10) ;
			\draw [bend left,  line width=1pt, red] (-2.8,9.73) to  (-5.8,8.5);
			\draw [bend left,  line width=1pt, red] (-6.58,7.3) to (-8,4.5) ;

			\draw [bend left,  line width=1pt, blue ] (7.75,4.25) to (6.58,7.3) ;
			\draw [bend left,  line width=1pt, blue  ] (4.5,0) to (7.75,4.25) ;
			\draw [bend right,  line width=1pt, blue  ] (2.8,9.73) to (6.58,7.3);
			\draw [bend left,  line width=1pt, blue  ] (2.8,9.73) to (-4.2,9.28) ;
			\draw [bend left,  line width=1pt, blue  ] (-4.2,9.28) to (-7.4,6.25) ;
			\draw [bend left,  line width=1pt, blue  ] (-7.4,6.25) to (-6.65,1) ;
			\draw [bend left,  line width=1pt, blue ] (0,0) to (4.5,0) ;
			\draw [bend right,  line width=1pt, black  ] (0,0) to (-6.65,1) ;

			\draw[cyan,line width=1pt,black] (0,0) to[out=45,in=-170](7.75,4.25);
			\draw[cyan,line width=1pt,black] (0,0) to[out=65,in=-150](6.58,7.3);
			\draw [line width=1pt ] (0,0) to (2.8,9.73) ;
			\draw[cyan,line width=1pt,black] (0,0) to[out=90,in=-35](-4.2,9.28) ;
			\draw[cyan,line width=1pt,black] (0,0) to[out=105,in=-15](-7.4,6.25);

			\draw[thick,fill=white] (0,0) circle (0.1);
			\draw[thick,fill=white] (4.5,0)  circle (0.1);
			\draw[thick,red, fill=red] (-4,0)  circle (0.1);
			\draw[thick,red, fill=red] (2,0)  circle (0.1);

			\draw[thick,fill=white] (-4.2,9.28) circle (0.1);
			\draw[thick,red, fill=red] (-2.8,9.73)  circle (0.1);
			\draw[thick,red, fill=red] (4.2,9.28) circle (0.1);
			\draw[thick,fill=white] (2.8,9.73)  circle (0.1);
			
			\draw[thick,fill=white] (-7.4,6.25) circle (0.1);
			\draw[thick,red, fill=red] (-6.58,7.3)  circle (0.1);
			\draw[thick,red, fill=red] (7.4,6.25) circle (0.1);
			\draw[thick,fill=white] (6.58,7.3)  circle (0.1);
			
			\draw[thick,red, fill=red] (-7.2,2.7) circle (0.1);
			\draw[thick,fill=white] (-6.65,1)  circle (0.1);
			\draw[thick,fill=white] (7.75,4.25) circle (0.1);
			\draw[thick,red, fill=red] (7.3,3)  circle (0.1);		
\draw (1,5.3) node{$\gamma^{m}_{p_{j}}$};
\draw (2,1.1) node[blue ]{$\gamma^{n+1}_{p_{j}}$};
\draw (0,-0.6) node{$p =p_{j}$};
\draw (-5.5,2.6) node[red] {$\ell^{*}_{j}$};
\draw (-3.2,1.9) node {$ t^{-1}(\ell^{*}_{j})$};
\draw (-7.5,1) node {$p'_{j}$};
\end{tikzpicture}
\caption{Two types of polygons $\bbp$ in $\zD_T$, where the arcs in $\bbp$ are colored blue. The above picture is called type (I), which is a $(\nu+3)$-gon formed by arcs in $\zD_I$, where $1\leqslant \nu \leqslant n$. The picture below is called type (II), which is a $(n+3)$-gon obtained from a $(n+3)$-gon $t^{-1}(\bbp_j), 1\leqslant j \leqslant s,$ in $\zD_I$, after replacing $t^{-1}(\ell^*_j)$ by $\zg_{p_j}^{n+1}$.} 
	\label{fig:ncom2} 
\end{figure}
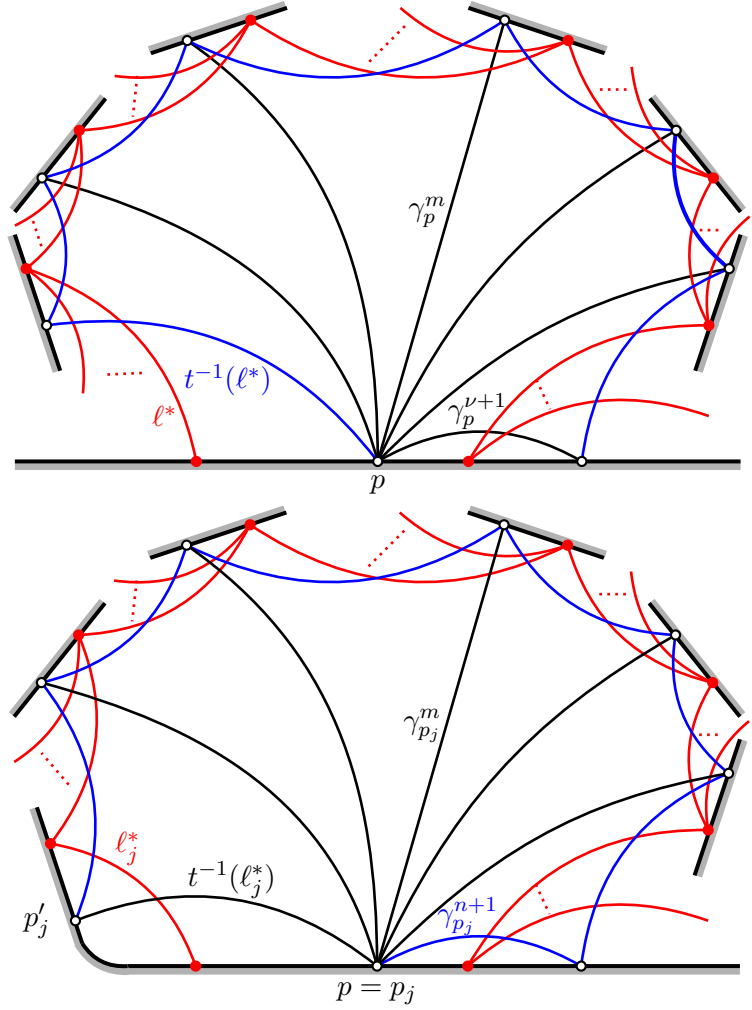

\end{proof}
\begin{proof}(Proof of Theorem \ref{thm:ncomp})
One direction has been proved in Lemma \ref{lem:comp1}.
Now we prove another direction, that is, if $A$ satisfies the conditions in the statement, then it is $n$-complete, that is, the conditions $(A_n), (B_n),$ and $(C_n)$ in Definition \ref{complete} hold. Note that $(A_n)$ has been proved in the proof of Lemma \ref{lem:3.12}, that is, the additive generator $T$ of $\mathfrak{P}(\mathfrak{M})$ is a tilting module in $\ma$.
As in the above, we assume that there are $s$ walks $\sigma_j=\aaa^j_1\cdots\aaa^j_{n}, 1\leqslant j \leqslant s$, in $A$ with full relations, and we use the notation as above.  

Now we prove $(B_n)$, that is, $\mathfrak{M}$ is a cluster tilting subcategory in $T^\perp$.
At first, note that the weights of the oriented intersections of arcs in the associated arc system $\calp_\mathfrak{M}$ equal zero, except for $t^{-1}(\ell^*_j)=\zg^1_{p_j}$ and $\zg^{n+1}_{p_j}, 1\leqslant j \leqslant s$, for which the associated weights are $n$.
Therefore, $\mathfrak{M}$ is $n$-rigid in $\ma$. 
Furthermore, $\mathfrak{M}$ is a $n$-cluster tilting subcategory in $T^\perp$ by directly checking that for any arc in $\calp_{T^\perp}\setminus\calp_\mathfrak{M}$, there exists an arc in $\calp_\mathfrak{M}$ such that there is an oriented intersection between them with weight from one to $n-1$. 

For condition $(C_n)$, note that we have $\mathfrak{M}_P=\oplus_{j=1}^sI_j$, where $I_j$ is the injective module associated with an injective arc $t^{-1}(\ell^*_j)$. Let $P$ be an indecomposable projective module associated with a projective arc $t(\ell^*)$. Then there is no interior intersection between $t^{-1}(\ell^*_j)$ and $t(\ell^*)$, since $t^{-1}(\ell^*_j)$ is isotopic to a boundary segment without $\bpoint$-point. Except for $p_j$, denote by $p'_j$ the other endpoint of $t^{-1}(\ell^*_j)$. Then the weight of any oriented intersection (if it exists) between $t^{-1}(\ell^*_j)$ and $t(\ell^*)$ arising from $p'_j$ equals zero, since $p'_j$ belongs to a triangle of $\zD^*$ consisting of an $\rpoint$-arc together with two boundary segments, cf. the below picture in Figure \ref{fig:ncom2}. On the other hand, suppose that there is an oriented intersection $\mathfrak{p}$ from $t^{-1}(\ell^*_j)$ to $t(\ell^*)$ at $p_j$. Since $t(\ell^*)$ is a twist of the $\rpoint$-arc $\ell^*$, the weight $\omega_{p_j}(t(\ell^*))$ of it at $p_j$ equals $n+1$. Then \[\omega(\mathfrak{p})=\omega_{p_j}(t(\ell^*))-\omega_{p_j}(t^{-1}(\ell_j^*))=n.\] 
Therefore $\Ext^i(I_j,P)=0$ for any $1\leqslant i\leqslant n-1$, and thus $\Ext^i(\mathfrak{M}_P,A)=0$ for any $1\leqslant i\leqslant n-1$, that is, the condition $(C_n)$ in Definition \ref{complete} is met. 

We have completed the proof.
\end{proof}

Let $A=kQ/I$ be a $n$-complete gentle algebra. We construct $\overline{Q}$ and $\overline{I}$ in the following way: $\overline{Q}$ is obtained from $Q$ by adding an additional vertex $v$ and an arrow $\bbb$ from $t(\aaa_{n})$ to $v$ for any walk $\sigma=\aaa_1\cdots\aaa_{n}$ with full relations; $\overline{I}$ is obtained from $I$ by adding the path $\aaa_{n}\bbb$ for each $\bbb$ constructed above.
	
\begin{proposition}
	Let $A=kQ/I$ be a $n$-complete gentle algebra. Then its cone $\End_A(M)\cong k\overline{Q}/\overline{I}$.
\end{proposition}
\begin{proof}
By the proof of the above theorem, $\mathfrak{M}$ corresponds to an admissible $n$-partial triangulation $\zD_I\cup \{\zg_{p_j}^n\}$. The gentle algebra associated with it is isomorphic to $k\overline{Q}/\overline{I}$, where the extra $s$ vertices $v's$ arise from $\zg_{p_j}^n, 1\leqslant j \leqslant s$. Thus, we have $\End_A(M)\cong k\overline{Q}/\overline{I}$, where $M$ is an additive generator of $\mathfrak{M}$.
\end{proof}

It can be seen from the above proposition and Theorem \ref{thm:ncomp} that the cone of a $n$-complete gentle algebra is $(n+1)$-complete, which is a special case of a main result in \cite{I11}

\begin{example}
Figure \ref{fig:ex2} provides an example that illustrates the arc systems $\calp_\mathfrak{M}$, $\zD_T$ and $\calp_{T^\perp}$ respectively corresponding to subcategories $\mathfrak{M}$, $\mathfrak{P}(\mathfrak{M})=\add T$ and $T^\perp$. The global dimension of the associated gentle algebra is five.
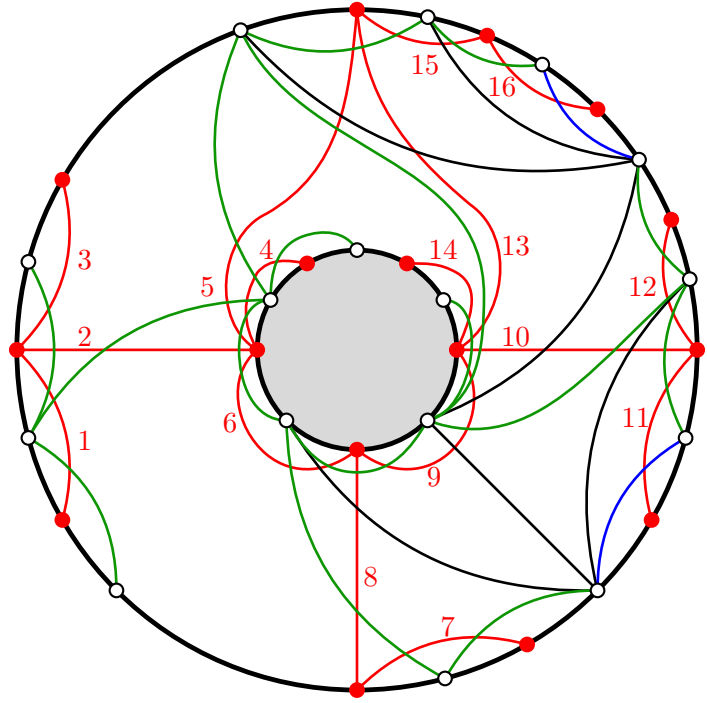
\begin{figure}
\begin{center}
\begin{tikzpicture}[>=stealth,scale=.6]
\draw[line width=1.8pt,fill=white] (0,0) circle (7.5cm);
				\draw[line width=1.8pt,fill=gray!30] (0,0) circle (2.2cm);
\path 
				(90:7.5) coordinate (r1)
				(67.5:7.5) coordinate (r2)
				(45:7.5) coordinate (r3)
				(22.5:7.5) coordinate (r4)
				(0:7.5) coordinate (r5)
				(-30:7.5) coordinate (r6)
				(-60:7.5) coordinate (r7)
				(-90:7.5) coordinate (r8)
				(-150:7.5) coordinate (r9)
				(180:7.5) coordinate (r10)
				(150:7.5) coordinate (r11)
				(60:2.2) coordinate (r12)
				(0:2.2) coordinate (r13)
				(-90:2.2) coordinate (r14)
				(180:2.2) coordinate (r15)
				(120:2.2) coordinate (r16)
				
				(78:7.5) coordinate (b1)
				(57:7.5) coordinate (b2)
				(34:7.5) coordinate (b3)
				(12:7.5) coordinate (b4)
				(-15:7.5) coordinate (b5)
				(-45:7.5) coordinate (b6)
				(-75:7.5) coordinate (b7)
				(-135:7.5) coordinate (b8)
				(-165:7.5) coordinate (b9)
				(165:7.5) coordinate (b10)
				(110:7.5) coordinate (b11)
				(90:2.2) coordinate (b12)
				(30:2.2) coordinate (b13)
				(-45:2.2) coordinate (b14)
				(-135:2.2) coordinate (b15)
				(150:2.2) coordinate (b16)	;

				\draw[bend right,line width=1pt,red] (r1) to (r2) to (r3);
                \draw[cyan,line width=1pt,red] (r1) to[out=-95,in=30](-2,3) to[out=-155,in=150](r15);
				\draw[cyan,line width=1pt,red] (r1) to[out=-85,in=140](2.5,3.2) to[out=-35,in=30](r13);
			    \draw[bend right,line width=1pt,red] (r4) to (r5) to (r6);
			    \draw[line width=1pt,red] (r5) to (r13);
			    \draw[line width=1pt,red] (r8) to (r14);
			    \draw[line width=1pt,red] (r10) to (r15);
				\draw[bend right,line width=1pt,red] (r7) to (r8) ;
				\draw[bend right,line width=1pt,red] (r9) to (r10) to (r11);
			    \draw[cyan,line width=1pt,red](r12) to[out=0,in=150]  (2.2,1.7) to[out=-30,in=65] (r13) to[out=-50,in=40] (2,-2.3) to[out=-140,in=-40] (r14) to[out=-140,in=-40] (-2,-2.3) to[out=140,in=-140] (r15)   to[out=120,in=-120]  (-2.2,1.7) to[out=50,in=170] (r16);

			   \draw[bend left,line width=1pt,black] (b3) to (b1) ;
			   \draw[bend left,line width=1pt,black] (b3) to (b11) ;
			   \draw[bend left,line width=1pt,black] (b3) to (b14) ;
				\draw[bend left,line width=1pt,black] (b6) to (b4) ;
				\draw[bend left,line width=1pt,black] (b6) to (b15) ;
				\draw[line width=1pt,black] (b6) to (b14) ;

				\draw[bend right,line width=1pt, color=dark-green] (b3) to (b4);
				\draw[bend right,line width=1pt, color=dark-green] (b6) to (b7);

				\draw[bend right,line width=1pt, color=dark-green] (b11) to (b1) to (b2) ;
				\draw[bend right,line width=1pt, color=blue] (b2) to (b3) ;
				\draw[bend right,line width=1pt, color=blue] (b5) to (b6)  ;
				\draw[bend right,line width=1pt, color=dark-green] (b4) to (b5)  ;
				\draw[bend right,line width=1pt, color=dark-green] (b8) to (b9) to (b10)  ;
				\draw[cyan,line width=1pt,dark-green] (b13) to[out=0,in=10](b14)
				to[out=-120,in=0](0,-2.7) to[out=180,in=-60](b15) to[out=180,in=-90](-2.6,0) to[out=90,in=180](b16)  to[out=90,in=-160](-1.2,2.5)  to[out=20,in=120](b12);
				\draw[bend right,line width=1pt, color=dark-green] (b11) to (b16)to (b9) ;
				\draw[bend right,line width=1pt, color=dark-green] (b15) to (b7) ;
				\draw[cyan,line width=1pt,dark-green] (b14) to[out=-20,in=-140](b4) ;
				\draw[cyan,line width=1pt,dark-green] (b14) to[out=20,in=-90](2.8,1.2) to[out=90,in=-65](b11) ;
			
				\draw[thick, fill=white] 
				(b1) circle (0.15cm)
				(b2) circle (0.15cm)
				(b3) circle (0.15cm)
				(b4) circle (0.15cm)
				(b5) circle (0.15cm)
				(b6) circle (0.15cm)
				(b7) circle (0.15cm)
				(b8) circle (0.15cm)
				(b9) circle (0.15cm)
				(b10) circle (0.15cm)
				(b11) circle (0.15cm)
				(b12) circle (0.15cm)
				(b13) circle (0.15cm)
				(b14) circle (0.15cm)
				(b15) circle (0.15cm)
				(b16) circle (0.15cm)
				;	
				
				\draw[thick,red, fill=red] 
				(r1) circle (0.15cm)
				(r2) circle (0.15cm)
				(r3) circle (0.15cm)
				(r4) circle (0.15cm)
				(r5) circle (0.15cm)
				(r6) circle (0.15cm)
				(r7) circle (0.15cm)
				(r8) circle (0.15cm)
				(r9) circle (0.15cm)
				(r10) circle (0.15cm)
				(r11) circle (0.15cm)
				(r12) circle (0.15cm)
				(r13) circle (0.15cm)
				(r14) circle (0.15cm)
				(r15) circle (0.15cm)
				(r16) circle (0.15cm)
				;

				\draw(-6,-2) node[red] {{$1$}};
				\draw(-6,0.3) node[red] {{$2$}};
				\draw(-6,2) node[red] {{$3$}};
				\draw(-2,2.2) node[red] {{$4$}};
				\draw(-3.3,1.4) node[red] {{$5$}};
				\draw(-2.8,-1.6) node[red] {{$6$}};
				
				\draw(2,-6.1) node[red] {{$7$}};
				\draw(0.3,-5) node[red] {{$8$}};
				\draw(1.7,-2.8) node[red] {{$9$}};
				\draw(3.5,0.3) node[red] {{$10$}};
				\draw(6.15,-1.5) node[red] {{$11$}};
				\draw(6.3,1.4) node[red] {{$12$}};
				\draw(3.5,2.3) node[red] {{$13$}};
				\draw(1.9,2.2) node[red] {{$14$}};
				\draw(1.5,6.3) node[red] {{$15$}};
				\draw(3.2,5.8) node[red] {{$16$}};

\end{tikzpicture}
\caption{The arcs in the injective dissection $\zD_I$ is colored green. After replacing the two injective arcs associated with vertices $7$ and $12$ by the blue arcs, we obtain a tilting dissection $\zD_T$. $\calp_{\mathfrak{M}}$ is the union of the set of green arcs and the set of blue arcs. Furthermore, after adding the black arcs, we get $\calp_{T^\perp}$.}\label{fig:ex2}
\end{center}
\end{figure}
\end{example}

{\bf The author declares that he has no conflict of interest.

No datasets were generated or analyzed during the current study.}


\begin{thebibliography}{99}

\newcommand{\au}[1]{\textrm{#1},}
\newcommand{\ti}[1]{\textrm{#1},}
\newcommand{\jo}[1]{\textit{#1}}
\newcommand{\vo}[1]{\textbf{#1}}
\newcommand{\yr}[1]{(#1)}
\newcommand{\pp}[2]{#1--#2.}
\newcommand{\arxiv}[1]{\href{http://arxiv.org/abs/#1}{arXiv:#1}}



\bibitem[ABCJP10]{ABCJP10}
I. Assem, T. Br\"ustle, G. Charbonneau-Jodoin and P-G. Plamondon,
\newblock Gentle algebras arising from surface triangulations.
\newblock {\em Algebra Number Theory} 4 (2010), no. 2, 201-229.

\bibitem[AH81]{AH81}
I. Assem and D. Happel,
\newblock Generalized tilted algebras of type {$A_{n}$}.
\newblock {\em Comm. Algebra} 9 (1981), no. 20, 2101-2125.

\bibitem[AIR15]{AIR15}
C. Amiot, O. Iyama and I. Reiten,
\newblock Stable categories of Cohen-Macaulay modules and cluster categories.
\newblock {\em Amer. J. Math.} 137 (2015) 813-857. 

\bibitem[APS23]{APS23}
C. Amiot, P-G. Plamondon and S. Schroll,
\newblock A complete derived invariant for gentle algebras via winding numbers and Arf invariants. 
\newblock {\em Selecta Math. (N.S.)} 29 (2023), no. 2, Paper No. 30, 36 pp.

\bibitem[ARS95]{ARS} M. Auslander, I. Reiten, S. O. Smalo, Representation theory of Artin algebras,
Cambridge Studies in Advanced Mathematics, \textbf{36}. Cambridge
University Press, Cambridge, 1995.

\bibitem[AS87]{AS87}
I. Assem and A. Skowro\'{n}ski,
\newblock Iterated tilted algebras of type affine $A_n$.
\newblock {\em Math. Z.} 195 (1987), no. 2, 269-290.

\bibitem[ASS06]{ASS} I. Assem, D. Simson and A. Skowro\'nski,
Elements of the representation theory of associative algebras. Vol. 1. Techniques of representation theory, London
Mathematical Society Student Texts, 65. Cambridge University Press,
Cambridge, 2006.

\bibitem[BC21]{BC21}
K. Baur and R. Coelho Sim\~oes,
\newblock A geometric model for the module category of a gentle algebra.
\newblock {\em Int. Math. Res. Not. IMRN} (2021), no. 15, 11357-11392.


\bibitem[BC24]{BC24}
K. Baur and R. Coelho Sim\~oes,
\newblock A geometric model for the module category of a string algebra.
\newblock (2024) Preprint \arxiv{2403.07810}.


\bibitem[BCGS24]{BCGS24}
E. Barnard, R. Coelho Simoes, E. Gunawan and R. Schiffler,
\newblock Maximal almost rigid modules over gentle algebras.
\newblock (2024) Preprint \arxiv{2408.16904}.

\bibitem[BMGS23]{BMGS23}
E. Barnard, E. Meehan, E. Gunawan and R. Schiffler,
\newblock Cambrian combinatorics on quiver representations (type A).
\newblock {\em Adv. in Appl. Math.} 143 (2023), p. 102428.

\bibitem[BR87]{BR87}
M. C. R. Butler and C. M. Ringel,
\newblock Auslander-Reiten sequences with few middle terms and applications to string algebras.
\newblock {\em Comm. Algebra} 15 (1987), no. 1-2, 145-179.


\bibitem[BS21]{BS21}
K. Baur and S. Schroll,
\newblock Higher extensions for gentle algebras.
\newblock {\em Bull. Sci. Math.} 170 (2021), Paper No. 103010, 21 pp.

\bibitem[C89]{C89}
W. W. Crawley-Boevey,
\newblock Maps between representations of zero-relation algebras.
\newblock {\em J. Algebra} 126, no. 2 (1989): 259–63.



\bibitem[C24]{C24}
W. Chang,
\newblock Tilting-completion for gentle algebras. \newblock (2024) Preprint
\href{https://arxiv.org/abs/2412.13971}{\texttt{arXiv:2412.13971}}.


\bibitem[C25]{C25}
W. Chang,
\newblock Geometric models for the algebraic hearts in the derived category of a gentle algebra. \emph{Isr. J. Math.} (2025). https://doi.org/10.1007/s11856-025-2801-7.

\bibitem[HI11]{HI11}
M. Herschend and O. Iyama, 
\newblock Selfinjective quivers with potential and 2-representation-finite algebras.
\newblock {\em Compos. Math.}, 147 (2011) 1885-1920.

\bibitem[HIMO23]{HIMO23}
M. Herschend, O. Iyama, H. Minamoto and S. Oppermann, \newblock 
Representation theory of Geigle–Lenzing complete intersections. \newblock {\em Mem. Amer. Math. Soc.} 285 (2023), no. 1412, vii+141 pp.

\bibitem[HIO14]{HIO14}
M. Herschend, O. Iyama, and S. Oppermann,
\newblock n-representation infinite algebras. \newblock {\em Adv. Math.} 252 (2014) 292–342.

\bibitem[HJS22]{HJS22}
J. Haugland, K. M. Jacobsen and S. Schroll,
\newblock The role of gentle algebras in higher homological algebra
\newblock {\em Forum Math.} 34 (2022), no. 5, 1255–1275.

\bibitem[HKK17]{HKK17}
F. Haiden, L. Katzarkov and M. Kontsevich,
\newblock Flat surfaces and stability structures.
\newblock {\em Publ. Math. Inst. Hautes \'Etudes Sci.}
126 (2017), 247-318.

\bibitem[HS05]{HS05}
B. Huisgen-Zimmermann and S. O. Smal\o,
\newblock The homology of string algebras. I.
\newblock {\em J. Reine Angew. Math.}
580 (2005), 1-37.

\bibitem[HZZ23]{HZZ23}
P. He, Y Zhou, and B. Zhu,
\newblock A geometric model for the module category of a skew-gentle algebra.
\newblock {\em Math. Z.}
304 (2023), no. 1, Paper No. 18, 41 pp.

\bibitem[I07a]{I07a}  O. Iyama, Auslander correspondence,
{\em Adv. Math.} \textbf{210} (2007), no. 1, 51–82.

\bibitem[I07b]{I07b}  O. Iyama, Higher-dimensional Auslander-Reiten theory on maximal orthogonal subcategories,
{\em Adv. Math.} \textbf{210} (2007), no. 1, 22--50.

\bibitem[I08]{I08}
O. Iyama,
\newblock 
Auslander-Reiten theory revisited,
Trends in representation theory of algebras and related topics, \newblock {\em European Mathematical Society Series of Congress Report} (European Mathematical Society, Zürich, 2008) 349–397.  

\bibitem[I11]{I11}
O. Iyama,
\newblock Cluster tilting for higher Auslander algebras.
\newblock {\em Adv. Math.} 226 (2011), no. 1, 1-61.

\bibitem[IO11]{IO11}
O. Iyama and S. Oppermann,
\newblock 
n-representation-finite algebras and n-APR tilting.
\newblock {\em Trans. Amer. Math. Soc.} 363 (2011) 6575–6614. 

\bibitem[IW14]{IW14}
O. Iyama and M. Wemyss,
\newblock Maximal modifications and Auslander-Reiten duality for non-isolated singularities. \newblock {\em Invent. Math.} 197 (2014) 521-586. 

\bibitem[JK19]{JK19}
G. Jasso and S. Kvamme,
\newblock 
An introduction to higher Auslander-Reiten theory. (English summary)
\newblock {\em Bull. Lond. Math. Soc.} 51 (2019), no. 1, 1–24.

\bibitem[LP20]{LP20}
Y. Lekili and A. Polishchuk,
\newblock Derived equivalences of gentle algebras via Fukaya categories.
\newblock {\em Math. Ann.} 376 (2020), no. 1-2, 187-225.

\bibitem[M14]{M14}
Y. Mizuno,
\newblock 
A Gabriel-type theorem for cluster tilting.
\newblock {\em Proc. Lond. Math. Soc.} (3) 108 (2014) 836–868.

\bibitem[OPS18]{OPS18}
\au{S. Opper, P-G Plamondon \and S. Schroll}
\ti{A geometric model for the derived category of gentle algebras}
\yr{2018} Preprint
\arxiv{1801.09659}.

\bibitem[OT12]{OT12}
S. Oppermann and H. Thomas,
\newblock Higher-dimensional cluster combinatorics and representation theory.
\newblock {\em J. Eur. Math. Soc. (JEMS)} 14 (2012) 1679–1737.

\bibitem[PPP19]{PPP19}
Y. Palu, V. Pilaud and P-G Plamondon,
\newblock Non-kissing and non-crossing complexes for locally gentle algebras.
\newblock  {\em J. Comb. Algebra} 3 (2019), no. 4, 401-438.

\bibitem[PPP21]{PPP21}
Y. Palu, V. Pilaud and P-G Plamondon,
\newblock Non-kissing complexes and tau-tilting for gentle algebras.
\newblock  {\em Mem. Amer. Math. Soc.} 274 (2021), no. 1343, vii+110 pp.

\bibitem[S99]{S99}
J. Schröer,
\newblock Modules without self-extensions over gentle algebras.
\newblock {\em J. Algebra}, 216 (1999), 1, 178-189.

\end{thebibliography}
\end{document}